\documentclass{amsart}

\usepackage{amssymb}
\usepackage{bbm}
\usepackage{enumerate}
\usepackage[all]{xy}

\theoremstyle{plain}
\newtheorem{theorem}{Theorem}[section]
\newtheorem{proposition}[theorem]{Proposition}
\newtheorem{corollary}[theorem]{Corollary}
\newtheorem{lemma}[theorem]{Lemma}

\theoremstyle{definition}
\newtheorem{definition}[theorem]{Definition}
\newtheorem{passage}[theorem]{}
\newtheorem{example}[theorem]{Example}

\theoremstyle{remark}
\newtheorem{remark}[theorem]{Remark}

\numberwithin{equation}{section}

\begin{document}

\title[Lift derived equivalences of abelian surface to generalized Kummer]{Lifting derived equivalences of abelian surfaces to generalized Kummer varieties}

\author{Yuxuan Yang}

\address{Department of Mathematics and Statistics, University of Massachusetts Amherst, Amherst, Massachusetts 01002}

\curraddr{Department of Mathematics and Statistics,
University of Massachusetts Amherst, Amherst, Massachusetts 01002}
\email{yuxuanyang@umass.edu or yuxuanyangalggeom@gmail.com}


\date{March 4th, 2026.}


\begin{abstract}
In this article, we study the $G$-autoequivalences of the derived category $\mathbf{D}^b_G(A)$ of $G$-equivariant objects for an abelian variety $A$ with $G$ being a finite subgroup of $\mathrm{Pic}^0(A)$. We provide a result analogue to Orlov's short exact sequence for derived equivalences of abelian varieties. It can be generalized to the derived equivalences of abelian varieties for a same $G$ in general. Furthermore, we find derived equivalences of generalized Kummer varieties by lifting derived equivalences of abelian surfaces using the $G$-equivariant version of Orlov's short exact sequence and some ``splitting" propositions. 
\end{abstract}

\maketitle

\tableofcontents

\section*{Introduction}
The prototype of the main result of this paper is the lift of a derived equivalence of K3 surfaces to a derived equivalence of the corresponding Hilbert schemes. Ploog states in \cite{Ploog:05} and \cite{Ploog:07} that a Fourier-Mukai equivalence, $\mathrm{FM}_{P}:\mathbf{D}^b(S_1)\to\mathbf{D}^b(S_2)$, of K3 surfaces yields $$\mathrm{FM}^{\mathfrak{S}_n}_{P^{\boxtimes n}}:\mathbf{D}^b(S_1^n)\xrightarrow[]{\simeq}\mathbf{D}^b(S_2^n),$$
where the kernel $P^{\boxtimes n}:=P\boxtimes\dots\boxtimes P=\overset{n}{\underset{i=1}{\oplus}}\pi_i^*P\in\mathbf{D}^b((S_1\times S_2)^n)$ is canonically $(\mathfrak{S}_n)_\Delta$-linearized, by \cite{Ploog:05}. It gives a derived equivalence of the corresponding Hilbert schemes via Haiman's result and the derived McKay correspondence.

The situation for generalized Kummer varieties is not easy. We use the following notation to set up.

Let $A$ be an abelian variety, and let $G\leq \mathrm{Pic}^0(A)\cong \widehat{A}$ be a finite subgroup. Conjugating the translation action of $G$ on $\mathbf{D}^b(\widehat{A})$ via Mukai's equivalence $$\Phi_{\mathcal{P}}:\mathbf{D}^b(A)\xrightarrow[]{\simeq}\mathbf{D}^b(\widehat{A}),$$
where $\mathcal{P}$ is the Poincar\'{e} line bundle over $A$, we get the action $$\rho_g:=\Phi_{\mathcal{P}}^{-1}\circ(t_g)_*\circ\Phi_{\mathcal{P}}=-\otimes\mathcal{L}_g:\mathbf{D}^b(A)\xrightarrow[]{\simeq}\mathbf{D}^b(A)$$
associated to $g\in G\leq \widehat{A}$ corresponding to $\mathcal{L}_g\in\mathrm{Pic}^0(A)$. As in \cite[subsections 2.1 and 3.1]{Beckmann Oberdieck:23} (see also subsection \ref{BO 2.1 3.1}), we get the equivariant category $\mathbf{D}^b_G(A)$ with objects being $G$-equivariant objects. Then the derived category $\mathbf{D}^b_G(A)$ is equivalent to $\mathbf{D}^b(B)$, where $B$ is the abelian variety with homomorphism $q:B\to A$ such that $G=\mathrm{ker}(\widehat{q})$, where the induced map $\widehat{q}$ is $\mathrm{Pic}^0(A)\cong\widehat{A}\to\widehat{B}\cong\mathrm{Pic}^0(B)$.

The set $G\textrm{-}\mathrm{Aut}(\mathbf{D}^b_G(A))$ of the derived autoequivalences of $\mathbf{D}^b_G(A)$, using the definition in \cite[subsection 2.1]{Beckmann Oberdieck:23}, consists of $G$-functors of the form $(f,\sigma)$, where $f$ is in $\mathrm{Aut}(\mathbf{D}^b(A))$ and $\sigma$ is a set of $G$-equivariance natural transformations for $f$. We get the forgetful natural map $$F_q:G\textrm{-}\mathrm{Aut}(\mathbf{D}^b_G(A))\to\mathrm{Aut}(\mathbf{D}^b(A)), (f,\sigma)\mapsto f.$$

In addition, we obtain a composition of natural maps $$\lambda_q:G\textrm{-}\mathrm{Aut}(\mathbf{D}^b_G(A))\to\mathrm{Aut}(\mathbf{D}^b_G(A))\simeq\mathrm{Aut}(\mathbf{D}^b(B)).$$

Let notation be as above over an algebraically closed field of characteristic $0$. Set $V_A:=\mathrm{H}^1(A\times \widehat{A}, \mathbb{Z})=\mathrm{H}^1(A, \mathbb{Z})\times\mathrm{H}^1(\widehat{A}, \mathbb{Z})$ with the pairing $(-,-)_{V_A}$ defined to be $$((\alpha_1,\beta_1),(\alpha_2,\beta_2))_{V_A}:=\beta_1(\alpha_2)+\beta_2(\alpha_1),$$ where $(\alpha_1,\beta_1), (\alpha_2,\beta_2)\in V_A\simeq\mathrm{H}^1(A, \mathbb{Z})\times\mathrm{H}^1({A}, \mathbb{Z})^*$. Let $\mathrm{SO}(V_A)$ be the group of special orthogonal maps from $V_A$ to $V_A$ with respect to the pairing $(-,-)_{V_A}$. Its subgroup $\mathrm{SO}^+(V_A):=\mathrm{SO}(V_A)\cap\mathrm{ker}(\mathrm{sn}_{\mathbb{R}})$ consists of maps with trivial spinor norms, where $\mathrm{ker}(\mathrm{sn}_{\mathbb{R}})\subset\mathrm{O}(V_A\otimes\mathbb{R})$ as below.

As in the notation before \cite[Theorem 1.1]{Gritsenko Hulek Sankaran:09}, in general, the spinor norm over a field $K\not=\mathbb{F}_2$ for an integral even lattice $(L,(-,-))$ is defined to be a group homomorphism $\mathrm{sn}_K:\mathrm{O}(L\otimes K)\to K^\times/(K^\times)^2$ with $$\mathrm{sn}_K(g)=(-\frac{(v_1,v_1)}{2})\cdot...\cdot(-\frac{(v_m,v_m)}{2})(K^\times)^2,$$ where for $K\not=\mathbb{F}_2$, any $g\in\mathrm{O}(L\otimes K)$ can be represented as a product of reflections $g=s_{v_1}s_{v_2}...s_{v_n}$, where $v_i\in L\otimes K$.

The subgroup $\mathrm{SO}_{\mathrm{Hdg}}(V_A)$ consists of all the maps in $\mathrm{SO}(V_A)$ that preserve the Hodge structure of $V_A$.

Denote Orlov's representation by $$\rho_A:\mathrm{Aut}(\mathbf{D}^b(A))\to \mathrm{SO}_{\mathrm{Hdg}}(V_A),$$
where $\mathrm{image}(\rho_A)=\mathrm{SO}^+_{\mathrm{Hdg}}(V_A)$ by \cite[Proposition 5.44 and Corollary 9.57]{Huybrechts:06} and \cite[Lemma 4.1]{Markman:23}. The group $\mathrm{SO}^+_{\mathrm{Hdg}}(V_A)$ embeds naturally as a group of automorphisms of $A\times\widehat{A}$. Then we have an exact sequence $$0\to\mathrm{Alb}(A)\times\widehat{A}\times\mathbb{Z}\to\mathrm{Aut}(\mathbf{D}^b(A))\xrightarrow[]{\rho_A}\mathrm{SO}^+_{\mathrm{Hdg}}(V_A)\to0$$ via \cite[Corollaries 9.57 and 9.61]{Huybrechts:06} and Remark \ref{two Orlov's representation remark}. 

We have the isometry $(q^*,(\widehat{q}^*)^{-1}): V_{A,\mathbb{Q}}:=V_A\otimes_{\mathbb{Z}}\mathbb{Q}\to V_{B,\mathbb{Q}}:=V_B\otimes_{\mathbb{Z}}\mathbb{Q}$. Then we define $G\textrm{-}\mathrm{SO}^+_{\mathrm{Hdg}}(V_B)\leq\mathrm{SO}^+_{\mathrm{Hdg}}(V_B)$ as the subgroup consisting of all Hodge isometries $\gamma\in\mathrm{SO}^+_{\mathrm{Hdg}}(V_B)$ such that the extension of $\gamma$ to $V_{B,\mathbb{Q}}=V_B\otimes_{\mathbb{Z}}\mathbb{Q}$ leaves the lattice $q^*\mathrm{H}^1(A,\mathbb{Z})\oplus\mathrm{H}^1(\widehat{A},\mathbb{Z})$ invariant as the lattice in $V_{B,\mathbb{Q}}$ via the isometry $(q^*,(\widehat{q}^*)^{-1})$.

We have the following result, which is the $G$-equivariant case of the Orlov's short exact sequence.

\begin{theorem}[Theorem \ref{Th_GrhoA}]
    Let the notation be as above over an algebraically closed field of characteristic $0$. We have a short exact sequence $$0\to\mathrm{Alb}(B)\times\widehat{A}\times\mathbb{Z}\to G\textrm{-}\mathrm{Aut}(\mathbf{D}^b_G(A))\xrightarrow[]{G\textrm{-}\rho_A}G\textrm{-}\mathrm{SO}^+_{\mathrm{Hdg}}(V_B)\to 0.$$ 
    It fits into the commutative diagram with exact rows below, where the first and the third exact rows come from Orlov's Theorem.
    \begin{align*}
    \xymatrix{
    0\ar[r]&\mathrm{Alb}(B)\times\widehat{B}\times\mathbb{Z}\ar[r]&\mathrm{Aut}(\mathbf{D}^b(B))\ar[r]^{\rho_B}&\mathrm{SO}^+_{\mathrm{Hdg}}(V_B)\ar[r]&0\\
    0\ar[r]&\mathrm{Alb}(B)\times\widehat{A}\times\mathbb{Z}\ar[r]\ar[u]^{\mathrm{id}_{\mathrm{Alb}(B)}\times\widehat{q}\times\mathrm{id}_{\mathbb{Z}}}\ar[d]_{q_*\times\mathrm{id}_{\widehat{A}}\times\mathrm{id}_{\mathbb{Z}}}&G\textrm{-}\mathrm{Aut}(\mathbf{D}^b(A))\ar[r]^{G\textrm{-}\rho_A}\ar[u]^{\lambda_q}\ar[d]_{F_q}&G\textrm{-}\mathrm{SO}^+_{\mathrm{Hdg}}(V_B)\ar[r]\ar[u]^{\cup}\ar[d]_{\mathrm{res}}&0\\
    0\ar[r]&\mathrm{Alb}(A)\times\widehat{A}\times\mathbb{Z}\ar[r]&\mathrm{Aut}(\mathbf{D}^b(A))\ar[r]^{\rho_A}&\mathrm{SO}^+_{\mathrm{Hdg}}(V_A)\ar[r]&0
    }
    \end{align*}
\end{theorem}

Moreover, by subsection \ref{finite index Th_GrhoA}, the image $F_q(G\textrm{-}\mathrm{Aut}(\mathbf{D}^b(A)))$ is a subgroup of finite index $[G\textrm{-}\mathrm{SO}^+_{\mathrm{Hdg}}(V_B):\mathrm{SO}^+_{\mathrm{Hdg}}(V_A)]$ of $\mathrm{Aut}(\mathbf{D}^b(A))$. It fits in a short exact sequence $0\to\mathrm{ker}(q)\to G\textrm{-}\mathrm{Aut}(\mathbf{D}^b(A))\xrightarrow{F_q}F_q(G\textrm{-}\mathrm{Aut}(\mathbf{D}^b(A)))\to0$.

Furthermore, given two $n$-dimensional abelian varieties $A$ and $A'$ and embeddings of a finite group $G$ in both $\mathrm{Pic}^0(A)$ and $\mathrm{Pic}^0(A')$, the above can be generalized to $G$-functors between $\mathbf{D}^b_G(A)$ and $\mathbf{D}^b_G(A')$, which are derived equivalences. (See Theorem \ref{Th_GrhoAA'}.)
~\\

After the preparation in Section 1, we may use the tool to find some derived equivalences of generalized Kummer varieties in Section 2. 

Let $\Sigma:A^n\to A,n\geq2$, be the summation morphism for an abelian surface $A$ and let $N_A$ be its kernel. We consider the morphism 
\begin{align*}
    q:N_A\times A&\to A^n\\
    ((a_1,\dots,a_n),a)&\mapsto(a_1+a,\dots,a_n+a).
\end{align*}
Let $\mathfrak{S}_n$ act on $N_A\times A$ by the natural action on $N_A$ and the trivial action on $A$. Then $$\mathrm{ker}(q)=\{((a,a,\dots,a),-a)|a\in A[n]\}.$$

For the diagonal embedding $\widehat{\Sigma}:\widehat{A}\to\widehat{A}^n$ induced from $\Sigma:A^n\to A$, the summation, we have $\widehat{\Sigma}(\widehat{A}[n])=\mathrm{ker}(\widehat{q})$. Set $G:=\mathrm{ker}(\widehat{q})\leq\widehat{A}^n, \mathcal{G}:=\widehat{A}[n]$. Then $\widehat{\Sigma}$ restricts to an isomorphism $\mathcal{G}\xrightarrow[]{\cong}G$. We get the maps 
$$\mathrm{Aut}(\mathbf{D}^b(N_A\times A))\xleftarrow[]{\lambda_{q}}G\textrm{-}\mathrm{Aut}(\mathbf{D}^b_G(A^n))\xrightarrow[]{F_{q}}\mathrm{Aut}(\mathbf{D}^b(A^n)).$$

The Bridgeland-King-Reid theorem \cite{BKR:01} yields the equivalences $$\mathbf{D}^b_{\mathfrak{S}_n}(A^n)\simeq\mathbf{D}^b(A^{[n]}), \mathbf{D}^b_{\mathfrak{S}_n}(N_A\times A)\simeq\mathbf{D}^b(\mathrm{Kum}^{n-1}(A)\times A).$$
The group $\mathfrak{S}_n\times G$ acts on $\mathbf{D}^b(A^n)$, since the actions of $G$ and $\mathfrak{S}_n$ commute. We get an equivalence $\mathbf{D}^b_{\mathfrak{S}_n}(N_A\times A)\simeq\mathbf{D}^b_{\mathfrak{S}_n\times G}(A^n)\simeq\mathbf{D}_{\mathfrak{S}_n}(\mathbf{D}_G(A^n))$. By \cite[Proposition 3.3]{Beckmann Oberdieck:23} and BKR, we have maps $$\mathrm{Aut}(\mathbf{D}^b(\mathrm{Kum}^{n-1}(A)\times A)\xleftarrow[]{\widetilde{\lambda}_{q}}G\textrm{-}\mathrm{Aut}(\mathbf{D}_G(\mathbf{D}^b_{\mathfrak{S}_n}(A^n)))\xrightarrow[]{\widetilde{F}_{q}}\mathrm{Aut}(\mathbf{D}^b(A^{[n]})).$$
Let $\delta_{A}:\mathrm{Aut}(\mathbf{D}^b(A))\to\mathrm{Aut}(\mathbf{D}^b_{\mathfrak{S}_n}(A^n))\simeq\mathrm{Aut}(\mathbf{D}^b(A^{[n]}))$ be the natural functor and $n_A:A\to A$ be the map, multiplication by $n$. We have a map $\widetilde{\delta}_A$ as follows.

\begin{lemma}[Diagram (\ref{delta_EqA}) from Proposition \ref{delta_AA' Eq}]
    For an abelian surface $A$, there exists a map $\widetilde{\delta}_A$ that makes the square below commutative.
    \begin{align*}
        \xymatrix{
        \mathrm{Aut}(\mathbf{D}^b(\mathrm{Kum}^{n-1}(A)\times A))&
        G\textrm{-}\mathrm{Aut}(\mathbf{D}_G(\mathbf{D}^b_{\mathfrak{S}_n}(A^n)))\ar[l]_(0.45){\widetilde{\lambda}_{q}}\ar[r]^(0.58){\widetilde{F}_{q}}&\mathrm{Aut}(\mathbf{D}^b(A^{[n]}))\\
        \mathrm{Aut}(\mathbf{D}^b(A))&\mathcal{G}\textrm{-}\mathrm{Aut}(\mathbf{D}^b_{\mathcal{G}}(A))\ar[u]^{\widetilde{\delta}_{A}}\ar[r]_(0.55){F_{n_A}}\ar[l]^{{\lambda_{n_A}}}&\mathrm{Aut}(\mathbf{D}^b(A))\ar[u]^{{\delta_{A}}}
        }
    \end{align*}
\end{lemma}

Now that we obtain the derived autoequivalences of $\mathbf{D}^b(\mathrm{Kum}^{n-1}(A)\times A)$, we may split to get the derived autoequivalences of $\mathbf{D}^b(\mathrm{Kum}^{n-1}(A))$ by the following.

\begin{theorem}[Corollary \ref{split_NAA}]
    Let notation be as above over an algebraically closed field of characteristic $0$. For arbitrary $(f,\sigma)\in\mathcal{G}\textrm{-}\mathrm{Aut}(\mathbf{D}^b_{\mathcal{G}}(A))$, the splitting 
     \begin{equation*}
         \widetilde{\lambda}_{q}(\widetilde{\delta}_{A}(f,\sigma))=\Phi_{(f,\sigma)}\times\Psi_{(f,\sigma)}
     \end{equation*} holds for a unique combination of $\begin{cases}
         \Phi_{(f,\sigma)}\in\mathrm{Aut}(\mathbf{D}^b(\mathrm{Kum}^{n-1}(A)))\\
         \Psi_{(f,\sigma)}\in\mathrm{Aut}(\mathbf{D}^b(A)).
     \end{cases}$
     
\end{theorem}

So, we may lift any derived autoequivalence of an abelian surface in the image of $F_{n_A}$ to a derived autoequivalence of the corresponding generalized Kummer variety. This leads to a series of autoequivalences to generalized Kummer varieties with an explicit description. Moreover, we have
\begin{corollary}[Corollary \ref{split_NAA index n2}]
    Let notation be as above over an algebraically closed field of characteristic $0$. For an arbitrary $f\in\mathrm{Aut}(\mathbf{D}^b(A))$, up to finite index $n^2$, the splitting 
     \begin{equation*}
         \widetilde{\lambda}_{q}(\widetilde{\delta}_{A}(f,\sigma))=\Phi_{(f,\sigma)}\times\Psi_{(f,\sigma)}
     \end{equation*} holds for a unique combination of $\begin{cases}
         \Phi_{(f,\sigma)}\in\mathrm{Aut}(\mathbf{D}^b(\mathrm{Kum}^{n-1}(A)))\\
         \Psi_{(f,\sigma)}\in\mathrm{Aut}(\mathbf{D}^b(A))
     \end{cases}$,
     where $\sigma$ is a set of $\mathcal{G}$-equivariance natural transformations, such that $(f,\sigma)\in\mathcal{G}\textit{-}\mathrm{Aut}(\mathbf{D}^b_{\mathcal{G}}(A))$.

\end{corollary}

This idea can be extended to the derived equivalences among generalized Kummer varieties and $\mathcal{G}$-functors $\mathbf{D}^b_\mathcal{G}(A)\xrightarrow[]{\simeq}\mathbf{D}^b_\mathcal{G}(A')$, which are derived equivalences. See Theorem \ref{split_NAANA'A'} and Theorem \ref{split_NAANA'A' index n2}.

We may describe the derived (auto)equivalences among generalized Kummer varieties obtained above via Orlov's representation as in Theorem \ref{describe derived equiv NANA'} and Corollary \ref{describe derived equiv NA}. In particular, we may end this paper by some investigation about derived autoequivalences of Kummer K3 surfaces.

In Section \ref{Relating derived equivalences of isogeneous abelian varieties}, we do some prepartation about the investigation for the derived equivalences of generalized Kummer varities. In subsections \ref{BO 2.1 3.1} and \ref{G-functors between equivariant categories of abelian varieties}, we recall and develop notations and tools about group actions of categories, $G$-functors and equivariant categories, in general and of abelian varieties. In the rest of the subsections, we state and prove the $G$-equivariant case of the Orlov's short exact sequence (Theorem \ref{Th_GrhoA} and Theorem \ref{Th_GrhoAA'}). In Section \ref{Generalized Kummer varieties}, we state and prove Theorem \ref{split_NAANA'A'} and Corollary \ref{split_NAA} to get a series of derived equivalences of generalized Kummer varieties in subsections \ref{DbGA to DbKumXA} and \ref{DbGA to DbKum}. We may describe the derived (auto)equivalences among generalized Kummer varieties obtained above via Orlov's representation in subsection \ref{Derived equivalences between generalized Kummer varieties from Theorem}. Some note about derived autoequivalences of Kummer K3 surfaces follows at the end.
\phantom{1}

I am indebted to my advisor Eyal Markman for his thoughts, comments, and support on the first version of this paper. I would like to thank David Ploog, David Zhiyuan Bai and Ziwei Lu for their advice and comments to improve the exposition.

\section{Relating derived equivalences of isogeneous abelian varieties}\label{Relating derived equivalences of isogeneous abelian varieties}
As in the Introduction, we need to develop some notation about $G$-functors before stating the theorem about the $G$-equivariant case of Orlov's short exact sequence.

\subsection{Group actions of categories, $G$-functors and equivariant categories.}\label{BO 2.1 3.1}
In this subsection, we recall the definitions, notations, and some results of group actions of categories, $G$-functors and equivariant categories in \cite[subsections 2.1 and 3.1]{Beckmann Oberdieck:23} used later. Then we may introduce two natural maps $F,\lambda$ from the set of $G$-functors.

First, we focus on group actions of categories.

Let $G$ be a finite group, and let $\mathcal{D}$ be a category.

\begin{definition}\label{def of group action on category}
    An \textit{action} $(\rho,\theta)$ of $G$ on $\mathcal{D}$ consists of 
    \begin{enumerate}
        \item for every $g\in G$, an autoequivalence $\rho_g:\mathcal{D}\to\mathcal{D}$,
        \item for every pair $g,h\in G$, an isomorphism of functors $\theta_{g,h}:\rho_g\circ\rho_h\to\rho_{gh}$,
    \end{enumerate}
    such that for all triples $g,h,k\in G$, we have the commutative diagram
    \begin{align*}
        \xymatrix{
        \rho_g\circ\rho_h\circ\rho_k\ar[rr]^{\rho_g\theta_{h,k}}\ar[d]_{\theta_{g,h}\rho_k}&&\rho_g\circ\rho_{hk}\ar[d]^{\theta_{g,hk}}\\
        \rho_{gh}\circ\rho_k\ar[rr]^{\theta_{gh,k}}&&\rho_{ghk}.
        }
    \end{align*}
\end{definition}

Recall the 2-category of categories $\mathfrak{Cats}$, whose objects are categories, the morphisms are functors between categories, and the 2-morphisms are natural transformations. Similarly, we have the 2-category $G\textrm{-}\mathfrak{Cats}$ of categories with a $G$-action, by the morphisms and 2-morphisms defined below.

\begin{definition}\label{def of G-functor}
    A morphism or $G$-\textit{functor} $$(f,\sigma):(\mathcal{D},\rho,\theta)\to(\mathcal{D}',\rho',\theta')$$
    between categories with $G$-actions is a pair of a functor $f:\mathcal{D}\to\mathcal{D}'$, together with 2-isomorphisms $\sigma_g:f\circ\rho_g\to\rho'_g\circ f$ such that $(f,\sigma)$ intertwines the associativity relations on both sides, i.e. such that the following diagram commutes:
    \begin{align}\label{G-fun sigma commutative diagram}
        \xymatrix{
        f\circ\rho_g\circ\rho_h\ar[rr]^{f\theta_{g,h}}\ar[d]_{\sigma_g\rho_h}&&f\circ\rho_{gh}\ar[dd]^{\sigma_{gh}}\\
        \rho_g'\circ f\circ\rho_h\ar[d]_{\rho'_g\sigma_h}&&\\
        \rho_g'\circ\rho_h'\circ f\ar[rr]^{\theta'_{g,h}f}&&\rho'_{gh}\circ f.
        }
    \end{align}
\end{definition}

\begin{definition}
    A 2-morphism of $G$-functors $(f,\sigma)\to(\widetilde{f},\widetilde{\sigma})$ is a 2-morphism $t:f\to\widetilde{f}$ that interviews the $\sigma_g$, i.e. such that the following diagram commutes:
    \begin{align*}
        \xymatrix{
        f\circ\rho_g\ar[rr]^{\sigma_g}\ar[d]_{t\rho_g}&&\rho'_g\circ f\ar[d]^{\rho'_gt}\\
        \widetilde{f}\circ\rho_g\ar[rr]^{\widetilde{\sigma}_g}&&\rho'_g\circ\widetilde{f}.
        }
    \end{align*}
\end{definition}

By passing to an equivalent category, one can
(and we often will) assume that $\theta_{g,h}=\mathrm{id},\forall g,h\in G$ in the action $(\rho,\theta)$ because of \cite[Theorem 5.4]{Shinder:18}. Here, we say that a $G$-action $(\rho,\theta)$ on $\mathcal{D}$ is equivalent to a $G$-action $(\rho',\theta')$ on $\mathcal{D}'$ if we have an equivalence in $G$-$\mathfrak{Cats}$, that is, $(\mathcal{D},\rho,\theta)\simeq(\mathcal{D}',\rho',\theta')$.

Moreover, for a finite abelian group $G$ and a smooth projective variety $X$ over a field $k$, the set of equivalence classes of $G$-actions on $\mathbf{D}^b(X)$ is an $\mathrm{H}^2(G,k^*)$-torsor by \cite[Theorem 1.11]{Bayer Perry:23}.

Now, we get to the definition of equivariant categories.

Let $(\rho,\theta)$ be an action of a finite group $G$ on an additive $\mathbb{C}$-linear category $\mathcal{D}$.

\begin{definition}\label{def of equivariant category}
    The \textit{equivariant category} $\mathcal{D}_G$ is defined as follows.
    \begin{enumerate}
        \item Objects of $\mathcal{D}_G$ are pairs $(E,\phi_E)$, where $E$ is an object in $\mathcal{D}$ and $$\phi_E=(\phi_{E,g}:E\to\rho_g(E))_{g\in G}$$ 
        is a family of isomorphisms such that the diagram
        \begin{align}\label{phi_gh diagram}
            \xymatrix{
            E\ar@(dr,dl)[rrrrrr]^{\phi_{E,gh}}\ar[rr]^(0.45){\phi_{E,g}}&&\rho_g(E)\ar[rr]^(0.45){\rho_g(\phi_{E,h})}&&\rho_g(\rho_h(E))\ar[rr]^(0.55){\theta^E_{g,h}}&&\rho_{gh}(E)
            }
        \end{align}
        commutes for all $g,h\in G$.
        \item A morphism from $(E,\phi_E)$ to $(E',\phi_E)$ is a morphism $m:E\to E'$ in $\mathcal{D}$ which commutes with linearizations, i.e. such that the diagram
        \begin{align}\label{morphism in D_G diagram}
            \xymatrix{
            E\ar[rr]^{m}\ar[d]_{\phi_{E,g}}&&E'\ar[d]^{\phi_{E',g}}\\
            \rho_g(E)\ar[rr]^{\rho_g(m)}&&\rho_g(E')
            }
        \end{align}
        commutes for every $g\in G$.
    \end{enumerate}
\end{definition}

After recalling some contents of \cite[subsections 2.1 and 3.1]{Beckmann Oberdieck:23}, we develop two natural maps $F,\lambda$ from the set of $G$-functors.

\begin{passage}\label{def of F lambda}
    Let $G$ be a finite group, and let $(\rho,\theta),(\rho',\theta')$ be an action of $G$ on an additive $\mathbb{C}$-linear category $\mathcal{D},\mathcal{D'}$ respectively. Denote $G\textrm{-}\mathrm{Fun}(\mathcal{D}_G,\mathcal{D}'_G)$ as the set of $G$-functors from $\mathcal{D}_G$ to $\mathcal{D}'_G$. Denote $\mathrm{Fun}(\mathcal{D},\mathcal{D}')$ as the set of functors from $\mathcal{D}$ to $\mathcal{D}'$. Then we have a natural forgetful map $$F:G\textrm{-}\mathrm{Fun}(\mathcal{D}_G,\mathcal{D}'_G)\to\mathrm{Fun}(\mathcal{D},\mathcal{D}'), (f,\sigma)\mapsto f.$$

Denote $\mathrm{Fun}(\mathcal{D}_G,\mathcal{D}'_G)$ as the set of functors from $\mathcal{D}_G$ to $\mathcal{D}'_G$. We may define another natural map $\lambda:G\textrm{-}\mathrm{Fun}(\mathcal{D}_G,\mathcal{D}'_G)\to\mathrm{Fun}(\mathcal{D}_G,\mathcal{D}'_G)$. 

Precisely, for $(f,\sigma)\in G\textrm{-}\mathrm{Fun}(\mathcal{D}_G,\mathcal{D}'_G)$, the functor $\lambda(f,\sigma)\in\mathrm{Fun}(\mathcal{D}_G,\mathcal{D}'_G)$ can be described as follows.

\begin{enumerate}
    \item The functor $\lambda(f,\sigma)$ maps an object $(E,\phi_E)$ of $\mathcal{D}_G$, where $$\phi_E=(\phi_{E,g}:E\to\rho_g(E))_{g\in G}$$ is a family of isomorphisms that satisfy the commutative diagram (\ref{phi_gh diagram}), to an object $(f(E),\phi_{f(E)})$ of $\mathcal{D}'_G$. Here, the family of isomorphisms  $\phi_{f(E)}$ is composed of $\phi_{f(E),g}:f(E)\xrightarrow[\cong]{f(\phi_{E,g})}f(\rho_g(E))\xrightarrow[\cong]{\sigma_g(E)}\rho'_g(f(E)), \forall g\in G$, such that the diagram 
    \begin{align*}
            \xymatrix{
            f(E)\ar@(dr,dl)[rrrrrr]^{\phi_{f(E),gh}}\ar[rr]^(0.45){\phi_{f(E),g}}&&\rho'_g(f(E))\ar[rr]^(0.45){\rho'_g(\phi_{f(E),h})}&&\rho'_g(\rho'_h(f(E)))\ar[rr]^(0.55){\theta'^{f(E)}_{g,h}}&&\rho'_{gh}(f(E))
            }
        \end{align*}
    commutes for all $g,h\in G$.

    Actually, it comes from the lower left triangle of the following commutative diagram for all $g,h\in G$.
    \begin{align*}
        \xymatrix{
        f(E)\ar@(ur,ul)[rrrrrr]^{f(\phi_{E,gh})}\ar[rr]_(0.45){f(\phi_{E,g})}\ar@{=}[ddd]\ar[drr]_(0.45){\phi_{f(E),g}}&&f(\rho_g(E))\ar[rr]_(0.45){f(\rho_g(\phi_{E,h}))}\ar[d]^{\sigma_g(E)}_{\cong}&&f(\rho_g(\rho_h(E)))\ar[rr]_(0.55){f(\theta^E_{g,h})}\ar[d]_{\sigma_g(\rho_h(E))}^{\cong}&&f(\rho_{gh}(E))\ar[ddd]_{\sigma_{gh}(E)}^{\cong}\\
        &&\rho'_g(f(E))\ar[rr]^(0.45){\rho'_g(f(\phi_{E,h}))}\ar[drr]_{\rho'_g(\phi_{f(E),h})}&&\rho'_g(f(\rho_g(E)))\ar[d]^{\rho'_g(\sigma_h(E))}_{\cong}&&\\
        &&&&\rho'_g(\rho'_h(f(E)))\ar[drr]^(0.55){\theta'^{f(E)}_{g,h}}&&\\
        f(E)\ar[rrrrrr]^{\phi_{f(E),gh}}&&&&&&\rho'_{gh}(f(E))
        }
    \end{align*}
    Here, \begin{enumerate}[(i)]
        \item the top commutative part comes from applying the functor $f\in\mathrm{Fun}(\mathcal{D},\mathcal{D}')$ to (\ref{phi_gh diagram});
        \item the left and middle commutative triangles are originated from the definition of $\phi_{f(E)}$;
        \item the middle commutative square is from the 2-isomorphism $\sigma_g$;
        \item the right commutative part is from (\ref{G-fun sigma commutative diagram});
        \item the outer circle is originated from the definition of $\phi_{f(E)}$.
    \end{enumerate}  
    \item The functor $\lambda(f,\sigma)$ maps a morphism from $(E,\phi_E)$ to $(E',\phi_{E'})$ in $\mathcal{D}_G$, which is a morphism $m:E\to E'$ in $\mathcal{D}$ such that the diagram (\ref{morphism in D_G diagram}) commutes for every $g\in G$, to a morphism from $(f(E),\phi_{f(E)}))$ to $(f(E'),\phi_{f(E')})$ in $\mathcal{D}'_G$, which is a morphism $f(m):f(E)\to f(E')$ in $\mathcal{D}'$ from the functor $f\in\mathrm{Fun}(\mathcal{D},\mathcal{D}')$, such that the diagram 
    \begin{align*}
        \xymatrix{
        f(E)\ar[rr]^{f(m)}\ar[d]_{\phi_{f(E),g}}&&f(E')\ar[d]^{\phi_{f(E'),g}}\\
        \rho'_g(f(E))\ar[rr]^{\rho'_g(f(m))}&&\rho'_g(f(E'))
        }
    \end{align*}
    commutes for every $g\in G$.

    Actually, it comes from the outer circle of the following commutative diagram for every $g\in G$.
    \begin{align*}
        \xymatrix{
        f(E)\ar[dd]_{\phi_{f(E),g}}\ar[rrrrrr]^{f(m)}\ar[drr]_{f(\phi_{f(E),g})}&&&&&&f(E')\ar[dd]^{\phi_{f(E'),g}}\ar[dll]^{f(\phi_{E',g})}\\
        &&f(\rho_g(E))\ar[rr]^{f(\rho_g(m))}\ar[dll]_{\sigma_g(E)}^{\cong}&&f(\rho_g(E'))\ar[drr]^{\sigma_g(E')}_{\cong}&&\\
        \rho'_g(f(E))\ar[rrrrrr]^{\rho'_g(f(m))}&&&&&&\rho'_g(f(E'))
        }
    \end{align*}

    Here, \begin{enumerate}[(i)]
        \item the top commutative trapezoid comes from applying the functor $f$ in $\mathrm{Fun}(\mathcal{D},\mathcal{D}')$ to (\ref{morphism in D_G diagram});
        \item the left and right commutative triangles are originated from the definition of $\phi_{f(E)}$ and $\phi_{f(E')}$ respectively;
        \item the lower commutative trapezoid is from the 2-isomorphism $\sigma_g$.
    \end{enumerate}  
\end{enumerate}

\end{passage}

Note that the natural map $\lambda$ depends on the group $G$ and the group actions $(\rho,\theta),(\rho',\theta')$. But the forgetful map $F$ does not.

\subsection{$G$-functors between equivariant categories of abelian varieties}\label{G-functors between equivariant categories of abelian varieties}

From now on, we will specialize to the categories of abelian varieties. In this subsection, we develop the notation in the previous sections for the categories of abelian varieties, which will be used in later research.

\begin{passage}\label{Db(B) vs DbG(A)}
Let $A$ be an abelian variety, and let $G\leq \mathrm{Pic}^0(A)\cong\widehat{A}$ be a finite subgroup. By the Appell-Humbert theorem, there exists a lift of $G$ to a linearized group of autoequivalences of $\mathbf{D}^b(A)$. Alternatively, conjugating the translation action of $G$ on $\mathbf{D}^b(\widehat{A})$ via Mukai's equivalence $\Phi_{\mathcal{P}}:\mathbf{D}^b(A)\xrightarrow[]{\simeq}\mathbf{D}^b(\widehat{A})$, where $\mathcal{P}$ is the Poincar\'{e} line bundle over $A$, we get the action $$\rho_g:=\Phi_{\mathcal{P}}^{-1}\circ(t_g)_*\circ\Phi_{\mathcal{P}}=-\otimes\mathcal{L}_g:\mathbf{D}^b(A)\xrightarrow[]{\simeq}\mathbf{D}^b(A),$$
which is the derived equivalence associated to $g\in G\leq \widehat{A}$, where $\mathcal{L}_g\in\mathrm{Pic}^0(A)$ is the corresponding line bundle. As in Definition \ref{def of group action on category} for the notation of an action $(\rho,\theta)$ of a subgroup $G$ on $\mathbf{D}^b(A)$, where $\theta_{g,h}=\mathrm{id}$ for
all $g,h\in G$, we get the equivariant category $\mathbf{D}^b_G(A)$ with objects being $G$-equivariant objects. Then the derived category $\mathbf{D}^b_G(A)$ is equivalent to $\mathbf{D}^b(B)$, where $B$ is the abelian variety with homomorphism $q:B\to A$ such that $G=\mathrm{ker}(\widehat{q})$, where the induced map $\widehat{q}$ is $$\widehat{q}:\mathrm{Pic}^0(A)\cong\widehat{A}\to\widehat{B}\cong\mathrm{Pic}^0(B).$$ 
More precisely, $\widehat{B}=\widehat{A}/G$ and $B=\mathrm{Pic}^0(\mathrm{Pic}^0(A)/G)$. 

As in \cite[Example 3.6]{Beckmann Oberdieck:23}, this derived equivalence is denoted by $$\psi_q:\mathbf{D}^b(B)\xrightarrow{\simeq}\mathbf{D}^b_G(A),$$
where an object $F\in\mathbf{D}^b(B)$ is mapped to $(q_*F,\phi_{q_*F})$ with the family of isomorphisms being $$\phi_{q_*F}=(\phi_{q_*F,g}:=-\otimes\mathcal{L}_g:q_*F\xrightarrow{\cong}(q_*F)\otimes\mathcal{L}_g=\rho_g(q_*F))_{g\in G}.$$ The derived equivalence $\psi_q$ also maps a morphism $m:F\to F'$ in $\mathbf{D}^b(B)$ to a morphism from $(q_*F,\phi_{q_*F})$ to $(q_*F',\phi_{q_*F'})$ in $\mathbf{D}^b_G(A)$, which is a morphism $q_*(m)$ from $q_*F$ to $q_*F'$ in $\mathbf{D}^b(A)$ such that the diagram \begin{align*}
    \xymatrix{
    q_*F\ar[rr]^{q_*(m)}\ar[d]_{\phi_{q_*F,g}=-\otimes\mathcal{L}_g}&&q_*F'\ar[d]^{\phi_{q_*F',g}=-\otimes\mathcal{L}_g}\\
    \rho_g(q_*F)\ar[rr]^{\rho_g(q_*(m))}&&\rho_g(q_*F)
    }
\end{align*}
commutes for every $g\in G$, where $\rho_g(q_*(m)):=(-\otimes\mathcal{L}_g)\circ q_*(m)\circ(-\otimes\mathcal{L}_g^{-1})$.

In addition, we have an induced action $(\widehat{\rho},\widehat{\theta})$ of $G$ on $\mathbf{D}^b(\widehat{A})$ with $$\widehat{\rho}_g=(t_g)_*:\mathbf{D}^b(\widehat{A})\xrightarrow{\simeq}\mathbf{D}^b(\widehat{A}),g\in G\leq\widehat{A},\widehat{\theta}=\mathrm{id}.$$ Then we have a natural derived equivalence $$\Xi_q:\mathbf{D}^b(\widehat{B})\xrightarrow{\simeq}\mathbf{D}^b_G(\widehat{A}),$$
where an object $E\in\mathbf{D}^b(\widehat{B})$ is mapped to $(\widehat{q}^*E,\phi_{\widehat{q}^*E})$ with the family of isomorphisms being $$\phi_{\widehat{q}^*E}=(\phi_{\widehat{q}^*E,g}:=(t_g)_*:\widehat{q}^*E\xrightarrow{\cong}(t_g)_*(\widehat{q}^*E))_{g\in G}.$$
The derived equivalence $\Xi_q$ also maps a morphism $m:E\to E'$ in $\mathbf{D}^b(B)$ to a morphism from $(\widehat{q}^*E,\phi_{\widehat{q}^*E})$ to $(\widehat{q}^*E',\phi_{\widehat{q}^*E'})$ in $\mathbf{D}^b_G(\widehat{A})$, which is a morphism $\widehat{q}^*(m)$ from $\widehat{q}^*E$ to $\widehat{q}^*E'$ in $\mathbf{D}^b(\widehat{A})$ such that the diagram \begin{align*}
    \xymatrix{
    \widehat{q}^*E\ar[rr]^{\widehat{q}^*(m)}\ar[d]_{\phi_{\widehat{q}^*E,g}=(t_g)_*}&&\widehat{q}^*E'\ar[d]^{\phi_{\widehat{q}^*E',g}=(t_g)_*}\\
    \widehat{\rho}_g(\widehat{q}^*E)\ar[rr]^{\widehat{\rho}_g(\widehat{q}^*(m))}&&\widehat{\rho}_g(\widehat{q}^*E')
    }
\end{align*}
commutes for every $g\in G$, where $\widehat{\rho}_g(\widehat{q}^*(m)):=(t_g)_*\circ \widehat{q}^*(m)\circ(t_{-g})_*$.
\end{passage}

\begin{passage}\label{G-fun}
Using Definition \ref{def of G-functor}, a \textit{G-functor} from $\mathbf{D}^b_G(A)$ to itself is a pair $(f,\sigma)$ of a functor $f: \mathbf{D}^b(A)\to\mathbf{D}^b(A)$ with a set of natural transformations $$\sigma=(\sigma_g:f\circ \rho_g\to \rho_g\circ f)_{g\in G},$$
which are isomorphisms of functors, such that $\sigma_g$'s are compatible with the associativity natural transformations in the definition of a $G$-action $(\rho,\theta)$ on $\mathbf{D}^b(A)$. We refer to $\{\sigma_g|g\in G\}$ as a set of \textit{G-equivariance natural transformations for f}. 

Let $G\textrm{-}\mathrm{Fun}(\mathbf{D}^b_G(A), \mathbf{D}^b_G(A))$ be the set of $G$-functors from $\mathbf{D}^b_G(A)$ to $\mathbf{D}^b_G(A)$. We get the natural forgetful map $$F_q:G\textrm{-}\mathrm{Fun}(\mathbf{D}^b_G(A), \mathbf{D}^b_G(A))\to\mathrm{Fun}(\mathbf{D}^b(A), \mathbf{D}^b(A)), (f,\sigma)\mapsto f.$$
Also, we obtain a composition $$\lambda_q:G\textrm{-}\mathrm{Fun}(\mathbf{D}^b_G(A), \mathbf{D}^b_G(A))\to\mathrm{Fun}(\mathbf{D}^b_G(A), \mathbf{D}^b_G(A))\cong\mathrm{Fun}(\mathbf{D}^b(B), \mathbf{D}^b(B)).$$ 
Denote $G\textrm{-}\mathrm{Fun}(\mathbf{D}^b(B), \mathbf{D}^b(B)):=\lambda_q(G\textrm{-}\mathrm{Fun}(\mathbf{D}^b_G(A), \mathbf{D}^b_G(A)))$.
\end{passage}

Note that the map $\lambda_q$ is dependent on the group $G$ and the group action $(\rho,\theta)$. This information can be obtained by the map $q$. So, we use the subscript $q$. We use the subscript $q$ for the map $F_q$ in the same premises.

The following two remarks give an overview of the notation above in a general setting.

\begin{remark}
    We may generalize the notation to more general actions of a group $G$ on the derived category of an abelian variety or even a smooth projective variety $A$ to obtain the following maps. 
    \begin{align*}
    \xymatrix{
    &\mathrm{Fun}(\mathbf{D}^b(A), \mathbf{D}^b(A))\\
    G\textrm{-}\mathrm{Fun}(\mathbf{D}^b_G(A), \mathbf{D}^b_G(A))\ar[ur]^{F_q}\ar[dr]^{\lambda_q}&\\
    &\lambda_q(G\textrm{-}\mathrm{Fun}(\mathbf{D}^b_G(A), \mathbf{D}^b_G(A)))\ar[r]^(0.55){\subset}&\mathrm{Fun}(\mathbf{D}^b_G(A), \mathbf{D}^b_G(A))}
    \end{align*}
    But we can not get a variety $B$ such that $\mathbf{D}^b(B)\simeq\mathbf{D}^b_G(A)$ in general. 
\end{remark}

\begin{remark}\label{Fq from G-}
    The objects of $\mathbf{D}^b_G(A)$ are pairs $(E,\phi)$, where $E$ is an object of $\mathbf{D}^b(A)$ and $\phi=(\phi_g:E\to\rho_g(E))_{g\in G}$ is a family of isomorphisms compatible with the natural transformations $\theta=(\theta_{g,h}:\rho_g\circ\rho_h\to\rho_{gh})_{g,h\in G}$ (see Definition \ref{def of equivariant category}). We will refer to $\phi$ as a \textit{G-linearization} of $E$. Not all objects of $\mathbf{D}^b(A)$ admit $G$-linearizations. So, a functor $\Phi$ from $\mathbf{D}^b_G(A)$ to $\mathbf{D}^b_G(A)$ does not need to be a lift of a functor from $\mathbf{D}^b(A)$ to $\mathbf{D}^b(A)$. That is, $\Phi$ does not need to be a $G$-functor. Thus, we define the forgetful map $F_q$ from $G\textrm{-}\mathrm{Fun}(\mathbf{D}^b_G(A), \mathbf{D}^b_G(A))$, but not from $\mathrm{Fun}(\mathbf{D}^b_G(A), \mathbf{D}^b_G(A))$.
\end{remark}

Now we can go back to our notation that $A$ is an abelian variety and $G$ is a finite subgroup of $\mathrm{Pic}^0(A)\cong\widehat{A}$. 

\begin{remark}\label{lambda_q explain B}
Passage \ref{def of F lambda} gives the definition of $\lambda$ in general settings. For our case, another way to understand $\lambda_q$ is to use the commutative diagram by Lemma \ref{lambda preserve integral structure_A}
\begin{align*}
        \xymatrix{
    \mathrm{Fun}(\mathbf{D}^b(A), \mathbf{D}^b(A))\ar@{-->}[r]&\mathrm{Fun}(\mathbf{D}^b(B), \mathbf{D}^b(B))\ar[d]^{\simeq}&\\
    G\textrm{-}\mathrm{Fun}(\mathbf{D}^b_G(A), \mathbf{D}^b_G(A))\ar[u]^{F_q}\ar[r]^{\lambda_q}&\mathrm{Fun}(\mathbf{D}^b_G(A), \mathbf{D}^b_G(A))\ar[r]^{\simeq}&\mathrm{Fun}(\mathbf{D}^b(B), \mathbf{D}^b(B)),}
\end{align*}
where the dashed arrow comes from the lifting so that the diagram below commutes.
\begin{align*}
        \xymatrix{
    \mathbf{D}^b(B)\ar[rr]\ar[d]^{q_*}&\phantom{1}&\mathbf{D}^b(B)\ar[d]^{q_*}\\
    \mathbf{D}^b(A)\ar[rr]&\phantom{1}\ar@{=>}[u]^{\mathrm{lift}}&\mathbf{D}^b(A)}
\end{align*}

Note that such a lift does not always work for every element of $\mathrm{Fun}(\mathbf{D}^b(A), \mathbf{D}^b(A))$ by Remark \ref{Fq from G-}. But it works for every element of $F_q(G\textrm{-}\mathrm{Fun}(\mathbf{D}^b(A), \mathbf{D}^b(A)))$.

\end{remark}

\begin{remark}\label{lambda not inj}
    Note that the natural map $$\lambda_q:G\textrm{-}\mathrm{Fun}(\mathbf{D}^b_G(A), \mathbf{D}^b_G(A))\to\mathrm{Fun}(\mathbf{D}^b_G(A),\mathbf{D}^b_G(A))$$
    is not injective, as elements of its domain involve a choice of autoequivalence in $\mathrm{Fun}(\mathbf{D}^b(A),\mathbf{D}^b(A))$. 
    
    For example, if $L$ is a line bundle on $A$, then tensorization by $q^*L$ is an element of $\mathrm{Aut}(\mathbf{D}^b(B))\simeq\mathrm{Aut}(\mathbf{D}^b_G(A))$, while tensorization by $L$ extends to an element of $G\textrm{-}\mathrm{Aut}(\mathbf{D}^b_G(A))$. 
    
    That is, $\lambda_q((-\otimes L, \sigma=\mathrm{id}))=-\otimes q^*L$. But $\exists L\not=L'\in\mathrm{Pic}(A)$ with $q^*L\cong q^*L'$. So we have $\lambda_q((-\otimes L', \sigma=\mathrm{id}))=-\otimes q^*L$.
\end{remark}

So far, we get natural maps as follows.

\begin{align}\label{Flambda_FunA}
    \xymatrix{
    &\mathrm{Fun}(\mathbf{D}^b(A), \mathbf{D}^b(A))&\\
    G\textrm{-}\mathrm{Fun}(\mathbf{D}^b_G(A), \mathbf{D}^b_G(A))\ar[ur]^{F_q}\ar[dr]^{\lambda_q}&&\\
    &\lambda_q(G\textrm{-}\mathrm{Fun}(\mathbf{D}^b_G(A), \mathbf{D}^b_G(A)))\ar[r]^(0.55){\subset}\ar@{=}[d]&\mathrm{Fun}(\mathbf{D}^b_G(A), \mathbf{D}^b_G(A))\ar@{=}[d]\\
    &G\textrm{-}\mathrm{Fun}(\mathbf{D}^b(B), \mathbf{D}^b(B))\ar[r]^{\subset}&\mathrm{Fun}(\mathbf{D}^b(B), \mathbf{D}^b(B))}
\end{align}

All the notations above can be restricted to autoequivalences between categories. So we get the natural maps below, where we still use the notation $F_q, \lambda_q$ as the maps restricted to autoequivalences. 

\begin{align}\label{Flambda_AutA}
    \xymatrix{
    &\mathrm{Aut}(\mathbf{D}^b(A))&\\
    G\textrm{-}\mathrm{Aut}(\mathbf{D}^b_G(A))\ar[ur]^{F_q}\ar[dr]^{\lambda_q}&&\\
    &\lambda_q(G\textrm{-}\mathrm{Aut}(\mathbf{D}^b_G(A)))\ar[r]^(0.55){\subset}\ar@{=}[d]&\mathrm{Aut}(\mathbf{D}^b_G(A))\ar@{=}[d]\\
    &G\textrm{-}\mathrm{Aut}(\mathbf{D}^b(B))\ar[r]^{\subset}&\mathrm{Aut}(\mathbf{D}^b(B))}
\end{align}

\begin{remark}\label{AA'notation}
    Furthermore, given two $n$-dimensional abelian varieties $A$ and $A'$ and embeddings of a finite subgroup $G$ of both $\mathrm{Pic}^0(A)$ and $\mathrm{Pic}^0(A')$, we get linearization $(\rho',\theta')$ of $G$ on $\mathbf{D}^b(A')$ similar to $(\rho, \theta)$ of $G$ on $\mathbf{D}^b(A)$ in passage \ref{Db(B) vs DbG(A)}. In addition, the derived category $\mathbf{D}^b_G(A')$ is equivalent to $\mathbf{D}^b(B')$, where $B'$ is the abelian variety with homomorphism $q':B'\to A'$ such that $G'=\mathrm{ker}(\widehat{q'})\cong G$, where the induced map $\widehat{q'}$ is $\mathrm{Pic}^0(A')\cong\widehat{A'}\to\widehat{B'}\cong\mathrm{Pic}^0(B')$. More precisely, $$\widehat{B'}=\widehat{A'}/G, B'=\mathrm{Pic}^0(\mathrm{Pic}^0(A')/G).$$ 
    
    By Definition \ref{def of G-functor}, a \textit{G-functor} from $\mathbf{D}^b_G(A)$ to $\mathbf{D}^b_G(A')$ is a pair $(f,\sigma)$ of a functor $f: \mathbf{D}^b(A)\to\mathbf{D}^b(A')$ with a set of natural transformations $$\sigma=(\sigma_g:f\circ \rho_g\to \rho'_g\circ f)_{g\in G},$$
    which are isomorphisms of functors, such that $\sigma_g$'s are compatible with the associativity natural transformations in the definition of $G$-actions $(\rho,\theta)$ and $(\rho',\theta')$ on $\mathbf{D}^b(A)$ and $\mathbf{D}^b(A')$ respectively. With similar notations, we get maps and the ones restricted to equivalences of categories.

\begin{align}\label{Flambda_FunAA'}
    \xymatrix{
    &\mathrm{Fun}(\mathbf{D}^b(A), \mathbf{D}^b(A'))&\\
    G\textrm{-}\mathrm{Fun}(\mathbf{D}^b_G(A), \mathbf{D}^b_G(A'))\ar[ur]^{F_{q,q'}}\ar[dr]^{\lambda_{q,q'}}&&\\
    &\lambda_{q,q'}(G\textrm{-}\mathrm{Fun}(\mathbf{D}^b_G(A), \mathbf{D}^b_G(A')))\ar[r]^(0.55){\subset}\ar@{=}[d]&\mathrm{Fun}(\mathbf{D}^b_G(A), \mathbf{D}^b_G(A'))\ar@{=}[d]\\
    &G\textrm{-}\mathrm{Fun}(\mathbf{D}^b(B), \mathbf{D}^b(B'))\ar[r]^{\subset}&\mathrm{Fun}(\mathbf{D}^b(B), \mathbf{D}^b(B'))}
\end{align}

\begin{align}\label{Flambda_EqAA'}
    \xymatrix{
    &\mathrm{Eq}(\mathbf{D}^b(A), \mathbf{D}^b(A'))&\\
    G\textrm{-}\mathrm{Eq}(\mathbf{D}^b_G(A), \mathbf{D}^b_G(A'))\ar[ur]^{F_{q,q'}}\ar[dr]^{\lambda_{q,q'}}&&\\
    &\lambda_{q,q'}(G\textrm{-}\mathrm{Eq}(\mathbf{D}^b_G(A), \mathbf{D}^b_G(A')))\ar[r]^(0.55){\subset}\ar@{=}[d]&\mathrm{Eq}(\mathbf{D}^b_G(A), \mathbf{D}^b_G(A'))\ar@{=}[d]\\
    &G\textrm{-}\mathrm{Eq}(\mathbf{D}^b(B), \mathbf{D}^b(B'))\ar[r]^{\subset}&\mathrm{Eq}(\mathbf{D}^b(B), \mathbf{D}^b(B'))}
\end{align}

Moreover, we have a similar way to understand $\lambda_{q,q'}$ as in Remark \ref{lambda_q explain B}.

Indeed, we have commutative diagram by Lemma \ref{lambda preserve integral structure_AA'}
\begin{align*}
        \xymatrix{
    \mathrm{Fun}(\mathbf{D}^b(A), \mathbf{D}^b(A'))\ar@{-->}[r]&\mathrm{Fun}(\mathbf{D}^b(B), \mathbf{D}^b(B'))\ar[d]^{\simeq}&\\
    G\textrm{-}\mathrm{Fun}(\mathbf{D}^b_G(A), \mathbf{D}^b_G(A'))\ar[u]^{F_{q,q'}}\ar[r]^(0.5){\lambda_{q,q'}}&\mathrm{Fun}(\mathbf{D}^b_G(A), \mathbf{D}^b_G(A'))\ar[r]^{\simeq}&\mathrm{Fun}(\mathbf{D}^b(B), \mathbf{D}^b(B')),}
\end{align*}
where the dashed arrow comes from the lifting so that the diagram below commutes.
\begin{align*}
        \xymatrix{
    \mathbf{D}^b(B)\ar[rr]\ar[d]^{q_*}&\phantom{1}&\mathbf{D}^b(B')\ar[d]^{q'_*}\\
    \mathbf{D}^b(A)\ar[rr]&\phantom{1}\ar@{=>}[u]^{\mathrm{lift}}&\mathbf{D}^b(A')}
\end{align*}

Note that such a lift does not always work for every element of $\mathrm{Fun}(\mathbf{D}^b(A), \mathbf{D}^b(A'))$ by Remark \ref{Fq from G-}. But it works for every element of $F_{q,q'}(G\textrm{-}\mathrm{Fun}(\mathbf{D}^b(A), \mathbf{D}^b(A')))$.
\end{remark}

From now on, we focus on the finite group $G\leq \mathrm{Pic}^0(A)$ (and $G\cong G'\leq \mathrm{Pic}^0(A')$).

\begin{example}\label{G=Ahat[n]}
    If $G=\widehat{A}[n]$, which is the subgroup of torsion points of order dividing $n\in\mathbb{Z}\backslash\{0\}$, then $B=A$ and $q=n_A:B=A\to A$ is the multiplication by $n$. So we get 
    \begin{align*}
        \xymatrix{
        &\mathrm{Fun}(\mathbf{D}^b(A),\mathbf{D}^b(A))&\\
        G\textrm{-}\mathrm{Fun}(\mathbf{D}^b_G(A), \mathbf{D}^b_G(A))\ar[ur]^{F_{n_A}}\ar[dr]^{\lambda_{n_A}}&&\\
        &G\textrm{-}\mathrm{Fun}(\mathbf{D}^b(A),\mathbf{D}^b(A))\ar[r]^(0.53){\subset}&\mathrm{Fun}(\mathbf{D}^b(A), \mathbf{D}^b(A))
        }
    \end{align*}
    and
    \begin{align*}
        \xymatrix{
        &\mathrm{Aut}(\mathbf{D}^b(A))&\\
        G\textrm{-}\mathrm{Aut}(\mathbf{D}^b_G(A))\ar[ur]^{F_{n_A}}\ar[dr]^{\lambda_{n_A}}&&\\
        &G\textrm{-}\mathrm{Aut}(\mathbf{D}^b(A))\ar[r]^(0.53){\subset}&\mathrm{Aut}(\mathbf{D}^b(A)).
        }
    \end{align*}
    Moreover, if $G=\widehat{A'}[n]$, then $B'=A'$ and $q'=n_{A'}:B'=A'\to A'$ is the multiplication by $n$. We obtain 
    \begin{align*}
        \xymatrix{
        &\mathrm{Fun}(\mathbf{D}^b(A),\mathbf{D}^b(A'))&\\
        G\textrm{-}\mathrm{Fun}(\mathbf{D}^b_G(A), \mathbf{D}^b_G(A'))\ar[ur]^{F_{n_A,n_{A'}}}\ar[dr]^{\lambda_{n_A,n_{A'}}}&&\\
        &G\textrm{-}\mathrm{Fun}(\mathbf{D}^b(A),\mathbf{D}^b(A'))\ar[r]^(0.53){\subset}&\mathrm{Fun}(\mathbf{D}^b(A), \mathbf{D}^b(A'))
        }
    \end{align*}
    and
     \begin{align*}
        \xymatrix{
        &\mathrm{Eq}(\mathbf{D}^b(A),\mathbf{D}^b(A'))&\\
        G\textrm{-}\mathrm{Eq}(\mathbf{D}^b_G(A), \mathbf{D}^b_G(A'))\ar[ur]^{F_{n_A,n_{A'}}}\ar[dr]^{\lambda_{n_A,n_{A'}}}&&\\
        &G\textrm{-}\mathrm{Eq}(\mathbf{D}^b(A),\mathbf{D}^b(A'))\ar[r]^(0.53){\subset}&\mathrm{Eq}(\mathbf{D}^b(A), \mathbf{D}^b(A')).
        }
    \end{align*}
\end{example}

\subsection{Relating the images of $\lambda_q(f,\sigma)$ (and $\lambda_{q,q'}(f,\sigma)$) and $f$ via Orlov's representation.}\label{Relating the images of lambda_q(f,sigma) and f via Orlov's representation}

To start, we state Orlov's short exact sequence for the derived equivalences of abelian varieties clearly.

Let notation be as above over an algebraically closed field of characteristic $0$. Set $V_A:=\mathrm{H}^1(A\times \widehat{A}, \mathbb{Z})=\mathrm{H}^1(A, \mathbb{Z})\times\mathrm{H}^1(\widehat{A}, \mathbb{Z})$ with the pairing $(-,-)_{V_A}$ defined to be $$((\alpha_1,\beta_1),(\alpha_2,\beta_2))_{V_A}:=\beta_1(\alpha_2)+\beta_2(\alpha_1),$$ where $(\alpha_1,\beta_1), (\alpha_2,\beta_2)\in V_A\simeq\mathrm{H}^1(A, \mathbb{Z})\times\mathrm{H}^1({A}, \mathbb{Z})^*$. Then $V_A$ is an even unimodular
lattice, the orthogonal direct sum of four copies of the hyperbolic plane. For the bilinear symmetric form $Q_A(v,v'):=\frac{1}{2}(v,v')_{V_A}$ for $v,v'\in V_A$, we get an integral quadratic form $Q_A(v):=Q_A(v,v)$. Similar for the lattice $(V_{A'},(-,-)_{V_{A'}})$ and the form $Q_{A'}$. Let $\mathrm{SO}(V_A,V_A')$ be the set of special orthogonal maps from $V_A$ to $V_{A'}$ with respect to the pairings $(-,-)_{V_A}, (-,-)_{V_{A'}}$. The subset $\mathrm{SO}_{\mathrm{Hdg}}(V_A,V_{A'})$ consists of all the maps in $\mathrm{SO}(V_A,V_{A'})$ that preserve the Hodge structures of $V_A, V_{A'}$. 

Recall in \cite[Definition 9.46]{Huybrechts:06}, the set of symplectic maps $\mathrm{Sp}(A,A')$ is composed of all isomorphisms $g=\begin{pmatrix}
    g_1&g_2\\
    g_3&g_4
\end{pmatrix}:A\times \widehat{A}\xrightarrow{\cong}A'\times\widehat{A'}$ such that $g^{-1}=\begin{pmatrix}
    \phantom{-}\widehat{g_4}&-\widehat{g_2}\\
    -\widehat{g_3}&\phantom{-}\widehat{g_1}
\end{pmatrix}$. In particular, for $A=A'$ case, $\mathrm{Sp}(A):=\mathrm{Sp}(A,A)$ is a subgroup of $\mathrm{Aut}(A\times\widehat{A})$.

We recall Orlov's fundamental short exact sequence for derived equivalences of abelian varieties in the following Theorem.

\begin{theorem}[Orlov, {\cite[Theorem 5.1]{Magni:22}}]\label{Orlov fundamental th}
    Let $A$ and $A'$ be two abelian varieties over an algebraically closed field of characteristic $0$, then we have a short exact sequence of groups $$0\to\mathrm{Alb}(A)\times\widehat{A}\times\mathbb{Z}\to\mathrm{Aut}(\mathbf{D}^b(A))\xrightarrow{\gamma_A}\mathrm{Sp}(A)\to 0,$$ and a surjective map $\gamma_{A,A'}:\mathrm{Eq}(\mathbf{D}^b(A),\mathbf{D}^b(A'))\twoheadrightarrow\mathrm{Sp}(A,A')$. Here, $1\in\mathbb{Z}$ is mapped to the shift functor $[1]$, $a\in A\cong\mathrm{Alb}(A)$ is mapped to the derived equivalence $(t_a)_*$, induced from translation and $\alpha\in\widehat{A}$ is mapped to the derived equivalence $-\otimes\mathcal{L}_\alpha$, by tensoring with the line bundle $\mathcal{L}_{\alpha}\in\mathrm{Pic}^0(A)$ corresponding to $\alpha\in\widehat{A}$.

    Moreover, the maps $\gamma_A$ and $\gamma_{A,A'}$ are compatible in the sense that $$\gamma_{A,A'}(\Phi'\circ\Phi)=\gamma_{A,A'}(\Phi')\circ\gamma_A(\Phi),\forall\Phi\in\mathrm{Aut}(\mathbf{D}^b(A)),\forall\Phi'\in\mathrm{Eq}(\mathbf{D}^b(A),\mathbf{D}^b(A')).$$
\end{theorem}

Note that the proof of surjectivity of $\gamma_A$ is complete using the Appendix in the new 2025 version of the paper \cite{Orlov:25}.

\begin{remark}\label{Spin vs Orlov representation}
    On the other hand, using the notation of the spin group in \cite[subsection 4.3]{Golyshev:01}, we may denote the spin representation of the group of derived autoequivalences of an abelian variety $A$ by $$\dagger_A:\mathrm{Aut}(\mathbf{D}^b(A))\to\mathrm{Spin}_{\mathrm{Hdg}}(V_A).$$ For an arbitrary choice of a parallel transport operator $\eta:\mathrm{H}^*(A,\mathbb{Z})\to\mathrm{H}^*(A',\mathbb{Z})$ for abelian varieties $A,A'$, we may define the spin set $\mathrm{Spin}(V_A,V_{A'})$ to be the set of maps $f:\mathrm{H}^*(A,\mathbb{Z})\to\mathrm{H}^*(A',\mathbb{Z})$ such that $\eta^{-1}\circ f$ belongs to the image of $\mathrm{Spin}(V_A)$ in $\mathrm{GL}(\mathrm{H}^*(A,\mathbb{Z}))$. Then, we obtain the spin map $$\dagger_{A,A'}:\mathrm{Eq}(\mathbf{D}^b(A), \mathbf{D}^b(A'))\to\mathrm{Spin}_{\mathrm{Hdg}}(V_A,V_{A'}).$$

    The spin representations and the Orlov's representations above are related by \cite[Corollary 9.61]{Huybrechts:06} and \cite[Corollary 4.3.8]{Golyshev:01}. To be precise, we get a commutative diagram with exact rows and columns of group homomorphisms.
    \begin{align}\label{Spin vs Orlov diagram A}
      \xymatrix{
      &\mathbb{Z}/2\mathbb{Z}&&&\\
      0\ar[r]&\mathrm{Alb}(A)\times\widehat{A}\times\mathbb{Z}\ar[r]\ar@{->>}[u]&\mathrm{Aut}(\mathbf{D}^b(A))\ar[r]^(0.6){\gamma_A}&\mathrm{Sp}(A)\ar[r]&0\\
      0\ar[r]&\mathrm{Alb}(A)\times\widehat{A}\times2\mathbb{Z}\ar[r]\ar@{^(-_>}[u]&\mathrm{Aut}(\mathbf{D}^b(A))\ar[r]^{\dagger_A}\ar@{=}[u]&\mathrm{Spin}_{\mathrm{Hdg}}(V_A)\ar[r]\ar@{->>}[u]^{\lambda_A}_{2:1}&0\\
      &&&\mathbb{Z}/2\mathbb{Z}\ar@{^(-_>}[u]&
    }
  \end{align}
  Here, the map $\lambda_A$ is a homomorphism with kernel generated by $\dagger_A([1])$ of order 2. In addition, we have a commutative diagram in general.
  \begin{align*}
      \xymatrix{
      \mathrm{Eq}(\mathbf{D}^b(A), \mathbf{D}^b(A'))\ar@{->>}[r]^(0.6){\gamma_{A,A'}}&\mathrm{Sp}(A,A')\\
      \mathrm{Eq}(\mathbf{D}^b(A), \mathbf{D}^b(A'))\ar@{->>}[r]^{\dagger_{A,A'}}\ar@{=}[u]&\mathrm{Spin}_{\mathrm{Hdg}}(V_A,V_{A'})\ar@{->>}[u]^{\lambda_{A,A'}}_{2:1}\\
    }
  \end{align*}
  Here, the map $\lambda_{A,A'}$ is a map with kernel of order 2.

  Moreover, the maps $\dagger_A$ and $\dagger_{A,A'}$ are compatible in the sense that $$\dagger_{A,A'}(\Phi'\circ\Phi)=\dagger_{A,A'}(\Phi')\circ\dagger_A(\Phi),\forall\Phi\in\mathrm{Aut}(\mathbf{D}^b(A)),\forall\Phi'\in\mathrm{Eq}(\mathbf{D}^b(A),\mathbf{D}^b(A')).$$

\end{remark}

    \phantom{111}

    Now, we may introduce a subgroup $\mathrm{SO}^+(V_A):=\mathrm{SO}(V_A)\cap\mathrm{ker}(\mathrm{sn}_{\mathbb{R}})$ of $\mathrm{SO}(V_A)$ consisting of maps with a trivial spinor norm, where $\mathrm{ker}(\mathrm{sn}_{\mathbb{R}})\subset\mathrm{O}(V_A\otimes\mathbb{R})$ as below.%
    \footnote{Note that any element of $\mathrm{O}(V_A)$ can be represented as a product of reflections $s_{v_1}s_{v_2}...s_{v_m}$, where $v_i\in V_A$ and $(v_i,v_i)=\pm2$ by \cite[4.3]{Wall:62}.}%

    As in the notation before \cite[Theorem 1.1]{Gritsenko Hulek Sankaran:09}, in general, the spinor norm over a field $K\not=\mathbb{F}_2$ for an integral even lattice $(L,(-,-))$ is defined to be a group homomorphism $\mathrm{sn}_K:\mathrm{O}(L\otimes K)\to K^\times/(K^\times)^2$ with $$\mathrm{sn}_K(g)=(-\frac{(v_1,v_1)}{2})\cdot...\cdot(-\frac{(v_m,v_m)}{2})(K^\times)^2,$$ where for $K\not=\mathbb{F}_2$, any $g\in\mathrm{O}(L\otimes K)$ can be represented as a product of reflections $g=s_{v_1}s_{v_2}...s_{v_n}$, where $v_i\in L\otimes K$.

    The subgroup $\mathrm{SO}^+_{\mathrm{Hdg}}(V_A)$ consists of all the maps in $\mathrm{SO}^+(V_A)$ that preserve the Hodge structure of $V_A$. By \cite[Lemma 4.1]{Markman:23}, we have a homomorphism $$\mathrm{Spin}_{\mathrm{Hdg}}(V_A)\to\mathrm{SO}^+_{\mathrm{Hdg}}(V_A),$$ which is a double covering of $\mathrm{SO}^+_{\mathrm{Hdg}}(V_A)$.%
    \footnote{In general, for a real or rational vector space $V$ with a bilinear symmetric form $Q$, we have a double covering homomorphism $\mathrm{Spin}(V,Q)\to\mathrm{SO}^+(V,Q)$ by \cite[Lemma 4.1]{Markman:23}. But for a complex vector space $V$ with a bilinear symmetric form $Q$, we have a double covering homomorphism $\mathrm{Spin}(V,Q)\to\mathrm{SO}(V,Q)$ by \cite[Proposition 20.28]{Fulton Harris:04}. Actually, the difference between the two results comes from the part of the proof related to the generators of the orthogonal group $\mathrm{O}(V,Q)$. For the real or rational vector space case, the group $\mathrm{O}(V,Q)$ can be proved to be generated by reflections $s_v$, where $(v,v)_V=\pm2$ after normalization. For the case of complex vector space, the group $\mathrm{O}(V,Q)$ is generated by reflections $s_v$, where $(v,v)_V=-2$ after normalization over complex numbers. Moreover, the difference can be explained by the fact that the group $\mathrm{Spin}$ for a vector space over a field has no non-trivial characters.}%

\begin{remark}\label{two Orlov's representation remark}
    Actually, a symplectic map in $\mathrm{Sp}(A,A')$ corresponds to a Hodge isometry in $\mathrm{SO}_{\mathrm{Hdg}}(V_A,V_{A'})$ in the sense of Remark \ref{SP vs SO remark} over an algebraically closed field of characteristic $0$. Moreover, $f\in\mathrm{Sp}(A)\Leftrightarrow F\in\mathrm{SO}^+_{\mathrm{Hdg}}(V_A)$ over an algebraically closed field of characteristic $0$ using the diagram (\ref{Spin vs Orlov diagram A}) and the double covering $\mathrm{Spin}_{\mathrm{Hdg}}(V_A)\to\mathrm{SO}^+_{\mathrm{Hdg}}(V_A)$.%
    \footnote{It leads to $\mathrm{Sp}(A)=\mathrm{SO}^+(V_A)$ as an advanced result of \cite[Proposition 4.3.2]{Golyshev:01}.}%

    In the rest of the paper, we will use the following two sets of notation to present Orlov's representation and Orlov's map. \begin{align*}
    &\gamma_A:\mathrm{Aut}(\mathbf{D}^b(A))\to\mathrm{Sp}(A), \phantom{111111}\gamma_{A,A'}:\mathrm{Eq}(\mathbf{D}^b(A),\mathbf{D}^b(A'))\to\mathrm{Sp}(A,A')\\
    &\rho_A:\mathrm{Aut}(\mathbf{D}^b(A))\to\mathrm{SO}_{\mathrm{Hdg}}(V_A), \phantom{11}\rho_{A,A'}:\mathrm{Eq}(\mathbf{D}^b(A),\mathbf{D}^b(A'))\to\mathrm{SO}_{\mathrm{Hdg}}(V_A,V_{A'})
\end{align*} We define a subset of $\mathrm{SO}(V_A,V_{A'})$ to be $$\mathrm{SO}^+(V_A,V_{A'}):=\{F\in\mathrm{SO}(V_A,V_{A'})|F_{A,A'}^{-1}\circ F\in\mathrm{SO}^+(V_A)\}, \forall F_{A,A'}\in\mathrm{image}(\rho_{A,A'}).$$ It is a well-defined set. We get a subset $\mathrm{SO}^+_{\mathrm{Hdg}}(V_A,V_{A'})$ consists of all maps in $\mathrm{SO}^+(V_A,V_{A'})$ that preserve the Hodge structures of $V_A, V_{A'}$. Then, $$\mathrm{image}(\rho_A)=\mathrm{SO}^+_{\mathrm{Hdg}}(V_A), \mathrm{image}(\rho_{A,A'})=\mathrm{SO}^+_{\mathrm{Hdg}}(V_A, V_{A'}).$$ Moreover, we have the following two sets of equivalenet notation for Orlov's representation and Orlov's map.
\begin{align*}
    &\gamma_A:\mathrm{Aut}(\mathbf{D}^b(A))\to\mathrm{Sp}(A), \phantom{111111}\gamma_{A,A'}:\mathrm{Eq}(\mathbf{D}^b(A),\mathbf{D}^b(A'))\to\mathrm{Sp}(A,A')\\
    &\rho_A:\mathrm{Aut}(\mathbf{D}^b(A))\to\mathrm{SO}^+_{\mathrm{Hdg}}(V_A), \phantom{11}\rho_{A,A'}:\mathrm{Eq}(\mathbf{D}^b(A),\mathbf{D}^b(A'))\to\mathrm{SO}^+_{\mathrm{Hdg}}(V_A,V_{A'})
\end{align*}

Here, $\gamma_{A,A'}(\Phi_\mathcal{E})\in\mathrm{Sp}(A,A')$ corresponds to $\rho_{A,A'}(\Phi_\mathcal{E})\in\mathrm{SO}^+_{\mathrm{Hdg}}(V_A,V_{A'})$ for a derived equivalence $\Phi_\mathcal{E}\in\mathrm{Eq}(\mathbf{D}^b(A),\mathbf{D}^b(A'))$ in the sense of Remark \ref{SP vs SO remark}, where $\rho_{A,A'}(\Phi_\mathcal{E})$ is the restriction of the induced Hodge isometry in cohomology $$(\gamma_{A,A'}(\Phi_\mathcal{E}))_*:\mathrm{H}^*(A\times\widehat{A},\mathbb{Q})\xrightarrow{\simeq}\mathrm{H}^*(A'\times\widehat{A'},\mathbb{Q})$$
to $\rho_{A,A'}(\Phi_\mathcal{E}):V_A=\mathrm{H}^1(A\times\widehat{A},\mathbb{Z})\xrightarrow{\simeq}\mathrm{H}^1(A'\times\widehat{A'},\mathbb{Z})=V_{A'}$. It is denoted by $$\rho_{A,A'}(\Phi_\mathcal{E})=(\gamma_{A,A'}(\Phi_\mathcal{E}))_*|_{V_A,V_{A'}}.$$ 
In particular, the group $\mathrm{SO}_{\mathrm{Hdg}}(V_A)$ embeds naturally as a group of automorphisms of $A\times\widehat{A}$. 

In addition, the map $\rho_A$ fits the exact sequence $$0\to\mathrm{Alb}(A)\times\widehat{A}\times\mathbb{Z}\to\mathrm{Aut}(\mathbf{D}^b(A))\xrightarrow[]{\rho_A}\mathrm{SO}^+_{\mathrm{Hdg}}(V_A)\to0.$$
As in Theorem \ref{Orlov fundamental th}, the maps $\rho_A$ and $\rho_{A,A'}$ are compatible in the sense that $$\rho_{A,A'}(\Phi'\circ\Phi)=\rho_{A,A'}(\Phi')\circ\rho_A(\Phi),\forall\Phi\in\mathrm{Aut}(\mathbf{D}^b(A)),\forall\Phi'\in\mathrm{Eq}(\mathbf{D}^b(A),\mathbf{D}^b(A')).$$
\end{remark}

\phantom{111}

Returning to the finite subgroup $G$, the homomorphism $q:B\to A$ induces several homomorphisms in cohomology.

The homomorphism $q:B\to A$ induces the map $\widehat{q}:\widehat{A}\to\widehat{B}$. They induce homomorphisms of cohomology $q^*:\mathrm{H}^1(A,\mathbb{Q})\to \mathrm{H}^1(B,\mathbb{Q}), \widehat{q}^*:\mathrm{H}^1(\widehat{B},\mathbb{Q})\to \mathrm{H}^1(\widehat{A},\mathbb{Q})$. We have $q^*(\mathrm{H}^1(A,\mathbb{Z}))\subset\mathrm{H}^1(B,\mathbb{Z})$ and $\mathrm{H}^1(\widehat{A},\mathbb{Z})\subset\mathrm{H}^1(\widehat{B},\mathbb{Q})$ as natural inclusions, where $\widehat{q}^*\mathrm{H}^1(\widehat{B},\mathbb{Z})\subset\mathrm{H}^1(\widehat{A},\mathbb{Z})\subset\mathrm{H}^1(\widehat{A},\mathbb{Q})\xrightarrow[\simeq]{(\widehat{q}^*)^{-1}}\mathrm{H}^1(\widehat{B},\mathbb{Q})$ fits the latter.

We have the isometry $\iota:=(q^*,(\widehat{q}^*)^{-1}): V_{A,\mathbb{Q}}:=V_A\otimes_{\mathbb{Z}}\mathbb{Q}\to V_{B,\mathbb{Q}}:=V_B\otimes_{\mathbb{Z}}\mathbb{Q}$. Then we define $G\textrm{-}\mathrm{SO}^+_{\mathrm{Hdg}}(V_B)\leq\mathrm{SO}^+_{\mathrm{Hdg}}(V_B)$ as the subgroup consisting of all Hodge isometries $\gamma\in\mathrm{SO}^+_{\mathrm{Hdg}}(V_B)$ such that the extension of $\gamma$ to $V_{B,\mathbb{Q}}$ leaves the lattice $q^*\mathrm{H}^1(A,\mathbb{Z})\oplus\mathrm{H}^1(\widehat{A},\mathbb{Z})$ invariant as the lattice in $V_{B,\mathbb{Q}}$ via the isometry $\iota$. Then
\begin{align}
    G\textrm{-}\mathrm{SO}^+_{\mathrm{Hdg}}(V_B)=\mathrm{SO}^+_{\mathrm{Hdg}}(V_B)\cap (\iota\circ\mathrm{SO}^+_{\mathrm{Hdg}}(V_A)\circ\iota^{-1}).
\end{align}

Similarly, the map $q':B'\to A'$ induces maps $\widehat{q'}, (q')^*,(\widehat{q'})^*$ with natural inclusions $(q')^*(\mathrm{H}^1(A',\mathbb{Z}))\subset\mathrm{H}^1(B',\mathbb{Z})$ and $\mathrm{H}^1(\widehat{A'},\mathbb{Z})\subset\mathrm{H}^1(\widehat{B'},\mathbb{Q})$. The latter is compatible with $(\widehat{q'})^*$. 
We have the isometry $\iota':=((q')^*,(\widehat{q'}^*)^{-1}): V_{A',\mathbb{Q}}\to V_{B',\mathbb{Q}}$. Then we define $G\textrm{-}\mathrm{SO}^+_{\mathrm{Hdg}}(V_B,V_{B'})\subset\mathrm{SO}^+_{\mathrm{Hdg}}(V_B,V_{B'})$ as the subset consisting of all Hodge isometries $\gamma\in\mathrm{SO}^+_{\mathrm{Hdg}}(V_B,V_{B'})$ such that the extension of $\gamma$ to $V_{B,\mathbb{Q}}\to V_{B',\mathbb{Q}}$ maps the lattice $q^*\mathrm{H}^1(A,\mathbb{Z})\oplus\mathrm{H}^1(\widehat{A},\mathbb{Z})$ in $V_{B,\mathbb{Q}}$ to the lattice $(q')^*\mathrm{H}^1(A',\mathbb{Z})\oplus\mathrm{H}^1(\widehat{A'},\mathbb{Z})$ in $V_{B',\mathbb{Q}}$ via the isometries $\iota,\iota'$. Then
\begin{align}
    G\textrm{-}\mathrm{SO}^+_{\mathrm{Hdg}}(V_B,V_{B'})=\mathrm{SO}^+_{\mathrm{Hdg}}(V_B,V_{B'})\cap (\iota'\circ\mathrm{SO}^+_{\mathrm{Hdg}}(V_A,V_{A'})\circ\iota^{-1}).
\end{align}

In addition, we may define the subgroup $G\textrm{-}\mathrm{Sp}(A)$ of $\mathrm{Sp}(A)$ and the subset $G\textrm{-}\mathrm{Sp}(A,A')$ of $\mathrm{Sp}(A,A')$ as the subset that corresponds to $G\textrm{-}\mathrm{SO}^+_{\mathrm{Hdg}}(V_A)$ and $G\textrm{-}\mathrm{SO}^+_{\mathrm{Hdg}}(V_A,V_{A'})$ by Remark \ref{two Orlov's representation remark}, respectively.

Now, we may state the main result for this section, which is the $G$-equivariant case of the Orlov's short exact sequence.

\begin{theorem}\label{Th_GrhoA}
    Let the notation be as above over an algebraically closed field of characterstic $0$. We have a short exact sequence $$0\to\mathrm{Alb}(B)\times\widehat{A}\times\mathbb{Z}\to G\textrm{-}\mathrm{Aut}(\mathbf{D}^b_G(A))\xrightarrow[]{G\textrm{-}\rho_A}G\textrm{-}\mathrm{SO}^+_{\mathrm{Hdg}}(V_B)\to 0.$$ It fits into the commutative diagram with exact rows below, where the first and the third exact rows come from Orlov's Theorem \ref{Orlov fundamental th} and Remark \ref{two Orlov's representation remark}.
    \begin{align}\label{diagram_GrhoA}
    \xymatrix{
    0\ar[r]&\mathrm{Alb}(B)\times\widehat{B}\times\mathbb{Z}\ar[r]&\mathrm{Aut}(\mathbf{D}^b(B))\ar[r]^{\rho_B}&\mathrm{SO}^+_{\mathrm{Hdg}}(V_B)\ar[r]&0\\
    0\ar[r]&\mathrm{Alb}(B)\times\widehat{A}\times\mathbb{Z}\ar[r]\ar[u]^{\mathrm{id}_{\mathrm{Alb}(B)}\times\widehat{q}\times\mathrm{id}_{\mathbb{Z}}}\ar[d]_{q_*\times\mathrm{id}_{\widehat{A}}\times\mathrm{id}_{\mathbb{Z}}}&G\textrm{-}\mathrm{Aut}(\mathbf{D}^b(A))\ar[r]^{G\textrm{-}\rho_A}\ar[u]^{\lambda_q}\ar[d]_{F_q}&G\textrm{-}\mathrm{SO}^+_{\mathrm{Hdg}}(V_B)\ar[r]\ar[u]^{\cup}\ar[d]_{\mathrm{res}}&0\\
    0\ar[r]&\mathrm{Alb}(A)\times\widehat{A}\times\mathbb{Z}\ar[r]&\mathrm{Aut}(\mathbf{D}^b(A))\ar[r]^{\rho_A}&\mathrm{SO}^+_{\mathrm{Hdg}}(V_A)\ar[r]&0
    }
\end{align}
\end{theorem}

By the surjectivity of $G\textrm{-}\rho_A$ in (\ref{diagram_GrhoA}), we get
\begin{corollary}\label{Lift derived eq_A}
    Let the notation be as above over an algebraically closed field of characterstic $0$. A derived autoequivalence $\Phi$ in $\mathrm{Aut}(\mathbf{D}^b(A))$ satisfies the condition $(\Phi,\sigma)\in G\textrm{-}\mathrm{Aut}(\mathbf{D}^b(A))$ for $\sigma=(\sigma_g)_{g\in G}$, a set of $G$-equivariance natural transformations, if and only if $\rho_A(\Phi)|_{V_B}\in G\textrm{-}\mathrm{SO}^+_{\mathrm{Hdg}}(V_B)$.
\end{corollary}

We have a generalized version for $\mathrm{Eq}(\mathbf{D}^b(A), \mathbf{D}^b(A'))$.

\begin{theorem}\label{Th_GrhoAA'}
    Let the notation be as above over an algebraically closed field of characteristic $0$. We have a set-theoretic surjective map $$G\textrm{-}\mathrm{Eq}(\mathbf{D}^b_G(A),\mathbf{D}^b_G(A'))\xrightarrow[]{G\textrm{-}\rho_{A,A'}}G\textrm{-}\mathrm{SO}^+_{\mathrm{Hdg}}(V_B,V_{B'}).$$ It fits into the commutative diagram below, where the top and bottom surjective maps come from Orlov's Theorem \ref{Orlov fundamental th} and Remark \ref{two Orlov's representation remark}.
    \begin{align}\label{diagram_GrhoAA'}
    \xymatrix{
    \mathrm{Eq}(\mathbf{D}^b(B),\mathbf{D}^b(B'))\ar[rr]^{\rho_{B,B'}}&&\mathrm{SO}^+_{\mathrm{Hdg}}(V_B,V_{B'})\\
    G\textrm{-}\mathrm{Eq}(\mathbf{D}^b(A),\mathbf{D}^b(A'))\ar[rr]^{G\textrm{-}\rho_{A,A'}}\ar[u]^{\lambda_{q,q'}}\ar[d]_{F_{q,q'}}&&G\textrm{-}\mathrm{SO}^+_{\mathrm{Hdg}}(V_B,V_{B'})\ar[u]^{\cup}\ar[d]_{\mathrm{res}}\\
    \mathrm{Eq}(\mathbf{D}^b(A), \mathbf{D}^b(A'))\ar[rr]^{\rho_{A,A'}}&&\mathrm{SO}^+_{\mathrm{Hdg}}(V_A,V_{A'})
    }
\end{align}
\end{theorem}

By the surjectivity of $G\textrm{-}\rho_{A,A'}$ in (\ref{diagram_GrhoAA'}), we get
\begin{corollary}\label{Lift derived eq_AA'}
    Let the notation be as above over an algebraically closed field of characterstic $0$. A derived equivalence $\Phi\in\mathrm{Eq}(\mathbf{D}^b(A),\mathbf{D}^b(A'))$ satisfies the condition $(\Phi,\sigma)\in G\textrm{-}\mathrm{Eq}(\mathbf{D}^b(A),\mathbf{D}^b(A'))$ for $\sigma=(\sigma_g)_{g\in G}$, a set of $G$-equivariance natural transformations if and only if $\rho_{A,A'}(\Phi)|_{V_B,V_{B'}}\in G\textrm{-}\mathrm{SO}^+_{\mathrm{Hdg}}(V_B,V_{B'})$.
\end{corollary}

Note that the main results above can be expressed in another way using the Orlov's representation $\gamma_{-}$ to the set of symplectic maps. We will not use this kind of expression in the rest of the paper.

\begin{example}\label{G-SO V_B G=Ahat[n]}
    In Example \ref{G=Ahat[n]}, $G\textrm{-}\mathrm{SO}^+_{\mathrm{Hdg}}(V_A)$ is the subgroup of $\mathrm{SO}^+_{\mathrm{Hdg}}(V_A)$ consisting of all elements of $\mathrm{SO}^+_{\mathrm{Hdg}}(V_A)$, which map $n\mathrm{H}^1(A,\mathbb{Z})\oplus\frac{1}{n}\mathrm{H}^1(\widehat{A},\mathbb{Z})$ to itself. Let $\Phi_{\mathcal{P}}:\mathbf{D}^ b(A)\to\mathbf{D}^b(\widehat{A})$ be the equivalence with  the Fourier-Mukai kernel being the Poincar\'e line bundle $\mathcal{P}$. We get an isometry $$\phi_{\mathcal{P}}:=\rho_{A,\widehat{A}}(\Phi_{\mathcal{P}})=\begin{pmatrix}
        0&-1\\
        1&0
    \end{pmatrix}:V_A=\mathrm{H}^1(A,\mathbb{Z})\oplus\mathrm{H}^1(\widehat{A},\mathbb{Z})\to V_{\widehat{A}}=\mathrm{H}^1(\widehat{A},\mathbb{Z})\oplus\mathrm{H}^1(A,\mathbb{Z}),$$ by \cite[Example 9.38 v)]{Huybrechts:06}. Obviously, $\phi_\mathcal{P}\not\in G\textrm{-}\mathrm{SO}^+_{\mathrm{Hdg}}(V_A,V_{A'})$. By Corollary \ref{Lift derived eq_A}, it suggests that $\Phi_{\mathcal{P}}$ does not lift to a $G$-functor between $\mathbf{D}^b_G(A)$ and $\mathbf{D}^b_G(\widehat{A})$.
\end{example}

\subsection{Relating the cohomological actions of $\lambda_q(f,\sigma)$ (and $\lambda_{q,q'}(f,\sigma)$) and $f$ on the spin representation.}

As the first step in proving Theorem \ref{Th_GrhoA}, we define the map $G\textrm{-}\rho_A:=\rho_B\circ\lambda_q$. We prove the inclusion $$G\textrm{-}\rho_A(G\textrm{-}\mathrm{Aut}(\mathbf{D}^b_G(A)))\subset G\textrm{-}\mathrm{SO}^+_{\mathrm{Hdg}}(V_B)$$ in this subsection. Similarly, we define the map $G\textrm{-}\rho_{A,A'}:=\rho_{B,B'}\circ\lambda_{q,q'}$ in Theorem \ref{Th_GrhoAA'} and prove the inclusion $$G\textrm{-}\rho_{A,A'}(G\textrm{-}\mathrm{Eq}(\mathbf{D}^b_G(A),\mathbf{D}^b_G(A')))\subset G\textrm{-}\mathrm{SO}^+_{\mathrm{Hdg}}(V_B,V_{B'})$$ in this subsection.

Let $q:B\to A, \widehat{q}, q_*,\widehat{q}^*$ be as above with a similar notation for $q',\widehat{q'}, (q')_*,\widehat{q'}^*$. We also denote the integral cohomology homomorphisms and the induced maps of the derived categories by the same symbols. Let $\mathcal{P}_A$ be the Poincar\'e line bundle over $A\times\widehat{A}$ and $\mathcal{P}_B, \mathcal{P}_{A'}, \mathcal{P}_{B'}$ be its analogues. Let $\phi_{\mathcal{P}_A}:\mathrm{H}^*(A,\mathbb{Z})\to\mathrm{H}^*(\widehat{A},\mathbb{Z})$ be the isomorphism in cohomology associated with the derived equivalence $$\Phi_{\mathcal{P}_A}:\mathbf{D}^b(A)\to\mathbf{D}^b(\widehat{A}).$$ Denote by $\phi_{\mathcal{P}_B}, \phi_{\mathcal{P}_{A'}}, \phi_{\mathcal{P}_{B'}}$ its analogues. Then, we have the following relation between $\Phi_{\mathcal{P}_A}$ and $\Phi_{\mathcal{P}_B}$. (Analogue for $\Phi_{\mathcal{P}_{A'}}$ and $\Phi_{\mathcal{P}_{B'}}$.)

\begin{proposition}\label{Phi PA vs Phi VB}
    Let the notation be as above. Then $\Phi_{\mathcal{P}_A}\circ q_*\circ\Phi_{\mathcal{P}_B}^{-1}\cong\widehat{q}^*$. That is, we have the following commutative diagram.
    \begin{align*}
        \xymatrix{
        \mathbf{D}^b(B)\ar[r]^{q_*}\ar[d]^{\simeq}_{\Phi_{\mathcal{P}_B}}&\mathbf{D}^b(A)\ar[d]^{\Phi_{\mathcal{P_A}}}_{\simeq}\\
        \mathbf{D}^b(\widehat{B})\ar[r]_{{\widehat{q}}^*}&\mathbf{D}^b(\widehat{A})
        }
    \end{align*}
\end{proposition}

\begin{proof}
    By \cite[Exercise 5.12 $ii)$ and $iii)$]{Huybrechts:06}, the Fourier-Mukai kernel of $\Phi_{\mathcal{P}_A}\circ q_*$ and $\widehat{q}^*\circ\Phi_{\mathcal{P}_B}$ is $(q\times\mathrm{id}_{\widehat{A}})^*\mathcal{P}_A\in\mathbf{D}^b(B\times\widehat{A})$ and $(\mathrm{id}_{B}\times\widehat{q})^*\mathcal{P}_B\in\mathbf{D}^b(B\times\widehat{A})$, respectively. It suffices to prove $(q\times\mathrm{id}_{\widehat{A}})^*\mathcal{P}_A\cong(\mathrm{id}_{B}\times\widehat{q})^*\mathcal{P}_B$.
    
    Using the defining property of the Poincar\'e line bundles \cite[Section 5 in Chapter 2]{Birkenhake Lange:04}, we have $((q\times\mathrm{id}_{\widehat{A}})^*\mathcal{P}_A)|_{B\times\{[L]\}}\cong((\mathrm{id}_{B}\times\widehat{q})^*\mathcal{P}_B)|_{B\times\{[L]\}}$ for all $L\in\mathrm{Pic}^0(A)$, where $[L]\in\widehat{A}$ denotes the corresponding point. We may also find that both $(q\times\mathrm{id}_{\widehat{A}})^*\mathcal{P}_A$ and $(\mathrm{id}_{B}\times\widehat{q})^*\mathcal{P}_B$ restrict to the trivial line bundle over $\{0\}\times \widehat{A}$. Then the isomorphism of two line bundles can be proved by using the Seesaw Principle \cite[Corollary A.9]{Birkenhake Lange:04}.
\end{proof}

We have similar results in the level of cohomology. Since $\mathrm{H}^1(\widehat{A},\mathbb{Z})$ is naturally isomorphic to $\mathrm{H}^1({A},\mathbb{Z})^*$, $\mathrm{H}^*(\widehat{A},\mathbb{Z})$ is naturally the dual of $\mathrm{H}^*({A},\mathbb{Z})$. Under this isomorphism, the isomorphism $\phi_{\mathcal{P}_A}$ corresponds to Poincar\'e duality, up to a sign change in each degree as in \cite[Lemma 9.23 and Corollary 9.24]{Huybrechts:06}. Similar for $A'$. Hence, we get $\mathrm{deg}(q)\cdot\phi_{\mathcal{P}_A}=\widehat{q}^*\circ\phi_{\mathcal{P}_B}\circ q^*, \mathrm{deg}(q')\cdot\phi_{\mathcal{P}_{A'}}=\widehat{q'}^*\circ\phi_{\mathcal{P}_{B'}}\circ (q')^*$. That is to say, the diagrams below commute.
\begin{align*}
    \xymatrix{
    \mathrm{H}^*(B,\mathbb{Z})\ar[d]^{\phi_{\mathcal{P}_B}}&\mathrm{H}^*(A,\mathbb{Z})\ar[d]^{\mathrm{deg}(q)\cdot\phi_{\mathcal{P}_A}}\ar[l]_{q^*}&\mathrm{H}^*(B',\mathbb{Z})\ar[d]^{\phi_{\mathcal{P}_{B'}}}&\mathrm{H}^*(A',\mathbb{Z})\ar[d]^{\mathrm{deg}(q')\cdot\phi_{\mathcal{P}_{A'}}}\ar[l]_{(q')^*}\\
    \mathrm{H}^*(\widehat{B},\mathbb{Z})\ar[r]^{\widehat{q}^*}&\mathrm{H}^*(\widehat{A},\mathbb{Z})&\mathrm{H}^*(\widehat{B'},\mathbb{Z})\ar[r]^{\widehat{q'}^*}&\mathrm{H}^*(\widehat{A'},\mathbb{Z})\\
    }
\end{align*}

Indeed, since $\phi_{\mathcal{P}_A}=(-1)^{\frac{m(m+1)}{2}}\mathrm{PD}_{A,m}:\mathrm{H}^m(A,\mathbb{Q})\xrightarrow{\simeq}\mathrm{H}^{4-m}(\widehat{A},\mathbb{Q})$ (and the analog result for $\phi_{\mathcal{P}_B}$) by \cite[Lemma 9.23 and Corollary 9.24]{Huybrechts:06}, we have the following computation for arbitrary $\alpha\in\mathrm{H}^m(A,\mathbb{Z})$.
\begin{align*}
    &\widehat{q}^*\circ\phi_{\mathcal{P}_B}\circ q^*(\alpha)=(-1)^{\frac{m(m+1)}{2}}\widehat{q}^*(\mathrm{PD}_B(q^*(\alpha)))=(-1)^{\frac{m(m+1)}{2}}\widehat{q}^*(\int_{B}q^*(\alpha)\cup(-))\\
    =&(-1)^{\frac{m(m+1)}{2}}\int_Bq^*\alpha\cup q^*(-)=(-1)^{\frac{m(m+1)}{2}}\mathrm{deg}(q)\int_A\alpha\cup(-)\\
    =&(-1)^{\frac{m(m+1)}{2}}\mathrm{deg}(q)\mathrm{PD}_A(\alpha)=\mathrm{deg}(q)\phi_{\mathcal{P}_A}(\alpha)
\end{align*}
It gives $\mathrm{deg}(q)\cdot\phi_{\mathcal{P}_A}=\widehat{q}^*\circ\phi_{\mathcal{P}_B}\circ q^*$. Analogue for $\mathrm{deg}(q')\cdot\phi_{\mathcal{P}_{A'}}=\widehat{q'}^*\circ\phi_{\mathcal{P}_{B'}}\circ (q')^*$.

Keep the notation of (\ref{Flambda_EqAA'}). Let $(f,\sigma)$ be an element of $G\textrm{-}\mathrm{Eq}(\mathbf{D}^b_G(A),\mathbf{D}^b_G(A'))$. Let $f^{\mathrm{H}}:\mathrm{H}^*(A,\mathbb{Z})\to\mathrm{H}^*(A',\mathbb{Z})$ be the cohomological action of $f$ and denote 
by $\lambda_{q,q'}(f,\sigma)^{\mathrm{H}}:\mathrm{H}^*(B,\mathbb{Z})\to\mathrm{H}^*(B',\mathbb{Z})$ that of $\lambda_{q,q'}(f,\sigma)$.

\begin{lemma}\label{lambda preserve integral structure_AA'}
    The equality $\lambda_{q,q'}(f,\sigma)^\mathrm{H}\circ q^*=(q')^*\circ f^\mathrm{H}$ holds in $\mathrm{Hom}(\mathrm{H}^*(A,\mathbb{Q}), \mathrm{H}^*(B',\mathbb{Q}))$. In particular, $\lambda_{q,q'}(f,\sigma)^\mathrm{H}$ is an element of the image of $\mathrm{Spin}(V_B,V_{B'})$ in $\mathrm{GL}(\mathrm{H}^*(B,\mathbb{Z}), \mathrm{H}^*(B',\mathbb{Z}))$, which maps the lattice $q^*\mathrm{H}^1(A,\mathbb{Z})$ of $V_{B,\mathbb{Q}}$ to the lattice $(q')^*\mathrm{H}^1(A',\mathbb{Z})$ of $V_{B',\mathbb{Q}}$.
\end{lemma}

\begin{proof}
    Consider the following diagram with left and right squares being commutative.
    \begin{align*}
    \xymatrix{
    \mathrm{H}^*(\widehat{A},\mathbb{Q})&&\mathrm{H}^*(A,\mathbb{Q})\ar[d]^{q^*}\ar[ll]_{\mathrm{deg}(q)\cdot\phi_{\mathcal{P}_A}}\ar[rr]^{f^\mathrm{H}}&&\mathrm{H}^*(A',\mathbb{Q})\ar[d]^{(q')^*}\ar[rr]^{\mathrm{deg}(q')\cdot\phi_{\mathcal{P}_{A'}}}&&\mathrm{H}^*(\widehat{A'},\mathbb{Q})\\
    \mathrm{H}^*(\widehat{B},\mathbb{Q})\ar[u]^{\widehat{q}^*}&&\mathrm{H}^*({B},\mathbb{Q})\ar[ll]_{\phi_{\mathcal{P}_B}}\ar[rr]^{\lambda_{q,q'}(f,\sigma)^\mathrm{H}}&&\mathrm{H}^*({B'},\mathbb{Q})\ar[rr]^{\phi_{\mathcal{P}_{B'}}}&&\mathrm{H}^*(\widehat{B'},\mathbb{Q})\ar[u]_{\widehat{q'}^*}
    }
\end{align*}
    It suffices to prove the commutativity of the whole diagram.
    
    Since $\mathrm{deg}(q')=\#G=\mathrm{deg}(q)$. It is equivalent to saying, 
    \begin{align}\label{outer diagram commute for lambda_AA'}
        (\phi_{\mathcal{P}_{A'}}^{-1}\circ\widehat{q'}^*\circ\phi_{\mathcal{P}_{B'}})\circ\lambda_{q,q'}(f,\sigma)^\mathrm{H}=f^H\circ(\phi_{\mathcal{P}_A}^{-1}\circ\widehat{q}^*\circ\phi_{\mathcal{P}_B}).
    \end{align}
    The equivalence $\Phi_{\mathcal{P}_A}$ conjugates the $G$-action on $\mathbf{D}^b(A)$ via the subgroup $\mathrm{ker}(\widehat{q})$ of $\widehat{A}$ to the action of $G$ on $\mathbf{D}^b(\widehat{A})$ via translation automorphisms. An object $\widetilde{F}\in\mathbf{D}^b(B)$ corresponds to $(F,\phi_F)\in\mathbf{D}^b_G(A)$ by the equivalence $$\psi_q:\mathbf{D}^b(B)\simeq\mathbf{D}^b_G(A)$$
    in passage \ref{Db(B) vs DbG(A)}. Such an object $(F,\phi_F)$ corresponds to a $G$-equivariant object $G\textrm{-}\Phi_{\mathcal{P}_A}(F,\phi_F)=(\widehat{q}^*(E),\phi_{\widehat{q}^*(E)})$, the pullback of some element $E\in\mathbf{D}^b(\widehat{B})$. So we get $\Phi_{\mathcal{P}_B}(\widetilde{F})=E$. Here, we use the induced $G$-functor $G\textrm{-}\Phi_{\mathcal{P}_A}$ and the result $\Xi_q\cong G\textrm{-}\Phi_{\mathcal{P}_A}\circ\psi_q\circ\Phi_{\mathcal{P}_B}^{-1}$ in Lemma \ref{G-Phi P_A equation}, where the derived equivalence $$\Xi_q:\mathbf{D}^b(\widehat{B})\to\mathbf{D}^b_G(\widehat{A})$$ is in passage \ref{Db(B) vs DbG(A)}.
    
    The equality $\widehat{q}^*(\mathrm{ch}(E))=\mathrm{ch}(\Phi_{\mathcal{P}_A}(F))$ yields $$(\phi_{\mathcal{P}_A}^{-1}\circ\widehat{q}^*\circ\phi_{\mathcal{P}_B})\mathrm{ch}(\widetilde{F})=\mathrm{ch}(F).$$
    
    Equation (\ref{outer diagram commute for lambda_AA'}) follows from the equation above applied to $\lambda_{q,q'}(f,\sigma)(\widetilde{F})$ and $f(F)$ on two sides. To be explicit, we get the equalities:
    \begin{align*}
        (\phi_{\mathcal{P}_{A'}}^{-1}\circ\widehat{q'}^*\circ\phi_{\mathcal{P}_{B'}})\circ\lambda_{q,q'}(f,\sigma)^\mathrm{H}\mathrm{ch}(\widetilde{F})&=(\phi_{\mathcal{P}_{A'}}^{-1}\circ\widehat{q'}^*\circ\phi_{\mathcal{P}_{B'}})\mathrm{ch}(\lambda_{q,q'}(f,\sigma)\widetilde{F})=\mathrm{ch}(f(F))\\
        f^\mathrm{H}(\phi_{\mathcal{P}_A}^{-1}\circ\widehat{q}^*\circ\phi_{\mathcal{P}_B})\mathrm{ch}(\widetilde{F})&=f^\mathrm{H}\mathrm{ch}(F)=\mathrm{ch}(f(F)).
    \end{align*}
    Then, we get equation (\ref{outer diagram commute for lambda_AA'}) by comparing the left-hand sides.
\end{proof}

Now, we state the result used in the proof above.

\begin{lemma}\label{G-Phi P_A equation}
    Let the notation be as above. Then $\Xi_q= G\textrm{-}\Phi_{\mathcal{P}_A}\circ\psi_q\circ\Phi_{\mathcal{P}_B}^{-1}$, where $G\textrm{-}\Phi_{\mathcal{P}_A}:\mathbf{D}^b_G(A)\to\mathbf{D}^b_G(\widehat{A})$ is the $G$-functor naturally induced by the functor $\Phi_{\mathcal{P}_A}$. That is, we have the following commutative diagram for the derived equivalences.
    \begin{align*}
        \xymatrix{
        \mathbf{D}^b(\widehat{B})\ar[r]^{\Phi_{\mathcal{P}_B}^{-1}}_{\simeq}\ar[d]_{\Xi_q}^{\simeq}&\mathbf{D}^b(B)\ar[d]^{\psi_q}_{\simeq}\\
        \mathbf{D}^b_G(\widehat{A})&\mathbf{D}^b_G(A)\ar[l]^{G\textrm{-}\Phi_{\mathcal{P}_A}}_{\simeq}
        }
    \end{align*}
\end{lemma}

\begin{proof}
    On the level of objects, we start from an arbitrary object $\widetilde{F}$ in $\mathbf{D}^b(B)$. Then $\psi_q(\widetilde{F})=(F,\phi_F)\in\mathbf{D}^b_G(A)$, where $$F=q_*\widetilde{F},\phi_{F,g}=-\otimes\mathcal{L}_g:F\xrightarrow{\cong}F\otimes\mathcal{L}_g=\rho_g(F)$$ for $g\in G$ with the corresponding line bundle being $\mathcal{L}_g\in\mathrm{Pic}^0(A)$. Then, we define the induced $G$-functor $G\textrm{-}\Phi_{\mathcal{P}_A}$ such that $$G\textrm{-}\Phi_{\mathcal{P}_A}(F,\phi_F)=(\Phi_{\mathcal{P}_A}(F),\Phi_{\mathcal{P}_A}(\phi_F)),$$
    where $$(\Phi_{\mathcal{P}_A}(\phi_F))_g=\Phi_{\mathcal{P}_A}(\phi_{F,g})=(t_g)_*:\Phi_{\mathcal{P}_A}(F)\xrightarrow{\cong}\Phi_{\mathcal{P}_A}(F\otimes\mathcal{L}_g)=(t_g)_*(\Phi_{\mathcal{P}_A}(F))=\widehat{\rho}_g(\Phi_{\mathcal{P}_A}(F))$$ for $g\in G$.

    On the other hand, the natural derived equivalence $\Xi_q:\mathbf{D}^b(\widehat{B})\to\mathbf{D}^b_G(\widehat{A})$ in passage \ref{Db(B) vs DbG(A)} maps the object $E:=\Phi_{\mathcal{P}_B}(\widetilde{F})$ in $\mathbf{D}^b(\widehat{B})$ to $\Xi_q(E)=(\widehat{q}^*E,\phi_{\widehat{q}^*E})$, where $\phi_{\widehat{q}*E,g}:=(t_g)_*:\widehat{q}^*E\xrightarrow{\cong}(t_g)_*(\widehat{q}^*E)=\widehat{\rho}_g(\widehat{q}^*E)$ for $g\in G$. 
    
    Since $\Phi_{\mathcal{P}_A}(F)=\Phi_{\mathcal{P}_A}(q_*\widetilde{F})=\Phi_{\mathcal{P}_A}\circ q_*\circ\Phi_{\mathcal{P}_B}^{-1}(E)\cong \widehat{q}^*E$ by Proposition \ref{Phi PA vs Phi VB}, we have $(\Phi_{\mathcal{P}_A}(F),\Phi_{\mathcal{P}_A}(\phi_F))\cong(\widehat{q}^*E,\phi_{\widehat{q}^*E})$ in $\mathbf{D}^b_G(\widehat{A})$. It gives a proof on the level of objects. To be explicit, we have 
    \begin{align*}
        \xymatrix{
        E=\Phi_{\mathcal{P}_B}(\widetilde{F})\ar@{|->}[d]&&\widetilde{F}\ar@{|->}[ll]\ar@{|->}[d]\\
        (\widehat{q}^*E,\phi_{\widehat{q}^*E})\ar@{<->}[d]^{\cong}&&(q_*\widetilde{F}=F,\phi_F)\ar@{|->}[lld]\\
        (\Phi_{\mathcal{P}_A}(F),\Phi_{\mathcal{P}_A}(\phi_F))&&
        }
    \end{align*} for the derived equivalence in the diagram in Lemma \ref{G-Phi P_A equation}.

    Using similar technology, we can prove on the level of cohomology. Briefly speaking, an arbitrary morphism $m:\widetilde{F}\to\widetilde{F}'$ in $\mathbf{D}^b(B)$ maps as in the following for the derived equivalence in the diagram in Lemma \ref{G-Phi P_A equation}.
    \begin{align*}
        \xymatrix{
        (\Phi_{\mathcal{P}_B}(m):E\to E')\ar@{|->}[d]&&(m:\widetilde{F}\to\widetilde{F'})\ar@{|->}[ll]\ar@{|->}[d]\\
        (\widehat{q}^*(\Phi_{\mathcal{P}_B}(m)):(\widehat{q}^*E,\phi_{\widehat{q}^*E})\to(\widehat{q}^*E',\phi_{\widehat{q}^*E'}))\ar@{<->}^{\cong}[d]&&(q_*m:(F,\phi_F)\to(F',\phi_{F'}))\ar@{|->}[lld]\\
        (\Phi_{\mathcal{P}_A}(q_*(m)):(\Phi_{\mathcal{P}_A}(F),\Phi_{\mathcal{P}_A}(\phi_F))\to(\Phi_{\mathcal{P}_A}(F'),\Phi_{\mathcal{P}_A}(\phi_F')))&&
        }
    \end{align*}
    Here, the morphism $q_*m:(F,\phi_F)\to(F',\phi_{F'})$ in $\mathbf{D}^b_G(A)$ is a morphism $q_*m$ from $F$ to $F'$ in $\mathbf{D}^b(A)$ such that the diagram \begin{align*}
        \xymatrix{
        F\ar[rr]^{q_*m}\ar[d]_{\phi_{F,g}=-\otimes\mathcal{L}_g}&&F'\ar[d]^{\phi_{F',g}=-\otimes\mathcal{L}_g}\\
        \rho_g(F)=F\otimes\mathcal{L}_g\ar[rr]^{\rho_g(q_*(m))}&&\rho_g(F')=F'\otimes\mathcal{L}_g
        }
    \end{align*}
commutes for every $g\in G$. The morphism $\widehat{q}^*(\Phi_{\mathcal{P}_B}(m)):(\widehat{q}^*E,\phi_{\widehat{q}^*E})\to(\widehat{q}^*E',\phi_{\widehat{q}^*E'})$ in $\mathbf{D}^b_G(\widehat{A})$ is a morphism $\widehat{q}^*(\Phi_{\mathcal{P}_B}(m)):\widehat{q}^*E\to\widehat{q}^*E'$ in $\mathbf{D}^b(\widehat{A})$ such that the diagram \begin{align*}
        \xymatrix{
        \widehat{q}^*E\ar[rr]^{\widehat{q}^*(\Phi_{\mathcal{P}_B}(m))}\ar[d]_{\phi_{\widehat{q}^*E,g}=(t_g)_*}&&\widehat{q}^*E'\ar[d]^{\phi_{\widehat{q}^*E',g}=(t_g)_*}\\
        \widehat{\rho}_g(\widehat{q}^*E)=(t_g)_*(\widehat{q}^*E)\ar[rr]^{\widehat{\rho}_g\circ\widehat{q}^*\circ\Phi_{\mathcal{P}_B}(m)}&&\widehat{\rho}_g(\widehat{q}^*E')=(t_g)_*(\widehat{q}^*E')
        }
    \end{align*}
commutes for every $g\in G$. The morphism $$\Phi_{\mathcal{P}_A}(q_*(m)):(\Phi_{\mathcal{P}_A}(F),\Phi_{\mathcal{P}_A}(\phi_F))\to(\Phi_{\mathcal{P}_A}(F'),\Phi_{\mathcal{P}_A}(\phi_F'))$$ in $\mathbf{D}^b_G(\widehat{A})$ is a morphism $\Phi_{\mathcal{P}_A}(q_*(m)):\Phi_{\mathcal{P}_A}(F)\to\Phi_{\mathcal{P}_A}(F')$ in $\mathbf{D}^b(\widehat{A})$ such that the diagram \begin{align*}
        \xymatrix{
        \Phi_{\mathcal{P}_A}(F)\ar[rr]^{\Phi_{\mathcal{P}_A}(q_*(m))}\ar[d]_{(\Phi_{\mathcal{P}_A}(\phi_F))_g=(t_g)_*}&&\Phi_{\mathcal{P}_A}(F')\ar[d]^{(\Phi_{\mathcal{P}_A}(\phi_{F'}))_g=(t_g)_*}\\
        \widehat{\rho}_g(\Phi_{\mathcal{P}_A}(F))=(t_g)_*(\Phi_{\mathcal{P}_A}(F))\ar[rr]^{\Phi_{\mathcal{P}_A}\circ\rho_g\circ q_*(m)}&&\widehat{\rho}_g(\Phi_{\mathcal{P}_A}(F'))=(t_g)_*(\Phi_{\mathcal{P}_A}(F'))
        }
    \end{align*}
commutes for every $g\in G$. 
\end{proof}

We have a version of the above lemma for derived autoequivalences.
\begin{lemma}\label{lambda preserve integral structure_A}
    The equality $\lambda_{q}(f,\sigma)^\mathrm{H}\circ q^*=q^*\circ f^\mathrm{H}$ holds in $\mathrm{Hom}(\mathrm{H}^*(A,\mathbb{Q}), \mathrm{H}^*(B,\mathbb{Q}))$. In particular, $\lambda_{q}(f,\sigma)^\mathrm{H}$ is an element of the image of $\mathrm{Spin}(V_B)$ in $\mathrm{GL}(\mathrm{H}^*(B,\mathbb{Z}))$, which keeps the lattice $q^*\mathrm{H}^1(A,\mathbb{Z})$ of $V_{B,\mathbb{Q}}$ invariant.
\end{lemma}

An application of Lemma \ref{lambda preserve integral structure_A} and Lemma \ref{lambda preserve integral structure_AA'} is the direct proof of the desired inclusions.

\begin{proposition}\label{GrhoA GrhoAA' inclusions}
    Let notation be as in Theorem \ref{Th_GrhoA} and Theorem \ref{Th_GrhoAA'}. We have inclusions 
    \begin{align*}
        G\textrm{-}\rho_A(G\textrm{-}\mathrm{Aut}(\mathbf{D}^b_G(A)))&\subset G\textrm{-}\mathrm{SO}^+_{\mathrm{Hdg}}(V_B)\\
        G\textrm{-}\rho_{A,A'}(G\textrm{-}\mathrm{Eq}(\mathbf{D}^b_G(A),\mathbf{D}^b_G(A')))&\subset G\textrm{-}\mathrm{SO}^+_{\mathrm{Hdg}}(V_B,V_{B'}).
    \end{align*}
\end{proposition}

This proposition proves the well-defined property of the maps $G\textrm{-}\rho_A, G\textrm{-}\rho_{A,A'}$ and the commutativity of the squares in (\ref{diagram_GrhoA}) (\ref{diagram_GrhoAA'}) related to $G\textrm{-}\rho_A, G\textrm{-}\rho_{A,A'}$. 

\begin{proof}[Proof of Proposition \ref{GrhoA GrhoAA' inclusions} using Lemma \ref{lambda preserve integral structure_A} and Lemma \ref{lambda preserve integral structure_AA'}]
    For every $G$-functor $(f,\sigma)\in G\textrm{-}\mathrm{Eq}(\mathbf{D}^b_G(A),\mathbf{D}^b_G(A'))$, the cohomology map $\lambda_{q,q'}(f,\sigma)^\mathrm{H}$ maps the lattice $q^*\mathrm{H}^1(A,\mathbb{Z})$ of $V_{B,\mathbb{Q}}$ to the lattice $(q')^*\mathrm{H}^1(A',\mathbb{Z})$ of $V_{B',\mathbb{Q}}$ by Lemma \ref{lambda preserve integral structure_AA'}. The map also maps the lattice $\widehat{q}^*\mathrm{H}^1(\widehat{B},\mathbb{Z})$ of $V_{B,\mathbb{Q}}$ to the lattice $\widehat{q'}^*\mathrm{H}^1(\widehat{B'},\mathbb{Z})$ of $V_{B',\mathbb{Q}}$ using the commutative diagram in the proof of Lemma \ref{lambda preserve integral structure_AA'}. Recall the natural inclusion $\mathrm{H}^1(\widehat{A},\mathbb{Z})\subset\mathrm{H}^1(\widehat{B},\mathbb{Q})$ fits $\widehat{q}^*\mathrm{H}^1(\widehat{B},\mathbb{Z})\subset\mathrm{H}^1(\widehat{A},\mathbb{Z})\subset\mathrm{H}^1(\widehat{A},\mathbb{Q})\xrightarrow[\simeq]{(\widehat{q}^*)^{-1}}\mathrm{H}^1(\widehat{B},\mathbb{Q})$. (Analogue for $A'$.) Then the cohomology map $\lambda_{q,q'}(f,\sigma)^\mathrm{H}$ maps the lattice $\mathrm{H}^1(\widehat{A},\mathbb{Z})$ of $V_{B,\mathbb{Q}}$ to the lattice $\mathrm{H}^1(\widehat{A'},\mathbb{Z})$ of $V_{B',\mathbb{Q}}$. By Remarks \ref{two Orlov's representation remark} and \ref{Spin vs Orlov representation}, we get $$G\textrm{-}\rho_{A,A'}(f,\sigma)=\lambda_{q,q'}(f,\sigma)^\mathrm{H}|_{V_B, V_{B'}}\in G\textit{-}\mathrm{SO}^+_{\mathrm{Hdg}}(V_B,V_{B'}).$$ 
    It gives the inclusion $G\textrm{-}\rho_{A,A'}(G\textrm{-}\mathrm{Eq}(\mathbf{D}^b_G(A),\mathbf{D}^b_G(A')))\subset G\textrm{-}\mathrm{SO}^+_{\mathrm{Hdg}}(V_B,V_{B'})$. 
    
    The derived autoequivalence case follows by Lemma \ref{lambda preserve integral structure_AA'} in the same way.
\end{proof}

Another application of Lemma \ref{lambda preserve integral structure_A} and Lemma \ref{lambda preserve integral structure_AA'} is to give an intuitive thought of Remark \ref{lambda_q explain B} and Remark \ref{AA'notation}, which provide an alternative way to explain $\lambda_q$ and $\lambda_{q,q'}$, respectively.

In fact, we have an alternative proof of Proposition \ref{GrhoA GrhoAA' inclusions} using the definition of $\gamma_{A,A'}, \gamma_{B,B'}$ in Orlov's Theorem \ref{Orlov fundamental th}.

\begin{proof}[Proof of Proposition \ref{GrhoA GrhoAA' inclusions}]
    ~\\
    For derived equivalence case, consider arbitrary $(\overline{f}=\Phi_{\overline{\mathcal{E}}},\sigma)\in G\textrm{-}\mathrm{Eq}(\mathbf{D}^b_G(A),\mathbf{D}^b_G(A'))$, we obtain the following via the notation in (\ref{Flambda_EqAA'}).
    \begin{align*}
    \xymatrix{
     G\textrm{-}\mathrm{Eq}(\mathbf{D}^b(B),\mathbf{D}^b(B'))&G\textrm{-}\mathrm{Eq}(\mathbf{D}^b_G(A),\mathbf{D}^b_G(A'))\ar[l]_{\lambda_{q,q'}}\ar[r]^(0.55){F_{q,q'}}&\mathrm{Eq}(\mathbf{D}^b(A),\mathbf{D}^b(A))\\
     f&(\overline{f},\sigma)\ar@{|->}[l]\ar@{|->}[r]&\overline{f}
    }
    \end{align*} 
    As in Remark \ref{two Orlov's representation remark}, we get $$(\gamma_{A,A'}(\overline{f}))_*|_{V_A, V_{A'}}=\rho_{A,A'}(\overline{f})\in\mathrm{SO}^+_{\mathrm{Hdg}}(V_A,V_{A'}), (\gamma_{B,B'}({f}))_*|_{V_B,V_{B'}}=\rho_{B,B'}({f})\in\mathrm{SO}^+_{\mathrm{Hdg}}(V_B,V_{B'}).$$
    By the alternative way to explain $\lambda_{q,q'}$ in Remark \ref{AA'notation}, the derived equivalence $\overline{f}$ can lift to $f$ via
    \begin{align}\label{lift f bar to f}
        \xymatrix{
    \mathbf{D}^b_G(A)\ar@{-}[r]^{\simeq}&\mathbf{D}^b(B)\ar[rr]^{f=\Phi_{\mathcal{E}}}\ar[d]^{q_*}&\phantom{1}&\mathbf{D}^b(B')\ar[d]^{(q')_*}\ar@{-}[r]^{\simeq}&\mathbf{D}^b_G(A').\\
    &\mathbf{D}^b(A)\ar[rr]_{\overline{f}=\Phi_{\overline{\mathcal{E}}}}&\phantom{1}\ar@{=>}[u]^{\mathrm{lift}}&\mathbf{D}^b(A')}&
    \end{align} The derived equivalence $\overline{f}^{-1}$ can be lifted similarly.
    
    By definition, $G\textrm{-}\rho_{A,A'}:=\rho_{B,B'}\circ\lambda_{q,q'}$, it suffices to prove that $$G\textrm{-}\rho_{A,A'}((\overline{f},\sigma))=(\gamma_{B,B'}({f}))_*|_{V_B,V_{B'}}\in G\textrm{-}\mathrm{SO}^+_{\mathrm{Hdg}}(V_B,V_{B'}).$$ 
    Obviously, $(\gamma_{B,B'}({f}))_*|_{V_B,V_{B'}}\in \mathrm{SO}^+_{\mathrm{Hdg}}(V_B,V_{B'})$ by applying Remark \ref{two Orlov's representation remark} to $B,B'$. It is enough to check that $(\gamma_{B,B'}({f}))_*$ maps $q^*\mathrm{H}^1(A,\mathbb{Z})\oplus\mathrm{H}^1(\widehat{A},\mathbb{Z})\subset V_{B,\mathbb{Q}}$ to $$(q')^*\mathrm{H}^1(A',\mathbb{Z})\oplus\mathrm{H}^1(\widehat{A'},\mathbb{Z})\subset V_{B',\mathbb{Q}}.$$
    
    \textit{Step 1:} Use the relation between ${f}=\Phi_{{\mathcal{E}}}$ and $\overline{f}=\Phi_{\overline{\mathcal{E}}}$ above to get relation between $F_{{\mathcal{E}}}$ and $F_{\overline{\mathcal{E}}}$.
    
    Here we use the notation in \cite[Definition 9.34]{Huybrechts:06} that $F_{{\mathcal{E}}}$ fits the commutative diagram
    \begin{align}\label{def of F_E}
        \xymatrix{
        \mathbf{D}^b(B\times\widehat{B})\ar[d]^{F_\mathcal{E}}\ar[r]^{\mathrm{id}\times\Phi_{{\mathcal{P}}_B}^{-1}}&\mathbf{D}^b(B\times{B})\ar[r]^{(\mu_{B})_*}&\mathbf{D}^b(B\times{B})\ar[d]^{\Phi_{\mathcal{E}}\times\Phi_{\mathcal{E}_R}}\\
        \mathbf{D}^b(B'\times\widehat{B'})&\mathbf{D}^b(B'\times{B'})\ar[l]^{\mathrm{id}\times\Phi_{{\mathcal{P}}_{B'}}}&\mathbf{D}^b(B'\times{B'}),\ar[l]^{(\mu_{B'})^*}
        }
    \end{align}
    where $\mu_B$ is the map $B\times B\to B\times B, (b_1,b_2)\mapsto(b_1+b_2,b_2)$ (similar to $\mu_{B'}$), and $\Phi_{{\mathcal{E}}_R}$ is the right adjoint of $\Phi_{{\mathcal{E}}}$ in the opposition direction. Similar for $F_{\overline{\mathcal{E}}}$. Using the lifting property of $\overline{f}=\Phi_{\overline{\mathcal{E}}}$ as in (\ref{lift f bar to f}) and of $\overline{f}^{-1}=(\Phi_{\overline{\mathcal{E}}})^{-1}$, we get the commutative diagram
    \begin{align*}
        \xymatrix{
        \mathbf{D}^b(B\times{B})\ar[rr]^{\Phi_{\mathcal{E}}\times\Phi_{\mathcal{E}_R}}\ar[d]^{q_*\times q_*}&&\mathbf{D}^b(B'\times{B'})\ar[d]_{(q')_*\times (q')_*}\\
        \mathbf{D}^b(A\times{A})\ar[rr]_{\Phi_{\overline{\mathcal{E}}}\times\Phi_{\overline{\mathcal{E}}_R}}&&\mathbf{D}^b(A'\times{A'}).
        }
    \end{align*}
    Combining with the definition of $F_{\mathcal{E}}$ and $F_{\overline{\mathcal{E}}}$, we have the following commutative diagram with the middle square as above.
    \begin{align*}
        \xymatrix{
        \mathbf{D}^b(B\times\widehat{B})\ar[r]^{\mathrm{id}\times\Phi_{{\mathcal{P}}_B}^{-1}}\ar[d]^{q_*\times(\Phi_{\mathcal{P}_A}\circ q_*\circ\Phi_{\mathcal{P}_B}^{-1})}&\mathbf{D}^b(B\times{B})\ar[r]^{(\mu_{B})_*}\ar[d]^{q_*\times q_*}&\mathbf{D}^b(B\times{B})\ar[r]^{\Phi_{\mathcal{E}}\times\Phi_{\mathcal{E}_R}}\ar[d]^{q_*\times q_*}&\mathbf{D}^b(B'\times{B'})\ar[r]^{(\mu_{B'})^*}\ar[d]_{(q')_*\times (q')_*}&\mathbf{D}^b(B'\times\widehat{B'})\ar[r]^{\mathrm{id}\times\Phi_{{\mathcal{P}}_{B'}}}\ar@<-2ex>[d]_{(q')_*\times (q')_*}&\mathbf{D}^b(B'\times\widehat{B'})\ar[d]_{q'_*\times(\Phi_{\mathcal{P}_{A'}}\circ (q')_*\circ\Phi_{\mathcal{P}_{B'}}^{-1})}\\
        \mathbf{D}^b(A\times\widehat{A})\ar[r]_{\mathrm{id}\times\Phi_{{\mathcal{P}}_A}^{-1}}&\mathbf{D}^b(A\times{A})\ar[r]_{(\mu_{A})_*}&\mathbf{D}^b(A\times{A})\ar[r]_{\Phi_{\overline{\mathcal{E}}}\times\Phi_{\overline{\mathcal{E}}_R}}&\mathbf{D}^b(A'\times{A'})\ar[r]_{(\mu_{A'})^*}&\mathbf{D}^b(A'\times\widehat{A'})\ar[r]_{\mathrm{id}\times\Phi_{{\mathcal{P}}_{A'}}}&\mathbf{D}^b(A'\times\widehat{A'})
        }
    \end{align*}
    Here, the squares related to $\mu$'s are commutative by direct computation. Obviously, the leftmost and rightmost ones are commutative. Consequently, we get the following relation between $F_{{\mathcal{E}}}$ and $F_{\overline{\mathcal{E}}}$. Here, we use Proposition \ref{Phi PA vs Phi VB} showing that $\Phi_{\mathcal{P}_A}\circ q_*\circ\Phi_{\mathcal{P}_B}^{-1}\cong\widehat{q}^*$ and its analogue for $q',\widehat{q'}$.
    \begin{align}\label{lift FEbar to FE}
        \xymatrix{
        \mathbf{D}^b(B\times\widehat{B})\ar[r]^{F_{\mathcal{E}}}\ar[d]_{q_*\times\widehat{q}^*}&\mathbf{D}^b(B'\times\widehat{B'})\ar[d]^{(q')_*\times\widehat{q'}^*}\\
        \mathbf{D}^b(A\times\widehat{A})\ar[r]_{F_{\overline{\mathcal{E}}}}&\mathbf{D}^b(A'\times\widehat{A'})
        }
    \end{align}
    That is to say, $F_{{\mathcal{E}}}$ and $F_{\overline{\mathcal{E}}}$ behave well with respect to $q_*\times\widehat{q}^*$ and $(q')_*\times\widehat{q'}^*$. (Remark \ref{lift FEbar to FE remark} contains further information.)
    
    \textit{Step 2:} Find the relation between $\gamma_{B,B'}(f)$ and $\gamma_{A,A'}(\overline{f})$.
    
    Using \cite[Proposition 9.39 and Corollary 9.47]{Huybrechts:06} and Theorem \ref{Orlov fundamental th}, we get 
    \begin{equation}\label{F_E vs gamma}
        \begin{aligned}
            F_{{\mathcal{E}}}&=(-\otimes {N}_{\mathcal{E}})\circ(\gamma_{B,B'}(f))_*, \exists{N}_{\mathcal{E}}\in\mathrm{Pic}(B'\times\widehat{B'})\\
        F_{\overline{\mathcal{E}}}&=(-\otimes {N}_{\overline{\mathcal{E}}})\circ(\gamma_{A,A'}(\overline{f}))_*, \exists{N}_{\overline{\mathcal{E}}}\in\mathrm{Pic}(A'\times\widehat{A'}).
        \end{aligned}
    \end{equation}
    Precisely, the derived equivalence $-\otimes {N}_{\mathcal{E}}$ is generated by shifts $[1]$, the derived equivalence by tensoring with line bundle $-\otimes{L},L\in\mathrm{Pic}^0(A')$, and the derived equivalence induced from translations $(t_{b'})_*$ for $b\in B'$. The pairs
    $$[1]\phantom{1}\mathrm{and}\phantom{1}[1];\phantom{1}-\otimes{L}\phantom{1}\mathrm{and}\phantom{1}-\otimes (q')^*L;\phantom{1}(t_{b'})_*\phantom{1}\mathrm{and}\phantom{1}(t_{b'})_*,$$
    behave well with respect to $(q')_*\times\widehat{q'}^*$. So, $-\otimes {N}_{\mathcal{E}}$ and $-\otimes {N}_{\overline{\mathcal{E}}}$ behave well with respect to $(q')_*\times\widehat{q'}^*$. Thus, we have a commutative diagram 
    \begin{align*}
        \xymatrix{
        \mathbf{D}^b(B\times\widehat{B})\ar[rr]^{(\gamma_{B,B'}(f))_*}\ar[d]_{q_*\times\widehat{q}^*}&&\mathbf{D}^b(B'\times\widehat{B'})\ar[d]^{(q')_*\times\widehat{q'}^*}\\
        \mathbf{D}^b(A\times\widehat{A})\ar[rr]_{(\gamma_{A,A'}(\overline{f}))_*}&&\mathbf{D}^b(A'\times\widehat{A'}).
        }
    \end{align*}
    
    Consider in the level of varieties, we have the following commutative diagram.

    \begin{align}\label{lift gamma BB' f}
        \xymatrix{
        B\times\widehat{B}\ar[rr]^{\gamma_{B,B'}(f)}&&B'\times\widehat{B'}\\
        B\times\widehat{A}\ar[u]^{\mathrm{id}_B\times\widehat{q}}\ar[d]_{q\times\mathrm{id}_{\widehat{A}}}&&B'\times\widehat{A'}\ar[u]_{\mathrm{id}_{B'}\times\widehat{q'}}\ar[d]^{q'\times\mathrm{id}_{\widehat{A'}}}\\
        A\times\widehat{A}\ar[rr]^{\gamma_{A,A'}(\overline{f})}&&A'\times\widehat{A'}
        }
    \end{align}

    \textit{Step 3:} Use the lifting criterion in cohomology to prove the result.
    
    The lifting criterion in the level of cohomology yields that the homomorphism $\gamma_{B,B'}(f)$ has a lift to a map $A\times\widehat{A}\to A'\times\widehat{A'}$ with respect to $q,\widehat{q}, q',\widehat{q'}$ as in (\ref{lift gamma BB' f}) if and only if $$(\gamma_{B,B'}(f))^*((q')^*\mathrm{H}^1(A',\mathbb{Z})\oplus((\widehat{q'})^*)^{-1}\mathrm{H}^1(\widehat{A'},\mathbb{Z}))=q^*\mathrm{H}^1(A,\mathbb{Z})\oplus(\widehat{q}^*)^{-1}\mathrm{H}^1(\widehat{A},\mathbb{Z}).$$
    Using the natural inclusions, it is equivalent to $$(\gamma_{B,B'}(f))^*((q')^*\mathrm{H}^1(A',\mathbb{Z})\oplus\mathrm{H}^1(\widehat{A'},\mathbb{Z}))=q^*\mathrm{H}^1(A,\mathbb{Z})\oplus\mathrm{H}^1(\widehat{A},\mathbb{Z}).$$
    By the diagram (\ref{lift gamma BB' f}), the homomorphism $\gamma_{B,B'}(f)$ lifts to $\gamma_{A,A'}(\overline{f})$. So the equation holds. Since $(\gamma_{B,B'}(f))^*|_{V_B, V_{B'}}\in\mathrm{SO}^+_{\mathrm{Hdg}}(V_{B'},V_B)$, we get
    $$(\gamma_{B,B'}(f))^*|_{V_B, V_{B'}}\in G\textrm{-}\mathrm{SO}^+_{\mathrm{Hdg}}(V_{B'},V_B).$$
    So we may claim $(\gamma_{B,B'}(f))_*|_{V_B, V_{B'}}\in G\textrm{-}\mathrm{SO}^+_{\mathrm{Hdg}}(V_{B},V_{B'})$.
    
    Since the derived equivalence case is set, the derived autoequivalence version follows.
\end{proof}

\begin{remark}\label{lift FEbar to FE remark}
    In the proof of Proposition \ref{GrhoA GrhoAA' inclusions}, we have a lift of $F_{\overline{\mathcal{E}}}$ to $F_\mathcal{E}$ as in (\ref{lift FEbar to FE}). This means that $F_{\overline{\mathcal{E}}}$ admits a set of $(G\times G)$-equivariance natural transformations $\overline{\sigma}$ and $F_{\mathcal{E}}$ is the derived equivalence corresponding to $(F_{\overline{\mathcal{E}}},\overline{\sigma})$ via the equivalence $\mathbf{D}^b_{G\times G}(A\times \widehat{A})\simeq\mathbf{D}^b(B\times\widehat{B})$ and its analogue for $A',B'$. Above, $G\times G$ is a subgroup of $\mathrm{Aut}(\mathbf{D}^b(A\times\widehat{A}))$, where $(L,M)\in\mathrm{ker}(\widehat{q})\times\mathrm{ker}(\widehat{q})=G\times G$ acts by $\pi_A^*L\otimes(\tau_{(0,[M])})_*(-)=(L\otimes-)\times((\tau_{[M]})_*(-)$. Here, $\tau_{(0,[M])}:A\times\widehat{A}\to A\times\widehat{A}$ and $\tau_{[M]}:\widehat{A}\to\widehat{A}$ are the translations by the points $(0,[M])$ and $[M]$, respectively. (The point $[M]\in\widehat{A}$ corresponds to $M\in\mathrm{Pic}^0(A)$.) By (\ref{def of F_E}), this is equivalent to the statement that the Fourier-Mukai kernel $\overline{\mathcal{E}}$ is $(G\times G)\times (G\times G)$-equivariant.%
    \footnote{Note, however, that its support is point-wise invariant with respect to a diagonal copy of $G\times G$.}%
    This is, a priori, a non-trivial statement as $\overline{\mathcal{E}}$ is supported as a line bundle $N_{\overline{\mathcal{E}}}$ over the graph of the isomorphism $\gamma_{A,A'}(\overline{f})$, where $c_1(N_{\overline{\mathcal{E}}})\not=0$, by recalling the proof of \cite[Proposition 9.39]{Huybrechts:06}. So, the translations of $N_{\overline{\mathcal{E}}}$ need not be isomorphic to $N_{\overline{\mathcal{E}}}$ even if they leave its support invariant (but not point-wise invariant).
\end{remark}

\begin{remark}\label{S1 S2 GrhoA GrhoAA' inclusions remark}
    We may review the reasoning in Step 1 and Step 2 of Proposition \ref{GrhoA GrhoAA' inclusions} to see that for $\lambda_{q,q'}(\overline{f},\sigma)=f$, the derived equivalence $\overline{f}\in\mathrm{Eq}(\mathbf{D}^b(A),\mathbf{D}^b(A'))$ can be lifted to $f\in\mathrm{Eq}(\mathbf{D}^b(A),\mathbf{D}^b(A'))$ as in (\ref{lift f bar to f}). That is,
    \begin{align}\label{lift f bar to f brief}
        \xymatrix{
    \mathbf{D}^b(B)\ar[rr]^{f}\ar[d]^{q_*}&\phantom{1}&\mathbf{D}^b(B')\ar[d]^{(q')_*}\\
    \mathbf{D}^b(A)\ar[rr]_{\overline{f}}&\phantom{1}\ar@{=>}[u]^{\mathrm{lift}}&\mathbf{D}^b(A').}
    \end{align}
    For $\overline{f}$ and $f$ related as in (\ref{lift f bar to f brief}), the diagram (\ref{lift gamma BB' f}) is commutative.
\end{remark}

\subsection{Surjectivity in Theorem \ref{Th_GrhoA} and Theorem \ref{Th_GrhoAA'}}
In this subsection, we state the surjectivity of $G\textrm{-}\rho_A$ in Theorem \ref{Th_GrhoA} and $G\textrm{-}\rho_{A,A'}$ in Theorem \ref{Th_GrhoAA'}.

\begin{proposition}\label{GrhoA GrhoAA' surjective}
    Let notation be as in Theorem \ref{Th_GrhoA} and Theorem \ref{Th_GrhoAA'}. The following two maps are surjective. 
    \begin{align*}
        G\textrm{-}\rho_A:G\textrm{-}\mathrm{Aut}(\mathbf{D}^b_G(A))&\to G\textrm{-}\mathrm{SO}^+_{\mathrm{Hdg}}(V_B)\\
        G\textrm{-}\rho_{A,A'}:G\textrm{-}\mathrm{Eq}(\mathbf{D}^b_G(A),\mathbf{D}^b_G(A'))&\to G\textrm{-}\mathrm{SO}^+_{\mathrm{Hdg}}(V_B,V_{B'}).
    \end{align*}
\end{proposition}

\begin{proof}
    For the map $G\textrm{-}\rho_{A,A'}$, we start from an arbitrary element of $G\textrm{-}\mathrm{SO}^+_{\mathrm{Hdg}}(V_B,V_{B'})$. It is also an element of $\mathrm{SO}^+_{\mathrm{Hdg}}(V_B,V_{B'})$. Such an element is of the form $\rho_{B,B'}(f)$ for some $f\in\mathrm{Eq}(\mathbf{D}^b(B),\mathbf{D}^b(B'))$, since $\rho_{B,B'}:\mathrm{Eq}(\mathbf{D}^b(B),\mathbf{D}^b(B'))\to \mathrm{SO}^+_{\mathrm{Hdg}}(V_B,V_{B'})$ is surjective. Suppose that we can find $(\overline{f},\sigma)\in G\textrm{-}\mathrm{Eq}(\mathbf{D}^b_G(A),\mathbf{D}^b_G(A'))$ such that $\lambda_{q,q'}(\overline{f},\sigma)=f$, we may get $G\textrm{-}\rho_{A,A'}(\overline{f},\sigma)=\rho_{B,B'}(f)$ by the definition of $G\textrm{-}\rho_{A,A'}$. So we obtain the surjectivity of $G\textrm{-}\rho_{A,A'}$.

    To find the desired $(\overline{f},\sigma)\in G\textrm{-}\mathrm{Eq}(\mathbf{D}^b_G(A),\mathbf{D}^b_G(A'))$, we need to find $\overline{f}$ such that $\overline{f}$ has the lifting $f$ as in (\ref{lift f bar to f brief}) and find a suitable set of natural transformations $\sigma=(\sigma_g)_{g\in G}$.

    \textit{Step 1:} Find some $\overline{f}_0$ using symplectic isomorphisms.

    By Remark \ref{S1 S2 GrhoA GrhoAA' inclusions remark}, we may consider the isomorphism $\overline{F}$ induced by $F:=\gamma_{B,B'}(f)$ by the following commutative diagram as the first step to find $\overline{f}$. \begin{align}\label{lift gamma BB' f Fbar}
        \xymatrix{
        B\times\widehat{B}\ar[rr]^{F:=\gamma_{B,B'}(f)}&&B'\times\widehat{B'}\\
        B\times\widehat{A}\ar[u]^{\mathrm{id}_B\times\widehat{q}}\ar[d]_{q\times\mathrm{id}_{\widehat{A}}}&&B'\times\widehat{A'}\ar[u]_{\mathrm{id}_{B'}\times\widehat{q'}}\ar[d]^{q'\times\mathrm{id}_{\widehat{A'}}}\\
        A\times\widehat{A}\ar[rr]^{\overline{F}}&&A'\times\widehat{A'}
        }
    \end{align}

    Indeed, such an isomorphism $\overline{F}$ can be obtained.
    
    Starting from $\rho_{B,B'}(f)=(\gamma_{B,B'}(f))_*|_{V_B,V_{B'}}\in G\textrm{-}\mathrm{SO}^+_{\mathrm{Hdg}}(V_B,V_{B'})$, we get $$(\gamma_{B,B'}(f))^*|_{V_B,V_B'}\in G\textrm{-}\mathrm{SO}^+_{\mathrm{Hdg}}(V_{B'},V_{B}).$$
    Using the lifting criterion in cohomology as in Step 3 of the proof for Proposition \ref{GrhoA GrhoAA' inclusions}, the desired $\overline{F}$ that lifts to $F$ can be achieved. As ${\gamma_{B,B'}(f)}\in\mathrm{Sp}(B,B')$, we obtain $\overline{F}\in\mathrm{Sp}(A,A')$ by definition of the symplectic map (\cite[Definition 4.4]{Magni:22}) and the commutative diagram (\ref{lift gamma BB' f Fbar}).

    Using the isomorphism $\overline{F}$, we may determine $\overline{f}\in\mathrm{Eq}(\mathbf{D}^b(A),\mathbf{D}^b(A'))$, up to $\mathrm{Alb}(A')\times \widehat{A'}\times\mathbb{Z}$, such that $\overline{F}=\gamma_{A,A'}(\overline{f})$ by Theorem \ref{Orlov fundamental th}. Say $\overline{f}_0\in\mathrm{Eq}(\mathbf{D}^b(A),\mathbf{D}^b(A'))$, such that $\overline{F}=\gamma_{A,A'}(\overline{f}_0)$. In this way, we form the diagram (\ref{lift gamma BB' f}). 
    
    The $\mathrm{Alb}(A')\times \widehat{A'}\times\mathbb{Z}$ part corresponds to the information of $-\otimes N_{\overline{\mathcal{E}}}$ in Step 2 of Proposition \ref{GrhoA GrhoAA' inclusions}. Actually, by \cite[Proposition 9.39]{Huybrechts:06}, the process to determine $\overline{f}$ via $\overline{F}$ involves a line bundle over the graph of $\overline{F}$ with non-zero first Chern class in general.

    \textit{Step 2:} Find a suitable set of natural transformations $\sigma=(\sigma_g)_{g\in G}$ such that $(\overline{f}_0,\sigma)\in G\textrm{-}\mathrm{Eq}(\mathbf{D}^b_G(A),\mathbf{D}^b_G(A'))$.

    Actually, it can be done using Lemma \ref{find sigma for overline f_0}, which gives a set of natural transformations. Note that the hypothesis of Lemma \ref{find sigma for overline f_0} is naturally satisfied since $\rho_{B,B'}(f)\in G\textit{-}\mathrm{SO}^+_{\mathrm{Hdg}}(V_B,V_{B'})$.

    \textit{Step 3:} Find a suitable $G$-functor that lifts to $f=F_{\mathcal{E}}$ to complete the proof.

    By Step 2, there exists a set of $G$-equivariance natural transformations $\sigma$ of $\overline{f}_0$. By the surjectivity result in Theorem \ref{Orlov fundamental th}, the graph of $\overline{F}$ is the support of the Fourier-Mukai kernel of $F_{\overline{f}_0}$ and the graph of $F=\gamma_{B,B'}(f)$ is the support of the Fourier-Mukai kernel of $F_{t\circ f}$, for a translate of $f$ by an element $t\in B\times\mathrm{Pic}^0(B)\times\mathbb{Z}$. So $(\overline{f}_0,\sigma)$ lifts to $t\circ f$ as in (\ref{lift f bar to f brief}) and $\lambda_{q,q'}(\overline{f}_0,\sigma)=t\circ f$. We get $$\lambda_{q,q'}(\overline{t}^{-1}\circ\overline{f}_0,\overline{t}^{-1}\circ\sigma)=f,$$
    where the translation $\overline{t}\in A\times\mathrm{Pic}^0(A)\times\mathbb{Z}$ lifts to $t$ by the contents about (\ref{F_E vs gamma}).
\end{proof}

We end this subsection by stating the following general lemma, which is used in the above reasoning.
\begin{lemma}\label{find sigma for overline f_0}
    Let $\Phi:\mathbf{D}^b(A)\to\mathbf{D}^b(A')$ be an equivalence of the derived categories of two abelian varieties $A$ and $A'$ and let $G\leq\mathrm{Pic}^0(A), \mathrm{Pic}^0(A')$ be a finite subgroup. Assume that the Orlov's isomorphism $\gamma_{A,A'}(\Phi):A\times\widehat{A}\to A'\times\widehat{A'}$ maps $G$ into the subgroup $\{0\}\times\widehat{A'}$ of $A'\times\widehat{A'}$. Then, there exists a set $\sigma$ of $G$-equivariance natural transformations for $\Phi$ such that $(\Phi,\sigma)\in G\textrm{-}\mathrm{Eq}(\mathbf{D}^b_G(A),\mathbf{D}^b_G(A'))$, where the $G$-equivariant categories $\mathbf{D}^b_G(A),\mathbf{D}^b_G(A')$ are defined as in passage \ref{Db(B) vs DbG(A)}. 
\end{lemma}

\begin{proof}
    By the contents related to \cite[Proposition 9.45]{Huybrechts:06}, the Orlov's isomorphism $\gamma_{A,A'}(\Phi):A\times\widehat{A}\to A'\times\widehat{A'}$ is the Rouquier isomorphism associated with $\Phi$. In other words, $A\times\widehat{A}\cong\mathrm{Aut}^0(A)\ltimes\mathrm{Pic}^0(A)$ is the identity component of $\mathrm{Aut}(\mathbf{D}^b(A))$ and $\gamma_{A,A'}(\Phi)$ sends $(a,\alpha)\in A\times\widehat{A}\leq\mathrm{Aut}(\mathbf{D}^b(A))$ to $\Phi\circ (a,\alpha)\circ\Phi^{-1}$ in the sense that the composition $\Phi\circ (a,\alpha)\circ\Phi^{-1}$ is in the identity component $A'\times\widehat{A'}$ of $\mathrm{Aut}(\mathbf{D}^b(A'))$.

    We have the following example as a phototype. For the element $(0,\alpha)$ being the functor $-\otimes\mathcal{L}_{\alpha}$ of tensorization by a line bundle $\mathcal{L}_{\alpha}\in\mathrm{Pic}^0(A)$, then the element $\Phi\circ (0,\alpha)\circ\Phi^{-1}$ is tensorization with the line bundle $\mathcal{L}'_{\alpha_1}\in\mathrm{Pic}^0(A')$ such that the corresponding points $\alpha\in\widehat{A}$ and $\alpha_1\in\widehat{A'}$ satisfy $(0,\alpha_1)=\gamma_{A,A'}(\Phi)(0,\alpha)$.

    As in Definition \ref{def of group action on category}, we consider the following group action $(\rho,\theta)$ of $\widehat{A}\cong\mathrm{Pic}^0(A)$ on $\mathbf{D}^b(A)$. Given $\alpha\in\widehat{A}$, let $\rho_{\alpha}:=-\otimes\mathcal{L}_{\alpha}:\mathbf{D}^b(A)\to\mathbf{D}^b(A)$ and $\theta=\mathrm{id}$, where $\mathcal{L}_{\alpha}\in\mathrm{Pic}^0(A)$ is the corresponding line bundle. Restricting to $G\leq \mathrm{Pic}^0(A)$, we get an action, still denoted by $(\rho,\theta)$, of $G$ on $\mathbf{D}^b(A)$. That is exactly the one defined in passage \ref{Db(B) vs DbG(A)}. Conjugation by $\Phi$ induces an action $(\rho_1',\theta_1')$ of $\mathrm{Pic}^0(A)$ on $\mathbf{D}^b(A')$, where $$\rho'_{1,\alpha}:=\Phi\circ\rho_{\alpha}\circ\Phi^{-1}, \theta_1'=\mathrm{id}.$$
    Precisely, $\rho'_{1,\alpha}=-\otimes\mathcal{L}'_{\alpha_1}$. It is a tautology that $\Phi$ induces a $\mathrm{Pic}^0(A)$-functor, $(\Phi,\sigma):\mathbf{D}^b_{\mathrm{Pic}^0(A)}(A)\to\mathbf{D}^b_{\mathrm{Pic}^0(A)}(A')$, where given $\alpha\in\mathrm{Pic}^0(A)$, the 2-isomorphism $\sigma_{\alpha}:\Phi\circ\rho_{\alpha}\to\rho'_{\alpha}\circ\Phi$ is the natural composition $$\Phi\circ\rho_{\alpha}\xrightarrow{\cong}\rho'_{1,\alpha}\circ\Phi\xrightarrow{(-\otimes\mathcal{L}'_{\alpha_1}\otimes\mathcal{L}'_{\alpha})\circ\Phi}\rho'_{\alpha}\circ\Phi.$$
    Restricting to the subgroup $G$ of $\mathrm{Pic}^0(A)$, we get the desired result.
\end{proof}

\subsection{Final part of the proof of Theorem \ref{Th_GrhoA} and Theorem \ref{Th_GrhoAA'}}
We may prove Theorem \ref{Th_GrhoAA'} directly using Proposition \ref{GrhoA GrhoAA' inclusions} for the well-defined property of the map $G\textrm{-}\rho_{A,A'}$ and the commutativity of (\ref{diagram_GrhoAA'}) and Proposition \ref{GrhoA GrhoAA' surjective} for the surjectivity.
Similarly, most parts of Theorem \ref{Th_GrhoA} can be proved except the one related to $\mathrm{ker}(G\textrm{-}\rho_A)$. Indeed, we get the commutative diagram with exact rows \begin{align*}\label{weakdiagram_GrhoA}
    \xymatrix{
    0\ar[r]&\mathrm{Alb}(B)\times\widehat{B}\times\mathbb{Z}\ar[r]&\mathrm{Aut}(\mathbf{D}^b(B))\ar[r]^{\rho_B}&\mathrm{SO}^+_{\mathrm{Hdg}}(V_B)\ar[r]&0\\
    &&G\textrm{-}\mathrm{Aut}(\mathbf{D}^b(A))\ar[r]^{G\textrm{-}\rho_A}\ar[u]^{\lambda_q}\ar[d]_{F_q}&G\textrm{-}\mathrm{SO}^+_{\mathrm{Hdg}}(V_B)\ar[r]\ar[u]^{\cup}\ar[d]_{\mathrm{res}}&0\\
    0\ar[r]&\mathrm{Alb}(A)\times\widehat{A}\times\mathbb{Z}\ar[r]&\mathrm{Aut}(\mathbf{D}^b(A))\ar[r]^{\rho_A}&\mathrm{SO}^+_{\mathrm{Hdg}}(V_A)\ar[r]&0.
    }
\end{align*}
Consequently, the elements of $\mathrm{ker}(G\textrm{-}\rho_A)$ are generated by shifts, the derived autoequivalences induced by translations, and the derived autoequivalences by tensoring with the line bundles whose first Chern classes are trivial.

For an arbitrary $L\in\mathrm{Pic}^0(A)$, we get $(\overline{f},\sigma)\in\mathrm{Aut}(\mathbf{D}^b(A))$, where $\overline{f}=-\otimes L$ and $\sigma=(\sigma_g=\mathrm{id}:\overline{f}\circ\rho_g\to\rho_g\circ\overline{f})_{g\in G}$. By Remark \ref{lambda not inj}, $\lambda_q((\overline{f},\sigma))=-\otimes q^*L$. Hence, $G\textrm{-}\rho_A((\overline{f},\sigma))=\rho_B(-\otimes q^*L)$ is trivial. Thus, we get $\widehat{A}\leq\mathrm{ker}(G\textrm{-}\rho_A)$. 

The derived autoequivalences generated by the shift $[1]$ are in the kernel. Therefore, $\widehat{A}\times\mathbb{Z}\leq\mathrm{ker}(G\textrm{-}\rho_A)$.

Pick an arbitrary $a\in \mathrm{Alb}(A)$. If $(\overline{f}=(t_a)_*,\sigma)\in\mathrm{ker}(G\textrm{-}\rho_A)$, then $(t_a)_*$ is $G$-invariant. That is, $a\in\mathrm{Image}(\mathrm{Alb}(q):\mathrm{Alb}(B)\hookrightarrow\mathrm{Alb}(A))$. It yields $$\mathrm{ker}(G\textrm{-}\rho_A)=\mathrm{Alb}(B)\times\widehat{A}\times\mathbb{Z}.$$

On the other hand, for $(f,\sigma)\in\mathrm{ker}(G\textrm{-}\rho_A)$, we have $\lambda_q(f,\sigma)\in\mathrm{ker}(\rho_B)$ and $F_q(f,\sigma)=f\in\mathrm{ker}(\rho_A)$. By the lifting information in $\lambda_q$, $$\mathrm{ker}(G\textrm{-}\rho_A)=\Delta(\mathrm{Alb}(A)\times\widehat{A}\times\mathbb{Z},\mathrm{Alb}(B)\times\widehat{B}\times\mathbb{Z}),$$
where the ``diagonal" is defined as the set of elements $(b,L_1,k_1)\in\mathrm{Alb}(B)\times\widehat{A}\times\mathbb{Z}$, which corresponds to $\begin{cases}
    (a,L_1,k_1)\in\mathrm{Alb}(A)\times\widehat{A}\times\mathbb{Z}\\
    (b,L_2,k_2)\in\mathrm{Alb}(B)\times\widehat{B}\times\mathbb{Z}
\end{cases}$ such that $\begin{cases}
    a=q(b)\\
    L_2=\widehat{q}(L_1)\\
    k_1=k_2.
\end{cases}$ It gives the morphisms between the three kernels.

Therefore, the proof of (\ref{diagram_GrhoA}) and Theorem \ref{Th_GrhoA} is complete.

\subsection{Finite index properties for Theorem \ref{Th_GrhoA}}\label{finite index Th_GrhoA}

In this subsection, we establish the finite index properties for Theorem \ref{Th_GrhoA} using Lemma \ref{find sigma for overline f_0}.

Indeed, Corollary \ref{Lift derived eq_A} and Lemma \ref{find sigma for overline f_0} are related by Step 2 of the proof of Proposition \ref{GrhoA GrhoAA' surjective}. So, the image of $F_q$ is a subgroup of finite index $$[G\textrm{-}\mathrm{SO}^+_{\mathrm{Hdg}}(V_B):\mathrm{SO}^+_{\mathrm{Hdg}}(V_A)]$$ of $\mathrm{Aut}(\mathbf{D}^b(A))$.

Moreover, we have a short exact sequence 
\begin{align*}
    0\to\mathrm{ker}(q)&\to G\textrm{-}\mathrm{Aut}(\mathbf{D}^b(A))\xrightarrow{F_q}F_q(G\textrm{-}\mathrm{Aut}(\mathbf{D}^b(A)))\to0,\\
    a&\mapsto((t_a)_*,\sigma=\mathrm{id})
\end{align*}
where the description of the first homomorphism comes from Definition \ref{def of G-functor} and the fact that $G\leq\widehat{A}\cong\mathrm{Pic}^0(A)$. More precisely, (\ref{diagram_GrhoA}) can be extended to the following.
\begin{equation}\label{diagram_GrhoA+}
    \xymatrix{
    0\ar[r]&\mathrm{Alb}(B)\times\widehat{B}\times\mathbb{Z}\ar[r]&\mathrm{Aut}(\mathbf{D}^b(B))\ar[r]^{\rho_B}&\mathrm{SO}^+_{\mathrm{Hdg}}(V_B)\ar[r]&0\\
    0\ar[r]&\mathrm{Alb}(B)\times\widehat{A}\times\mathbb{Z}\ar[r]\ar[u]^{\mathrm{id}_{\mathrm{Alb}(B)}\times\widehat{q}\times\mathrm{id}_{\mathbb{Z}}}\ar[d]_{q_*\times\mathrm{id}_{\widehat{A}}\times\mathrm{id}_{\mathbb{Z}}}&G\textrm{-}\mathrm{Aut}(\mathbf{D}^b(A))\ar[r]^{G\textrm{-}\rho_A}\ar[u]^{\lambda_q}\ar[d]_{F_q}&G\textrm{-}\mathrm{SO}^+_{\mathrm{Hdg}}(V_B)\ar[r]\ar[u]^{\cup}\ar@{=}[d]&0\\
    0\ar[r]&\mathrm{Alb}(A)\times\widehat{A}\times\mathbb{Z}\ar[r]\ar@{=}[d]&F_q(G\textrm{-}\mathrm{Aut}(\mathbf{D}^b(A)))\ar[r]^{\rho_A}\ar@{^(-_>}[d]&G\textrm{-}\mathrm{SO}^+_{\mathrm{Hdg}}(V_B)\ar[r]\ar[d]_{\mathrm{res}}&0\\
    0\ar[r]&\mathrm{Alb}(A)\times\widehat{A}\times\mathbb{Z}\ar[r]&\mathrm{Aut}(\mathbf{D}^b(A))\ar[r]^{\rho_A}&\mathrm{SO}^+_{\mathrm{Hdg}}(V_A)\ar[r]&0
    }
\end{equation}

See subsection \ref{finite index n2} for a more detailed investigation of a special case.

\subsection{Symplectic maps Versus Hodge isometries}
In this subsection, we establish the following lemma that indicates the relation between symplectic maps and Hodge isometries of the first integral homology of abelian varieties. 

We first focus on the field of complex numbers.

Recall the notation related to the dual complex torus as in \cite[Section 4 in Chapter 2]{Birkenhake Lange:04}. Let $V$ be a complex vector space and $\Lambda\subset V$ be a lattice, then the quotient $X=V/\Lambda$ is a compact complex torus. Actually, $\Lambda=\mathrm{H}_1(X,\mathbb{Z})$. Define $\overline{\Omega}:=\mathrm{Hom}_{\overline{\mathbb{C}}}(V,\mathbb{C})$. In fact, $\overline{\Omega}$ is canonically isomorphic to $\mathrm{Hom}_{\mathbb{R}}(V,\mathbb{R})$. The isomorphism is given by $l\mapsto k=\mathrm{Im}l$ with the inverse map $k\mapsto l(-)=-k(i-)+ik(-)$. So we get a nondegenerate canonical $\mathbb{R}$-bilinear form $$\langle-,-\rangle:\overline{\Omega}\times V\to\mathbb{R},\langle l,v\rangle=\mathrm{Im}l(v).$$ Denote the \textit{dual lattice} of $\Lambda$ to be a lattice $\widehat{\Lambda}:=\{l\in\overline{\Omega}|\langle l,\Lambda\rangle\subset\mathbb{Z}\}$ of $\overline{\Omega}$. Then, the dual complex torus is denoted by $\widehat{X}=\overline{\Omega}/\widehat{\Lambda}$. Indeed, \cite[Proposition 4.1 in Chapter 2]{Birkenhake Lange:04} gives a canonical isomorphism $\widehat{X}\xrightarrow{\cong}\mathrm{Pic}^0(X)$. 

\begin{lemma}\label{Sp implies SO lemma} 
    Let $X_i=V_i/\Lambda_i,i=1,2$ be two compact complex tori of same dimension. Let $\widehat{X_i}=\overline{\Omega_i}/\widehat{\Lambda_i}$ be their dual complex tori. Let $$F=\begin{pmatrix}
        F_1&F_2\\F_3&F_4
    \end{pmatrix}:\Lambda_1\oplus\widehat{\Lambda_1}\xrightarrow{\cong}\Lambda_2\oplus\widehat{\Lambda_2}$$ be an isomorphism between lattices. Then $F$ is an isometry with respect to the bilinear pairings on $\Lambda_i\oplus\widehat{\Lambda_i}$ given by $(\begin{pmatrix}
        \lambda\\\widehat{\lambda}
    \end{pmatrix},\begin{pmatrix}
        \mu\\\widehat{\mu}
        \end{pmatrix})_{\Lambda_i\oplus\widehat{\Lambda_i}}:=\langle\widehat{\lambda},\mu\rangle+\langle\widehat{\mu},\lambda\rangle$ if %
        \footnote{It is not always true for the ``only if" statement since the bilinear pairings on $\Lambda_i\oplus\widehat{\Lambda_i}$ are not always positive definite.}%
    \begin{equation}\label{symplectic F}
       \begin{pmatrix}
        F_1&F_2\\F_3&F_4
    \end{pmatrix}^{-1}=\begin{pmatrix}
        \phantom{-}\widehat{F_4}&-\widehat{F_2}\\
        -\widehat{F_3}&\phantom{-}\widehat{F_1}
    \end{pmatrix}. 
    \end{equation}
\end{lemma}

\begin{proof}
    It suffices to prove that $\forall(\lambda,\widehat{\lambda})\in\Lambda_1\oplus\widehat{\Lambda_1},(\mu,\widehat{\mu})\in\Lambda_2\oplus\widehat{\Lambda_2}$,
    \begin{equation}\label{equiv relation of lemma Sp vs SO}
        (\begin{pmatrix}
        F_1&F_2\\F_3&F_4
    \end{pmatrix}\begin{pmatrix}
        \lambda\\\widehat{\lambda}
    \end{pmatrix},\begin{pmatrix}
        \mu\\\widehat{\mu}
    \end{pmatrix})=(\begin{pmatrix}
        \lambda\\\widehat{\lambda}
    \end{pmatrix},\begin{pmatrix}
        \phantom{-}\widehat{F_4}&-\widehat{F_2}\\
        -\widehat{F_3}&\phantom{-}\widehat{F_1}
    \end{pmatrix}\begin{pmatrix}
        \mu\\\widehat{\mu}
    \end{pmatrix}).
    \end{equation}

    The isomorphism $F$ is the restriction of a map $F:V_1\times\overline{\Omega_1}\to V_2\times\overline{\Omega_2}$, still denoted by $F$. This map contains a lot of information, for example, $$F_2:\overline{\Omega_1}\to V_2, \widehat{F_2}:\overline{\Omega_2}\to\mathrm{Hom}_{\overline{\mathbb{C}}}(\overline{\Omega_1},\mathbb{C}).$$

    Consider the double anti-duality isomorphism $$\widetilde{\kappa_i}:V_i\to\mathrm{Hom}_{\overline{\mathbb{C}}}(\mathrm{Hom}_{\overline{\mathbb{C}}}(V_i,\mathbb{C}),\mathbb{C})=\mathrm{Hom}_{\overline{\mathbb{C}}}(\overline{\Omega_i},\mathbb{C})$$ mapping an element $\lambda$ to $\overline{\mathrm{ev}_\lambda}$, the dual of the evaluation map, we get, for example, $$\widetilde{\kappa_1}\circ\widehat{F_2}(l_2)=l_2\circ F_2,l_2\in\overline{\Omega_2}=\mathrm{Hom}_{\overline{\mathbb{C}}}(V_2,\mathbb{C}).$$ Similarly for $\widehat{F_1}, \widehat{F_3}, \widehat{F_4}$. 

    Implicitly using the double anti-duality isomorphisms and the maps $$\widehat{F_1}:\widehat{\Lambda_2}\to\widehat{\Lambda_1},\widehat{F_2}:\widehat{\Lambda_2}\to\widehat{\widehat{\Lambda_1}},\widehat{F_3}:\widehat{\widehat{\Lambda_2}}\to\widehat{\Lambda_1},\widehat{F_4}:\widehat{\widehat{\Lambda_2}}\to\widehat{\widehat{\Lambda_1}}$$ from $F:V_1\times\overline{\Omega_1}\to V_2\times\overline{\Omega_2}$, we find that the maps $\widehat{F_1},\widehat{F_2},\widehat{F_3},\widehat{F_4}$ on the right-hand side of (\ref{symplectic F}) are identified with $$\widehat{F_1}:\widehat{\Lambda_2}\to\widehat{\Lambda_1},\widetilde{\kappa_1}^{-1}\circ\widehat{F_2}:\widehat{\Lambda_2}\to{\Lambda_1},\widehat{F_3}\circ\widetilde{\kappa_2}:{\Lambda_2}\to\widehat{\Lambda_1},\widetilde{\kappa_1}^{-1}\circ\widehat{F_4}\circ\widetilde{\kappa_2}:{\Lambda_2}\to{\Lambda_1}$$ respectively. Therefore, the formula (\ref{equiv relation of lemma Sp vs SO}) is equivalent to $$(\begin{pmatrix}
        F_1&F_2\\F_3&F_4
    \end{pmatrix}\begin{pmatrix}
        \lambda\\\widehat{\lambda}
    \end{pmatrix},\begin{pmatrix}
        \mu\\\widehat{\mu}
    \end{pmatrix})=(\begin{pmatrix}
        \lambda\\\widehat{\lambda}
    \end{pmatrix},\begin{pmatrix}
        \widetilde{\kappa_1}^{-1}\circ\widehat{F_4}\circ\widetilde{\kappa_2}&-\widetilde{\kappa_1}^{-1}\circ\widehat{F_2}\\
        -\widehat{F_3}\circ\widetilde{\kappa_2}&\widehat{F_1}
    \end{pmatrix}\begin{pmatrix}
        \mu\\\widehat{\mu}
    \end{pmatrix}).$$ 
    Using definition of the pairings $\langle-,-\rangle,(-,-)_{\Lambda_i\oplus\widehat{\Lambda_i}}$, the formula is equivalent to 
    \begin{align*}
        &\mathrm{Im}(\widehat{\mu})(F_1(\lambda)+F_2(\widehat{\lambda}))+\mathrm{Im}(F_3(\lambda)+F_4(\widehat{\lambda}))(\mu)\\
        =&\mathrm{Im}((-\widehat{F_3}\circ\widetilde{\kappa_2})(\mu)+\widehat{F_1}(\widehat{\mu}))(\lambda)+\mathrm{Im}((\widetilde{\kappa_1}\circ\widehat{F_4}\circ\widetilde{\kappa_2})(\mu)-(\widetilde{\kappa_1}^{-1}\circ\widehat{F_2})(\widehat{\mu}))(\widehat{\lambda}).
    \end{align*}
    This formula is valid by the following computations of the four terms.\begin{align*}
        \mathrm{Im}(\widehat{F_1}(\widehat{\mu}))(\lambda)&=\mathrm{Im}(\widehat{\mu}\circ F_1)(\lambda)=\mathrm{Im}\widehat{\mu}(F_1(\lambda))\\
        \mathrm{Im}(-\widehat{F_3}\circ\widetilde{\kappa_2}(\mu))(\lambda)&=-\mathrm{Im}(\overline{\mathrm{ev}_\mu}\circ F_3)(\lambda)=\mathrm{Im}(\mathrm{ev}_\mu\circ F_3)(\lambda)=\mathrm{Im}F_3(\lambda)(\mu)\\
        -\mathrm{Im}(\widetilde{\kappa_1}^{-1}\circ\widehat{F_2}(\widehat{\mu}))(\widehat{\lambda})&=\mathrm{Im}(\widehat{\mu}\circ F_2)(\lambda)=\mathrm{Im}\widehat{\mu}(F_2(\widehat{\lambda}))\\
        \mathrm{Im}((\widetilde{\kappa_1}^{-1}\circ\widehat{F_4}\circ\widetilde{\kappa_2})(\mu))&=\mathrm{Im}(F_4(\widehat{\lambda}))(\mu)
    \end{align*}

    Here, the first two computations use the formula analogue to $$\widetilde{\kappa_1}\circ\widehat{F_2}(l_2)=l_2\circ F_2,l_2\in\overline{\Omega_2}=\mathrm{Hom}_{\overline{\mathbb{C}}}(V_2,\mathbb{C}),$$
    while the last two computations use the formula $$\langle\widetilde{\kappa}(\lambda),\widetilde{\lambda}\rangle=\mathrm{Im}(\overline{\mathrm{ev}_\lambda}(\widehat{\lambda}))=-\mathrm{Im}\widehat{\lambda}(\lambda)=-\langle\lambda,\widehat{\lambda}\rangle$$
    once and twice, respectively.
\end{proof}

\begin{remark}\label{SP vs SO remark}
    Given an arbitrary derived equivalence between abelian varieties, $\Phi:\mathbf{D}^b(X_1)\xrightarrow{\simeq}\mathbf{D}^b(X_2)$, we obtain a symplectic isomorphism $$f:=\gamma_{X_1,X_2}(\Phi):X_1\times \widehat{X_1}\xrightarrow{\cong}X_2\times\widehat{X_2}$$ by Theorem \ref{Orlov fundamental th}. (Note that every element of $\mathrm{Sp}(X_1,X_2)$ is in the image of Orlov's representation $\gamma_{X_1,X_2}$ by Theorem \ref{Orlov fundamental th}.) Its rational representation is an induced isomorphism $F:\Lambda_1\times\widehat{\Lambda_1}\xrightarrow{\cong}\Lambda_2\times\widehat{\Lambda_2}$. On the other hand, since the lattice $\Lambda_i\oplus\widehat{\Lambda_i}$ satisfies $$\Lambda_i\oplus\widehat{\Lambda_i}=\mathrm{H}_1(X_i,\mathbb{Z})\oplus\mathrm{H}_1(\widehat{X_i},\mathbb{Z})\simeq\mathrm{H}^1(\widehat{X_i},\mathbb{Z})\times\mathrm{H}^1(X_i,\mathbb{Z})=V_{X_i},$$ the isomorphism $F$ can be seen as the restriction of the induced isomorphism of cohomology $f_*:\mathrm{H^*}(X_1\times\widehat{X_1},\mathbb{Q})\xrightarrow{\simeq}\mathrm{H^*}(X_2\times\widehat{X_2},\mathbb{Q})$, that preserves Hodge structures, to $F:V_{X_1}\xrightarrow{\cong}V_{X_2}$.
    
    By the definition of symplectic maps, $f\in\mathrm{Sp}(X_1,X_2)$ implies that the isomorphism $F:\Lambda_1\oplus\widehat{\Lambda_1}\xrightarrow{\cong}\Lambda_2\oplus\widehat{\Lambda_2}$ satisfies (\ref{symplectic F}). Lemma \ref{equiv relation of lemma Sp vs SO} shows that such a condition is equivalent to $F$ being an isometry with respect to the bilinear pairings on $\Lambda_i\oplus\widehat{\Lambda_i}$ given by $$((\lambda,\widehat{\lambda}),(\mu,\widehat{\mu}))_{\Lambda_i\oplus\widehat{\Lambda_i}}:=\langle\widehat{\lambda},\mu\rangle+\langle\widehat{\mu},\lambda\rangle.$$ Since the dual lattices $\widehat{\Lambda_i}$ are defined in $\mathrm{Hom}_{\mathbb{R}}(V,\mathbb{R})$, which is canonically isomorphic to $\overline{\Omega}$, where $l\in\overline{\Omega}$ corresponds to $\mathrm{Im}l\in\mathrm{Hom}_{\mathbb{R}}(V,\mathbb{R})$, the bilinear pairing $(-,-)_{\Lambda_i\oplus\widehat{\Lambda_i}}$ is equivalent to the pairing $(-,-)_{V_{X_i}}$ given by $$((\alpha_1,\beta_1),(\alpha_2,\beta_2))_{V_{X_i}}:=\beta_1(\alpha_2)+\beta_2(\alpha_1),$$ where $(\alpha_1,\beta_1), (\alpha_2,\beta_2)\in V_{X_i}=\mathrm{H}^1(X_i, \mathbb{Z})\times\mathrm{H}^1({X_i}, \mathbb{Z})^*$. 

    In summary, $f\in\mathrm{Sp}(X_1,X_2)\Rightarrow F\in\mathrm{SO}_{\mathrm{Hdg}}(V_{X_1},V_{X_2})$ over $\mathbb{C}$. Since every algebraically closed field of characteristic $0$ is contained in $\mathbb{C}$, it is also valid for all algebraically closed fields of characteristic $0$.
\end{remark}
    Moreover, $f\in\mathrm{Sp}(X_1,X_2)\Leftrightarrow F\in\mathrm{SO}^+_{\mathrm{Hdg}}(V_{X_1},V_{X_2})$ over an algebraically closed field of characteristic $0$ as in Remark \ref{two Orlov's representation remark}.

\section{Generalized Kummer varieties}\label{Generalized Kummer varieties}
    In this section, we will focus on generalized Kummer varieties of abelian surfaces. Using the tool about the $G$-equivariant version of Orlov's short exact sequence, we may find derived equivalences of generalized Kummer varieties by lifting some derived equivalences of abelian surfaces. We will also give a description of the derived equivalences obtained.
    
    Let $A$ be an abelian surface. Let $\Sigma:A^n\to A,n\geq2$, be the summation morphism and let $N_A$ be its kernel. Consider the morphism 
    \begin{align*}
        q:N_A\times A&\to A^n\\
        ((a_1,\dots,a_n),a)&\mapsto(a_1+a,\dots,a_n+a).
    \end{align*}
    Let $\mathfrak{S}_n$ act on $N_A\times A$ by the natural action on $N_A$ and the trivial action on $A$. Then $\mathrm{ker}(q)=\{((a,a,\dots,a),-a)|a\in A[n]\}$. We have the commutative diagram 
    \begin{align}\label{Sigma diagram}
        \xymatrix{
        N_A\ar[r]\ar[d]&N_A\ar[d]\\
        N_A\times A\ar[r]^{q}\ar[d]_{\pi_A}&A^n\ar[d]^{\Sigma}\\
        A\ar[r]_{n_A}&A,
        }
    \end{align}
    where $\pi_A$ is the projection, $n_A$ is the multiplication by $n$ and the top vertical arrows are the inclusions of the fibers over $0\in A$. For the diagonal embedding $\widehat{\Sigma}:\widehat{A}\to\widehat{A}^n$, we have $\widehat{\Sigma}(\widehat{A}[n])=\mathrm{ker}(\widehat{q})$. Set $G:=\mathrm{ker}(\widehat{q})\leq\widehat{A}^n, \mathcal{G}:=\widehat{A}[n]$. Then $\widehat{\Sigma}$ restricts to an isomorphism $\mathcal{G}\xrightarrow[]{\cong}G$.
    
    We have a similar diagram for another abelian surface $A'$. Using the notation as in Remark \ref{AA'notation}, we obtain $G\cong G'=\mathrm{ker}(\widehat{q'})\leq\widehat{A}'^n, \mathcal{G'}:=\widehat{A'}[n]$. Then the morphism $\widehat{\Sigma'}$ is restricted to an isomorphism $\mathcal{G'}\xrightarrow[]{\cong}G'$. Since $\widehat{\Sigma},\widehat{\Sigma'}$ are diagonal embeddings, we get an isomorphism $\mathcal{G}\cong\mathcal{G}'$.

\subsection{Lifting $\mathcal{G}$-autoequivalences of $\mathbf{D}^b_{\mathcal{G}}(A)$ to autoequivalences of \\$\mathbf{D}^b(\mathrm{Kum}^{n-1}(A)\times A)$}\label{DbGA to DbKumXA}
~\\

Using the notation of Remark \ref{AA'notation}, we get the maps
$$\mathrm{Fun}(\mathbf{D}^b(N_A\times A), \mathbf{D}^b(N_{A'}\times A'))\xleftarrow[]{\lambda_{q,q'}}G\textrm{-}\mathrm{Fun}(\mathbf{D}^b_G(A^n), \mathbf{D}^b_G(A'^n))\xrightarrow[]{F_{q,q'}}\mathrm{Fun}(\mathbf{D}^b(A^n), \mathbf{D}^b(A'^n))$$
from (\ref{Flambda_FunAA'}).

The Bridgeland-King-Reid theorem \cite{BKR:01} yields the equivalences $$\mathbf{D}^b_{\mathfrak{S}_n}(A^n)\simeq\mathbf{D}^b(A^{[n]}), \mathbf{D}^b_{\mathfrak{S}_n}(N_A\times A)\simeq\mathbf{D}^b(\mathrm{Kum}^{n-1}(A)\times A).$$
The group $\mathfrak{S}_n\times G$ acts on $\mathbf{D}^b(A^n)$, since the actions of $G$ and $\mathfrak{S}_n$ commute. As $$\mathbf{D}^b_G(A^n)\simeq\mathbf{D}^b(N_A\times A),$$
we get $\mathbf{D}^b_{\mathfrak{S}_n}(N_A\times A)\simeq\mathbf{D}^b_{\mathfrak{S}_n\times G}(A^n)$. The latter is equivalent to both $\mathbf{D}_G(\mathbf{D}^b_{\mathfrak{S}_n}(A^n))$ and $\mathbf{D}^b_{\mathfrak{S}_n}(\mathbf{D}_G(A^n))$ by \cite[Proposition 3.3]{Beckmann Oberdieck:23}. The analogous statement holds for $A'$ as well.

We have the maps

\begin{align*}
    \xymatrix{
    &\mathrm{Fun}(\mathbf{D}^b_{\mathfrak{S}_n}(A^n), \mathbf{D}^b_{\mathfrak{S}_n}(A'^n))\\
    G\textrm{-}\mathrm{Fun}(\mathbf{D}_G(\mathbf{D}^b_{\mathfrak{S}_n}(A^n)), \mathbf{D}_G(\mathbf{D}^b_{\mathfrak{S}_n}(A'^n)))\ar[dr]^{\widetilde{\lambda}_{q,q'}}\ar[ur]^{\widetilde{F}_{q,q'}}&\\
    &\mathrm{Fun}(\mathbf{D}^b_{\mathfrak{S}_n}(N_A\times A), \mathbf{D}^b_{\mathfrak{S}_n}(N_{A'}\times A')),\\
    }
\end{align*}
because a linearization of $\mathfrak{S}_n\times G$ restricts to a linearization of the subgroup ${\mathfrak{S}_n}$. Note that via BKR we get

\begin{align*}
    \xymatrix{
    &\mathrm{Fun}(\mathbf{D}^b(A^{[n]}), \mathbf{D}^b(A'^{[n]}))\\
    G\textrm{-}\mathrm{Fun}(\mathbf{D}_G(\mathbf{D}^b_{\mathfrak{S}_n}(A^n)), \mathbf{D}_G(\mathbf{D}^b_{\mathfrak{S}_n}(A'^n)))\ar[dr]^{\widetilde{\lambda}_{q,q'}}\ar[ur]^{\widetilde{F}_{q,q'}}&\\
    &\mathrm{Fun}(\mathbf{D}^b(\mathrm{Kum}^{n-1}(A)\times A), \mathbf{D}^b(\mathrm{Kum}^{n-1}(A')\times A')).\\
    }
\end{align*}

Let $\delta_{A,A'}:\mathrm{Fun}(\mathbf{D}^b(A), \mathbf{D}^b(A'))\to\mathrm{Fun}(\mathbf{D}^b_{\mathfrak{S}_n}(A^n), \mathbf{D}^b_{\mathfrak{S}_n}(A'^n))\simeq\mathrm{Fun}(\mathbf{D}^b(A^{[n]}), \mathbf{D}^b(A'^{[n]}))$ be the natural 1-functor. It is inspired by Ploog's way of obtaining derived equivalences of hyper-K\"ahler manifolds of $K3^{[n]}$-type in \cite{Ploog:05} and \cite{Ploog:07}. Recall that $\mathcal{G}:=\widehat{A}[n],\mathcal{G'}:=\widehat{A'}[n]$.

The lemma below gives a clue of lifting derived equivalences of abelian surfaces to get derived equivalences of the corresponding generalized Kummer varieties.

\begin{lemma}\label{delta_AA' Fun}
    For abelian surfaces $A,A'$, there exists a lower horizontal map $\widetilde{\delta}_{A,A'}$ that makes the square below commutative.
    \begin{align}\label{delta_FunAA'}
        \xymatrix{
        \mathrm{Fun}(\mathbf{D}^b(A),\mathbf{D}^b(A'))\ar[r]^{\delta_{A,A'}}&\mathrm{Fun}(\mathbf{D}^b(A^{[n]}), \mathbf{D}^b(A'^{[n]}))\\
        \mathcal{G}\textrm{-}\mathrm{Fun}(\mathbf{D}^b_{\mathcal{G}}(A),\mathbf{D}^b_{\mathcal{G}}(A'))\ar[u]^{F_{n_A,n_{A'}}}\ar[d]_{{\lambda_{n_A,n_{A'}}}}\ar[r]^(0.4){\widetilde{\delta}_{A,A'}}& G\textrm{-}\mathrm{Fun}(\mathbf{D}_G(\mathbf{D}^b_{\mathfrak{S}_n}(A^n)), \mathbf{D}_G(\mathbf{D}^b_{\mathfrak{S}_n}(A'^n)))\ar[u]^{\widetilde{F}_{q,q'}}\ar[d]_{\widetilde{\lambda}_{q,q'}}\\
        \mathrm{Fun}(\mathbf{D}^b(A), \mathbf{D}^b(A'))&\mathrm{Fun}(\mathbf{D}^b(\mathrm{Kum}^{n-1}(A)\times A), \mathbf{D}^b(\mathrm{Kum}^{n-1}(A')\times A'))
        }
    \end{align}
    Moreover, for the case $A=A'$, denote the map $\widetilde{\delta}_{A,A}$ by $\widetilde{\delta}_{A}$.
\end{lemma}

\begin{proof}
    We consider an element, $(f,\sigma)\in\mathcal{G}\textrm{-}\mathrm{Fun}(\mathbf{D}^b_{\mathcal{G}}(A),\mathbf{D}^b_{\mathcal{G}}(A'))$, which consists of $f$ in $\mathrm{Fun}(\mathbf{D}^b(A),\mathbf{D}^b(A'))$ together with $\sigma=(\sigma_g)_{g\in\mathcal{G}}$, being a set of $\mathcal{G}$-equivariance natural transformations. Now $\delta_{A,A'}(f)$ is a functor from $\mathbf{D}^b_{\mathfrak{S}_n}(A^n)$ to $\mathbf{D}^b_{\mathfrak{S}_n}(A'^n)$. We need to show that the data $(\sigma_g)_{g\in \mathcal{G}}$, correspond to a set of $G$-equivariance natural transformations for $\delta_{A,A'}(f)$ via the isomorphism $\mathcal{G}\xrightarrow[]{\cong}G$ induced by the restriction of $\widehat{\Sigma}:\widehat{A}\to\widehat{A}^n$.

    Consider the functor $\Pi_A$ with objects and morphisms mapped as
    \begin{align*}
        \Pi_A:\mathbf{D}^b(A)&\to\mathbf{D}^b_{\mathfrak{S}_n}(A^n)\\
        E&\mapsto(E^{\boxtimes n},\tau)\\
        (f:E_1\to E_2)&\mapsto(f^{\boxtimes n}:(E_1^{\boxtimes n},\tau)\to (E_2^{\boxtimes n},\tau)),
    \end{align*}
    where $\tau$ is the permutation linearization, the object $E^{\boxtimes n}$ is defined as $$E\boxtimes\dots\boxtimes E=\overset{n}{\underset{i=1}{\oplus}}\pi_i^*E$$ 
    in $\mathbf{D}^b(A^n)$ and the morphism $f^{\boxtimes n}$ is defined as a morphism $$f\boxtimes\dots\boxtimes f=\overset{n}{\underset{i=1}{\oplus}}\pi_i^*f:(E_1^{\boxtimes n},\tau)\to (E_2^{\boxtimes n},\tau)$$
    in $\mathbf{D}^b_{\mathfrak{S}_n}(A^n)$ as in Definition \ref{def of equivariant category}.

    Note that $\Pi_A$ is not a functor of triangulated categories (it is not even additive). The functor $\Pi_{A}$ restricts to the subset of objects $\widehat{A}$ as $\widehat{\Sigma}$. The group $\mathrm{Pic}^0(A)$ acts on both $\mathbf{D}^b(A)$ and $\mathbf{D}^b_{\mathfrak{S}_n}(A^n)$ and $\Pi_{A}$ extends to a $\widehat{A}$-functor $(\Pi_{A},\sigma)$, where given $L\in\mathrm{Pic}^0(A)\cong\widehat{A}$, considered as an autoequivalence of $\mathbf{D}^b(A)$, we have the natural transformation $$\sigma_L:\Pi_A\circ(L\otimes -)\to\widehat{\Sigma}(L)\otimes\Pi_A(-).$$ 
    Since $\widehat{\Sigma}(L)=\Pi_A(L)$ and $\Pi_A$ is a multiplicative functor, $\Pi_A$ sends tensor products to tensor products. So, the natural transformation $\sigma_L$ is in fact an isomorphism.

    Moreover, we have a 2-functor $\Pi$ from the 2-category of derived categories of projective varieties $\mathcal{D}^b(\mathcal{P}roj)$ to the 2-category of $\mathfrak{S}_n$-equivariant derived categories of projective varieties $\mathcal{D}^b_{\mathfrak{S}_n}(\mathcal{P}roj)$.
    \begin{equation}\label{def of Pi}
        \begin{aligned}
            \Pi:\mathcal{D}^b(\mathcal{P}roj)&\to\mathcal{D}^b_{\mathfrak{S}_n}(\mathcal{P}roj)\\
        \mathbf{D}^b(A)&\xrightarrow[]{\Pi_A}\mathbf{D}^b_{\mathfrak{S}_n}(A^n)\\
        (F:\mathbf{D}^b(X)\to\mathbf{D}^b(Y))&\mapsto(F^{\boxtimes n}:\mathbf{D}^b_{\mathfrak{S}_n}(X^n)\to\mathbf{D}^b_{\mathfrak{S}_n}(Y^n))\\
        \mathrm{natural}\phantom{1}\mathrm{transformation}&\mapsto\mathrm{natural}\phantom{1}\mathrm{transformation}
        \end{aligned}
    \end{equation}
    Here, for the functor $F:\mathbf{D}^b(X)\to\mathbf{D}^b(Y)$ mapping a morphism $f:E_1\to E_2$ in $\mathbf{D}^b(X)$ to a morphism $F(f):F(E_1)\to F(E_2)$ in $\mathbf{D}^b(Y)$, the functor $$F^{\boxtimes n}:\mathbf{D}^b_{\mathfrak{S}_n}(X^n)\to\mathbf{D}^b_{\mathfrak{S}_n}(Y^n)$$
    is defined to map the morphism $\Pi (f):=f^{\boxtimes n}:(E_1^{\boxtimes n},\tau)\to (E_2^{\boxtimes n},\tau)$ in $\mathbf{D}^b_{\mathfrak{S}_n}(X^n)$ to a morphism $\Pi (F(f)):=F(f)^{\boxtimes n}:(F(E_1)^{\boxtimes n},\tau)\to (F(E_2)^{\boxtimes n},\tau)$ in $\mathbf{D}^b_{\mathfrak{S}_n}(Y^n)$.
    
    Hence, $\Pi$ induces a 1-functor $\Pi:\mathrm{Fun}(\mathbf{D}^b(A), \mathbf{D}^b(A'))\to\mathrm{Fun}(\mathbf{D}^b_{\mathfrak{S}_n}(A^{n}), \mathbf{D}^b_{\mathfrak{S}_n}(A'^{n}))$. Composing such a 1-functor $\Pi$ with conjugation by BKR, we get the 1-functor $\delta_{A,A'}$ above. Given $(f,\sigma=(\sigma_g)_{g\in\mathcal{G}})\in\mathcal{G}\textrm{-}\mathrm{Fun}(\mathbf{D}^b_{\mathcal{G}}(A),\mathbf{D}^b_{\mathcal{G}}(A'))$, we get the $G$-functor $$\widetilde{\delta}_{A,A'}(f,\sigma):=(\Pi(f),(\Pi(\sigma_g)=\sigma_g^{\boxtimes n})_{g\in\mathcal{G}\cong G})\in G\textrm{-}\mathrm{Fun}(\mathbf{D}_G(\mathbf{D}^b_{\mathfrak{S}_n}(A^n)), \mathbf{D}_G(\mathbf{D}^b_{\mathfrak{S}_n}(A'^n))).$$ So we get the map $\widetilde{\delta}_{A,A'}$ such that the square in (\ref{delta_FunAA'}) commute. To be explicit, we have 
    \begin{align*}
        \xymatrix{
        &\delta_{A,A'}(f)\ar@{<->}[d]^{BKR}\\
        f\ar@{|->}[ur]&\Pi(f)\\
        (f,\sigma)\ar@{|->}[u]\ar@{|->}[r]&\widetilde{\delta}_{A,A'}(f,\sigma)=(\Pi(f),(\Pi(\sigma_g)=\sigma_g^{\boxtimes n})_{g\in\mathcal{G}\cong G})\ar@{|->}[u]
        }
    \end{align*}
    for the maps in the square of (\ref{delta_FunAA'}).
\end{proof}

We may get the following proposition by considering derived (auto)equivalences.

\begin{proposition}\label{delta_AA' Eq}
    For abelian surfaces $A,A'$, there exists a lower horizontal map $\widetilde{\delta}_{A,A'}$ that makes the square below commutative.
    \begin{align}\label{delta_EqAA'}
        \xymatrix{
        \mathrm{Eq}(\mathbf{D}^b(A),\mathbf{D}^b(A'))\ar[r]^{\delta_{A,A'}}&\mathrm{Eq}(\mathbf{D}^b(A^{[n]}), \mathbf{D}^b(A'^{[n]}))\\
        \mathcal{G}\textrm{-}\mathrm{Eq}(\mathbf{D}^b_{\mathcal{G}}(A),\mathbf{D}^b_{\mathcal{G}}(A'))\ar[u]^{F_{n_A,n_{A'}}}\ar[d]_{{\lambda_{n_A,n_{A'}}}}\ar[r]^(0.4){\widetilde{\delta}_{A,A'}}& G\textrm{-}\mathrm{Eq}(\mathbf{D}_G(\mathbf{D}^b_{\mathfrak{S}_n}(A^n)), \mathbf{D}_G(\mathbf{D}^b_{\mathfrak{S}_n}(A'^n)))\ar[u]^{\widetilde{F}_{q,q'}}\ar[d]_{\widetilde{\lambda}_{q,q'}}\\
        \mathrm{Eq}(\mathbf{D}^b(A), \mathbf{D}^b(A'))&\mathrm{Eq}(\mathbf{D}^b(\mathrm{Kum}^{n-1}(A)\times A), \mathbf{D}^b(\mathrm{Kum}^{n-1}(A')\times A'))
        }
    \end{align}
    Moreover, for the case of derived autoequivalences, denote the map $\widetilde{\delta}_{A,A}$ by $\widetilde{\delta}_{A}$. Then the diagram above can be written as 
    \begin{align}\label{delta_EqA}
        \xymatrix{
        \mathrm{Aut}(\mathbf{D}^b(A))\ar[r]^{\delta_{A}}&\mathrm{Aut}(\mathbf{D}^b(A^{[n]}))\\
        \mathcal{G}\textrm{-}\mathrm{Aut}(\mathbf{D}^b_{\mathcal{G}}(A))\ar[u]^{F_{n_A}}\ar[d]_{{\lambda_{n_A}}}\ar[r]^(0.42){\widetilde{\delta}_{A}}& G\textrm{-}\mathrm{Aut}(\mathbf{D}_G(\mathbf{D}^b_{\mathfrak{S}_n}(A^n)))\ar[u]^{\widetilde{F}_{q}}\ar[d]_{\widetilde{\lambda}_{q}}\\
        \mathrm{Aut}(\mathbf{D}^b(A))&\mathrm{Aut}(\mathbf{D}^b(\mathrm{Kum}^{n-1}(A)\times A)).
        }
    \end{align}
\end{proposition}
So we can lift the $\mathcal{G}$-functors, which are derived equivalences from $\mathbf{D}^b_{\mathcal{G}}(A)$ to $\mathbf{D}^b_{\mathcal{G}}(A')$, to derived equivalences from $\mathbf{D}^b(\mathrm{Kum}^{n-1}(A)\times A)$ to $\mathbf{D}^b(\mathrm{Kum}^{n-1}(A')\times A')$. Moreover, we can also lift the $\mathcal{G}$-functors, which are derived autoequivalences of $\mathbf{D}^b_{\mathcal{G}}(A)$, to derived autoequivalences of $\mathbf{D}^b(\mathrm{Kum}^{n-1}(A)\times A)$.

\subsection{Lifting $\mathcal{G}$-autoequivalences of $\mathbf{D}^b_{\mathcal{G}}(A)$ to autoequivalences of $\mathbf{D}^b(\mathrm{Kum}^{n-1}(A))$}\label{DbGA to DbKum}

The result in Proposition \ref{delta_AA' Eq} is not sufficient to obtain the derived equivalences of generalized Kummer varieties. In this subsection, we shall prove that the derived equivalences obtained above, from $\mathbf{D}^b(\mathrm{Kum}^{n-1}(A)\times A)$ to $\mathbf{D}^b(\mathrm{Kum}^{n-1}(A')\times A')$, are composed of some derived equivalences from $\mathbf{D}^b(\mathrm{Kum}^{n-1}(A))$ to $\mathbf{D}^b(\mathrm{Kum}^{n-1}(A'))$ and some derived equivalences from $\mathbf{D}^b(A)$ to $\mathbf{D}^b(A')$ in a unique way.

Notations as above. We would like to show that given an arbitrary $\mathcal{G}$-functor $(f,\sigma)\in\mathcal{G}\textrm{-}\mathrm{Eq}(\mathbf{D}^b_{\mathcal{G}}(A),\mathbf{D}^b_{\mathcal{G}}(A'))$, we have \begin{equation}\label{split_AA'}
    \widetilde{\lambda}_{q,q'}(\widetilde{\delta}_{A,A'}(f,\sigma))=\Phi_{(f,\sigma)}\times\Psi_{(f,\sigma)}
\end{equation}
for some $\Phi_{(f,\sigma)}\in\mathrm{Eq}(\mathbf{D}^b(\mathrm{Kum}^{n-1}(A)),\mathbf{D}^b(\mathrm{Kum}^{n-1}(A'))), \Psi_{(f,\sigma)}\in\mathrm{Eq}(\mathbf{D}^b(A),\mathbf{D}^b(A'))$. Moreover, we will show that this splitting is unique.

Analogue for the derived autoequivalence case. We will show that given an arbitrary $(f,\sigma)\in\mathcal{G}\textrm{-}\mathrm{Eq}(\mathbf{D}^b_{\mathcal{G}}(A))$, we have 
\begin{equation}\label{split_A}
    \widetilde{\lambda}_{q}(\widetilde{\delta}_{A}(f,\sigma))=\Phi_{(f,\sigma)}\times\Psi_{(f,\sigma)}
\end{equation}
for some $\Phi_{(f,\sigma)}\in\mathrm{Aut}(\mathbf{D}^b(\mathrm{Kum}^{n-1}(A))), \Psi_{(f,\sigma)}\in\mathrm{Aut}(\mathbf{D}^b(A))$, which are unique.

Note that the splitting for general derived functors is not easy since important tools such as Orlov's Theorem \cite[Theorem 5.14]{Huybrechts:06} and Theorem \ref{Th_GrhoAA'} used in the proof later are valid for derived equivalences only. So we only consider derived equivalence in this subsection.

We consider the following two basic examples before the proof of splitting.

\begin{example}\label{translation split eg}
    Let $(f_*,\sigma)\in\mathcal{G}\textrm{-}\mathrm{Aut}(\mathrm{D}^b_{\mathcal{G}}(A))$ be such that $f:A\to A$ is an automorphism (no need for $f(0)=0$.) and $\lambda_{n_A}(f_*,\sigma)=\widetilde{f}_*$, where $\widetilde{f}:A\to A$ is an automorphism satisfying $\widetilde{f}(a)=f(a)+a_0, \forall a\in A$ for some fixed $a_0\in A[n]$. In fact, it follows from the understanding of $\lambda_q$ using the lifting in Remark \ref{lambda_q explain B}. Note that by the lifting related to $\lambda_q$, we get $n_A\circ\widetilde{f}=f\circ n_A$ and $n\widetilde{f}(a)=f(na)$. Hence, $n\widetilde{f}(0)=f(0)$ and $\widetilde{f}(a)-\widetilde{f}(0)=f(a)-f(0)$. The map $\widetilde{f}^n:A^n\to A^n$ maps $N_A$ to $N_A$. Let $f^\#:A\to A$ be an automorphism such that $f^\#(a)=f(a)-f(0)$. Then $f^\#(0)=0$ and $f^\#$ is a group homomorphism.

    Then $\Pi_A(f)=f^n:A^n\to A^n$ lifts to the automorphism $h:N_A\times A\to N_A\times A$, given by $$h((a_1,\dots,a_n),a)=((f^\#(a_1),\dots,f^\#(a_n)),f(a))$$ since $q\circ h=f^n\circ q$. Note that $\Sigma$ in the diagram (\ref{Sigma diagram}) satisfies $$\Sigma\circ\widetilde{f}^n=f\circ\Sigma$$ but $\Sigma\circ f^n\not=f\circ\Sigma$ when $nf(0)\not=f(0)$. However, we need to lift $\Pi_A(f)=f^n$ and not $\widetilde{f}^n$.

    We abuse the notation and consider $\widetilde{\lambda}_q(\widetilde{\delta}_A(f_*,\sigma))$ as a derived autoequivalence of $\mathbf{D}^b_{\mathfrak{S}_n}(N_A\times A)$ rather than $\mathbf{D}^b(\mathrm{Kum}^{n-1}(A)\times A)$. Then we have $$\widetilde{\lambda}_q(\widetilde{\delta}_A(f_*,\sigma))=\widetilde{\lambda}_q((f^n)_*,\widehat{\Sigma}_*\sigma)=h_*,$$ where $\widehat{\Sigma}_*\sigma$ is the set of $G$-equivariance natural transformations for $(f^n)_*$, which is induced by the set $\sigma$ of $\mathcal{G}$-equivariance natural transformations for $f$ via $\mathcal{G}\xrightarrow[]{\cong}G$ restricted from the morphism $\widehat{\Sigma}$. In this case, both $f^n$ and $h$ are $\mathfrak{S}_n$-equivariant. Note that we have a factorization 
    \begin{align*}
        &(h_*:\mathbf{D}^b_{\mathfrak{S}_n}(N_A\times A)\to\mathbf{D}^b_{\mathfrak{S}_n}(N_A\times A))\\
        =&(((f^\#)^n|_{N_A})_*:\mathbf{D}^b_{\mathfrak{S}_n}(N_A)\to\mathbf{D}^b_{\mathfrak{S}_n}(N_A))\times(f_*:\mathbf{D}^b(A)\to\mathbf{D}^b(A)),
    \end{align*}
    where $(f^\#)^n|_{N_A}$ is the restriction of the automorphism $(f^\#)^n$ to $N_A$. This restriction is an $\mathfrak{S}_n$-equivariant automorphism. Thus, $(f_*,\sigma)\in\mathcal{G}\textrm{-}\mathrm{Aut}(\mathrm{D}^b_{\mathcal{G}}(A))$ satisfies the condition (\ref{split_A}) and $\widetilde{\lambda}_q(\widetilde{\delta}_A(f_*,\sigma))$ splits as desired.

    In particular, if $f=t_a$ is the translation in $A$ with $a\in A$ (one may further assume $a\in A[n]$.), we have $(f_*,\sigma=\mathrm{id})\in\mathcal{G}\textrm{-}\mathrm{Aut}(\mathrm{D}^b_{\mathcal{G}}(A))$. Obviously, $f^{\#}=\mathrm{id}_A$ and $h=(\overbrace{\mathrm{id}_A,\cdots,\mathrm{id}_A}^{n\phantom{1}\mathrm{tuple}},t_a):N_A\times A\to N_A\times A$. So we have a factorization 
    \begin{equation}\label{f=t_a split}
        \begin{aligned}
            &(\widetilde{\lambda}_q(\widetilde{\delta}_A((t_a)_*,\mathrm{id}))=h_*:\mathbf{D}^b_{\mathfrak{S}_n}(N_A\times A)\to\mathbf{D}^b_{\mathfrak{S}_n}(N_A\times A))\\
        =&(\mathrm{id}_{\mathbf{D}^b_{\mathfrak{S}_n}(N_A)}:\mathbf{D}^b_{\mathfrak{S}_n}(N_A)\to\mathbf{D}^b_{\mathfrak{S}_n}(N_A))\times((t_a)_*:\mathbf{D}^b(A)\to\mathbf{D}^b(A)).
        \end{aligned}
    \end{equation}
\end{example}

\begin{example}\label{line bundle split eg}
    Let $L$ be a line bundle on $A$ and let $f=-\otimes L\in\mathrm{Aut}(\mathbf{D}^b(A))$. Let $\sigma:=(\sigma_{L'}=\mathrm{id}:f(-\otimes L')\xrightarrow[]{\cong} f(-)\otimes L')_{L'\in\widehat{A}[n]}$ be the set of natural transformations corresponding to the commutativity of the tensor products by the line bundles $L$ and $L'$. Then $\lambda_{n_A}(f,\sigma=\mathrm{id})=-\otimes n_A^*(L)$. We have $n_A^*(L)\cong L^{\frac{n^2+n}{2}}\otimes (-1_A)^*L^{\frac{n^2-n}{2}}$ by \cite[Proposition 2.3.5]{Birkenhake Lange:04}. We claim that the following factorization holds: $$\widetilde{\lambda}_q(\widetilde{\delta}_A(f,\sigma=\mathrm{id}))=\Phi_{(f,\sigma)}\times\Psi_{(f,\sigma)},$$
    where $\Phi_{(f,\sigma)}=(-\otimes M,\sigma'):=(-\otimes L^{\boxtimes n},\Pi(\sigma))|_{N_A}$ and $\Psi_{(f,\sigma)}=-\otimes L^n$. Here we abuse the notation and regard the first factor $\Phi_{(f,\sigma)}$ as an autoequivalence of $\mathbf{D}^b_{\mathfrak{S}_n}(N_A)$ rather than $\mathbf{D}^b(\mathrm{Kum}^{n-1}(A))$. Note that $\Pi(\sigma)=\mathrm{id}$ and $\sigma'=\mathrm{id}$ since $\sigma=\mathrm{id}$. Since $\widetilde{\lambda}_q(\widetilde{\delta}_A(f,\sigma))=q^*(-\otimes L^{\boxtimes n},\Pi(\sigma))$, we need to show that $$q^*(-\otimes L^{\boxtimes n},\Pi(\sigma))\cong(-\otimes M,\sigma')\times (-\otimes L^n).$$ It suffices to show that $q^*(L^{\boxtimes n})|_{\{(x_1,\dots,x_n)\}\times A}=L^n,\forall(x_1,\dots,x_n)\in N_A$.

    Recall that $\varphi_L:A\to\widehat{A}$, given by $\varphi_L(t):=L^{-1}\otimes\tau_t^*(L)$, is a group homomorphism. Let $x=(x_1,\dots, x_n)\in N_A$ and let $\iota_x:A\hookrightarrow N_A\times A$ be the inclusion given by $\iota_x(a)=(x,a)$. Note the equality $\pi_j\circ q\circ \iota_x=\tau_{x_j}$. We have 
    \begin{align*}
        \iota_x^*(q^*(L^{\boxtimes n}))&\cong\iota_x^*(q^*( \underset{j=1}{\overset{n}{\otimes}}\pi_j^*L))\cong\underset{j=1}{\overset{n}{\otimes}}(\pi_j\circ q\circ\iota_x)^*L\cong\underset{j=1}{\overset{n}{\otimes}}\tau_{x_j}^*L\\
        &\cong L^n\otimes(\underset{j=1}{\overset{n}{\otimes}}[L^{-1}\otimes\tau_{x_j}^*L])\cong L^n\otimes\varphi_L(\underset{j=1}{\overset{n}{\sum}}x_j)\cong L^n.
    \end{align*}
    It finishes the proof. Thus, $(f,\sigma)=(-\otimes L,\mathrm{id})\in\mathcal{G}\textrm{-}\mathrm{Aut}(\mathrm{D}^b_{\mathcal{G}}(A))$ satisfies the condition (\ref{split_A}). More precisely,
    \begin{equation}\label{f=otimesL split}
        \widetilde{\lambda}_q(\widetilde{\delta}_A(-\otimes L,\mathrm{id}))=q^*(-\otimes L^{\boxtimes n},\mathrm{id})\cong(-\otimes (L^{\boxtimes n}
        |_{N_A}),\mathrm{id})\times (-\otimes L^n).
    \end{equation}
\end{example}

Similarly, we have the following splitting $$\widetilde{\lambda}_q(\widetilde{\delta}_A([1],\mathrm{id}))=\Phi_{([1],\mathrm{id})}\times\Psi_{([1],\mathrm{id})}$$
for the shift $[1]\in\mathrm{Aut}(\mathbf{D}^b(A))$, where $$\Phi_{([1],\mathrm{id})}=[1]\in\mathrm{Aut}(\mathbf{D}^b(\mathrm{Kum}^{n-1}(A)),\Psi_{([1],\mathrm{id})}=[1]\in\mathrm{Aut}(\mathbf{D}^b(A)).$$ That is,\begin{equation}\label{f=shift split}
    \widetilde{\lambda}_q(\widetilde{\delta}_A([1],\mathrm{id}))=[1]\times[1].
\end{equation} 

Combining the computation of the shift with Examples \ref{translation split eg} and \ref{line bundle split eg}, we may see that any autoequivalence $f$ of $\mathbf{D}^b(A)$, which are trivial in the Orlov's representation $\rho_A$ (i.e. shifts, tensor with line bundles and translations, by Orlov's short exact sequence), together with the set of trivial natural transformations $\sigma=\mathrm{id}$, can make the splitting (\ref{split_A}) hold. It gives a hint to consider Orlov's representation in the reasoning below.

Now, let us prove the main theorem about splitting.
\begin{theorem}\label{split_NAANA'A'}
     Let notation be as above over an algebraically closed field of characteristic $0$. For arbitrary $(f,\sigma)\in\mathcal{G}\textrm{-}\mathrm{Eq}(\mathbf{D}^b_{\mathcal{G}}(A),\mathbf{D}^b_{\mathcal{G}}(A'))$, the splitting 
     \begin{equation*}
         \widetilde{\lambda}_{q,q'}(\widetilde{\delta}_{A,A'}(f,\sigma))=\Phi_{(f,\sigma)}\times\Psi_{(f,\sigma)}
     \end{equation*} holds for a unique combination of $\begin{cases}
         \Phi_{(f,\sigma)}\in\mathrm{Eq}(\mathbf{D}^b(\mathrm{Kum}^{n-1}(A)),\mathbf{D}^b(\mathrm{Kum}^{n-1}(A')))\\
         \Psi_{(f,\sigma)}\in\mathrm{Eq}(\mathbf{D}^b(A),\mathbf{D}^b(A')).
     \end{cases}$\\
\end{theorem}

\begin{proof}
    We abuse the notation and consider $\widetilde{\lambda}_{q,q'}(\widetilde{\delta}_{A,A'}(f,\sigma))$ as a derived equivalence in $\mathrm{Eq}(\mathbf{D}^b_{\mathfrak{S}_n}(N_A\times A),\mathbf{D}^b_{\mathfrak{S}_n}(N_{A'}\times A'))$ instead of a derived equivalence in $\mathrm{Eq}(\mathbf{D}^b(\mathrm{Kum}^{n-1}(A)\times A),\mathbf{D}^b(\mathrm{Kum}^{n-1}(A')\times A'))$.
    By Theorem \ref{Orlov fundamental th} and \cite[Proposition 9.39]{Huybrechts:06}, the derived equivalence $\widetilde{\lambda}_{q,q'}(\widetilde{\delta}_{A,A'}(f,\sigma))$ is highly dependent on the symplectic map $\eta:=\gamma_{N_A\times A, N_{A'}\times A'}(\widetilde{\lambda}_{q,q'}(\widetilde{\delta}_{A,A'}(f,\sigma)))\in\mathrm{Sp}(N_A\times A, N_{A'}\times A')$. So we start the proof by investigating $\eta$.
    
    \textit{Step 1:} Prove that $\eta$ splits into a symplectic map in $\mathrm{Sp}(N_A, N_{A'})$ and one in $\mathrm{Sp}(A, {A'})$.

    Since $\widetilde{\lambda}_{q,q'}(\widetilde{\delta}_{A,A'}(f,\sigma))\in\mathrm{Eq}(\mathbf{D}^b_{\mathfrak{S}_n}(N_A\times A),\mathbf{D}^b_{\mathfrak{S}_n}(N_{A'}\times A'))$ is $\mathfrak{S}_n$-equivariant with respect to the $\mathfrak{S}_n$-actions on $N_A\times A, N_{A'}\times A'$, which are composed of the natural actions of $\mathfrak{S}_n$ on $N_A,N_{A'}$ and the trivial actions of $\mathfrak{S}_n$ on $A,{A'}$, the symplectic map $\eta:N_A\times A\times\widehat{N_A}\times\widehat{A}\to N_{A'}\times A'\times\widehat{N_{A'}}\times\widehat{A'}$ is also $\mathfrak{S}_n$-equivariant. The group homomorphism $\eta$ is determined by the induced homomorphism in cohomology $\mathrm{H}^1(\eta):\mathrm{H}^1(N_A\times A\times\widehat{N_A}\times\widehat{A})\to \mathrm{H}^1(N_{A'}\times A'\times\widehat{N_{A'}}\times\widehat{A'})$. Since $\mathrm{H}^1(\eta)$ is also $\mathfrak{S}_n$-equivariant, it must map the $\mathfrak{S}_n$-invariant part of $\mathrm{H}^1(N_A\times A\times\widehat{N_A}\times\widehat{A})$ to the $\mathfrak{S}_n$-invariant part of $\mathrm{H}^1(N_{A'}\times A'\times\widehat{N_{A'}}\times\widehat{A'})$. So we have a homomorphism in cohomology, $\mathrm{H}^1(A\times\widehat{A})\to\mathrm{H}^1(A'\times\widehat{A'})$, which is induced by a group homomorphism $\eta_2:A\times\widehat{A}\to A'\times\widehat{A'}$. So we get $$\eta_2=p_{A\times\widehat{A},{A'}\times\widehat{{A'}}}\circ\eta\in\mathrm{Sp}(A,A'),$$ where $p_{A\times\widehat{A},{A'}\times\widehat{{A'}}}$ is a projection. The other component of $\eta$ is $$\eta_1=p_{N_A\times\widehat{N_A},N_{A'}\times\widehat{N_{A'}}}\circ\eta\in\mathrm{Sp}(N_A,N_{A'}),$$
    where $p_{N_A\times\widehat{N_A},N_{A'}\times\widehat{N_{A'}}}$ is a projection. It yields the desired splitting $$\eta=\eta_1\times\eta_2, \exists\eta_1\in\mathrm{Sp}(N_A, N_{A'}),\exists\eta_2\in\mathrm{Sp}(A,A').$$

    \textit{Step 2:} Construct a splitting of $\widetilde{\lambda}_{q,q'}(\widetilde{\delta}_{A,A'}(f,\sigma))$.

    Since the symplectic map $\eta$ splits from Step 1, we may find a derived equivalence $\Phi\in\mathrm{Eq}(\mathbf{D}^b(N_A\times A),\mathbf{D}^b(N_{A'}\times A'))$ with $\gamma_{N_A\times A, N_{A'}\times A'}(\Phi)=\eta$ such that the splitting $\Phi=\Phi_1\times\Phi_2$ holds, where $\begin{cases}
        \Phi_1\in\mathrm{Eq}(\mathbf{D}^b(N_A),\mathbf{D}^b(N_{A'}))\\
        \Phi_2\in\mathrm{Eq}(\mathbf{D}^b(A),\mathbf{D}^b({A'}))
    \end{cases}$ with $\begin{cases}
        \gamma_{N_A, N_{A'}}(\Phi_1)=\eta_1\\
        \gamma_{A, {A'}}(\Phi_2)=\eta_2
    \end{cases}$, by using the surjectivity result in Theorem \ref{Orlov fundamental th}. 
    
    Also by Theorem \ref{Orlov fundamental th}, since $\gamma_{N_A\times A, N_{A'}\times A'}(\Phi)=\eta$, the derived equivalence $\widetilde{\lambda}_{q,q'}(\widetilde{\delta}_{A,A'}(f,\sigma))$ differs from $\Phi$ by a translation $$t\in N_{A'}\times A'\times\widehat{N_{A'}}\times\widehat{A'}\times\mathbb{Z}\leq\mathrm{Aut}(\mathbf{D}^b(N_{A'}\times A')).$$
    That is, $\widetilde{\lambda}_{q,q'}(\widetilde{\delta}_{A,A'}(f,\sigma))=t\circ\Phi$, where the derived autoequivalence $t$ is generated by the ones induced from translations $(t_{(\vec{a_1'},a_2')})_*$ with $\begin{cases}
        \vec{a_1'}\in N_{A'}\\
        a_2'\in A'
    \end{cases}$, the ones by tensoring with line bundle $-\otimes(L_1\boxtimes L_2)$ with $\begin{cases}
        L_1\in\mathrm{Pic}^0(N_{A'})\cong\widehat{N_{A'}}\\
        L_2\in\mathrm{Pic}^0(A')\cong\widehat{{A'}}
    \end{cases}$, and the shift $[1]$.
    
    Since the generators $(t_{(\vec{a_1'},a_2')})_*$, $-\otimes(L_1\boxtimes L_2)$ and $[1]$ are derived autoequivalences of the equivariant category $\mathbf{D}^b_{\mathfrak{S}_n}(N_{A'}\times A')$, we have $t\in\mathrm{Aut}(\mathbf{D}^b_{\mathfrak{S}_n}(N_{A'}\times A'))$. Since $\widetilde{\lambda}_{q,q'}(\widetilde{\delta}_{A,A'}(f,\sigma))$ is a derived equivalence between equivariant categories, so are $\Phi$ and the splitting component $\Phi_1$. To summarize, we have the splitting $\Phi=\Phi_1\times\Phi_2$, where $\begin{cases}
        \Phi\in\mathrm{Eq}(\mathbf{D}^b_{\mathfrak{S}_n}(N_A\times A),\mathbf{D}^b_{\mathfrak{S}_n}(N_{A'}\times A'))\\
        \Phi_1\in\mathrm{Eq}(\mathbf{D}^b_{\mathfrak{S}_n}(N_A),\mathbf{D}^b_{\mathfrak{S}_n}(N_{A'}))\\
        \Phi_2\in\mathrm{Eq}(\mathbf{D}^b(A),\mathbf{D}^b({A'})).
    \end{cases}$

    For the generators of $t$, they can be split into derived autoequivalences of $\mathbf{D}^b_{\mathfrak{S}_n}(N_{A'})$ and derived autoequivalences of $\mathbf{D}^b({A'})$. That is,
    \begin{equation}\label{generators of t split}
        \begin{aligned}
            t_{(\vec{a_1'},a_2')})_*&=(t_{\vec{a_1'}})_*\times(t_{a_2'})_*\\
        -\otimes(L_1\boxtimes L_2)&=(-\otimes L_1)\times(-\otimes L_2)\\
        [1]&=[1]\times[1].
        \end{aligned}
    \end{equation}
    So we get the splitting $t=t_1\times t_2$, where $\begin{cases}
        t\in\mathrm{Eq}(\mathbf{D}^b_{\mathfrak{S}_n}(N_{A'}\times A'))\\
        t_1\in\mathrm{Eq}(\mathbf{D}^b_{\mathfrak{S}_n}(N_{A'}))\\
        t_2\in\mathrm{Eq}(\mathbf{D}^b({A'})).
    \end{cases}$
    Therefore, we obtain the desired splitting \begin{equation}\label{split result of main Th}
        \widetilde{\lambda}_{q,q'}(\widetilde{\delta}_{A,A'}(f,\sigma))=(t_1\circ\Phi_1)\times(t_2\circ\Phi_2),
    \end{equation}
    with $\begin{cases}
        t_1\circ\Phi_1\in\mathrm{Eq}(\mathbf{D}^b_{\mathfrak{S}_n}(N_{A}),\mathbf{D}^b_{\mathfrak{S}_n}(N_{A'}))\simeq\mathrm{Eq}(\mathbf{D}^b(\mathrm{Kum}^{n-1}(A)),\mathbf{D}^b(\mathrm{Kum}^{n-1}(A')))\\
        t_2\circ\Phi_2\in\mathrm{Eq}(\mathbf{D}^b({A}),\mathbf{D}^b({A'})).
    \end{cases}$

    \textit{Step 3:} Prove the uniqueness of the splitting of $\widetilde{\lambda}_{q,q'}(\widetilde{\delta}_{A,A'}(f,\sigma))$.

    We start from an arbitrary splitting of $\widetilde{\lambda}_{q,q'}(\widetilde{\delta}_{A,A'}(f,\sigma))$, say $$\widetilde{\lambda}_{q,q'}(\widetilde{\delta}_{A,A'}(f,\sigma))=(\widetilde{\lambda}_{q,q'}(\widetilde{\delta}_{A,A'}(f,\sigma)))_1\times(\widetilde{\lambda}_{q,q'}(\widetilde{\delta}_{A,A'}(f,\sigma)))_2$$ with $\begin{cases}
        (\widetilde{\lambda}_{q,q'}(\widetilde{\delta}_{A,A'}(f,\sigma)))_1\in\mathrm{Eq}(\mathbf{D}^b_{\mathfrak{S}_n}(N_{A}),\mathbf{D}^b_{\mathfrak{S}_n}(N_{A'}))\simeq\mathrm{Eq}(\mathbf{D}^b(\mathrm{Kum}^{n-1}(A)),\mathbf{D}^b(\mathrm{Kum}^{n-1}(A')))\\
        (\widetilde{\lambda}_{q,q'}(\widetilde{\delta}_{A,A'}(f,\sigma)))_2\in\mathrm{Eq}(\mathbf{D}^b({A}),\mathbf{D}^b({A'})).
    \end{cases}$ It should satisfy \begin{align*}
        \gamma_{N_A,N_{A'}}((\widetilde{\lambda}_{q,q'}(\widetilde{\delta}_{A,A'}(f,\sigma)))_1)&=\gamma_{N_A,N_{A'}}(\Phi_1)=\eta_1\\
        \gamma_{A,{A'}}((\widetilde{\lambda}_{q,q'}(\widetilde{\delta}_{A,A'}(f,\sigma)))_2)&=\gamma_{N_A,N_{A'}}(\Phi_2)=\eta_2
    \end{align*} by the result in Step 1. Using the exact sequence in Theorem \ref{Orlov fundamental th}, $(\widetilde{\lambda}_{q,q'}(\widetilde{\delta}_{A,A'}(f,\sigma)))_1$ and $(\widetilde{\lambda}_{q,q'}(\widetilde{\delta}_{A,A'}(f,\sigma)))_2$ differs from $\Phi_1$ and $\Phi_2$ by translations, respectively. In other words, \begin{align*}
        (\widetilde{\lambda}_{q,q'}(\widetilde{\delta}_{A,A'}(f,\sigma)))_1&=t_1'\circ\Phi_1,\phantom{1}\exists t_1'\in N_{A'}\times\widehat{N_{A'}}\times\mathbb{Z}\\
        (\widetilde{\lambda}_{q,q'}(\widetilde{\delta}_{A,A'}(f,\sigma)))_2&=t_2'\circ\Phi_2,\phantom{1}\exists t_2'\in {A'}\times\widehat{{A'}}\times\mathbb{Z}.
    \end{align*} Since $\widetilde{\lambda}_{q,q'}(\widetilde{\delta}_{A,A'}(f,\sigma))=((t_1'\circ\Phi_1)\times(t_2'\circ\Phi_2))=(t_1'\times t_2')\circ(\Phi_1\times\Phi_2)=(t_1'\times t_2')\circ\Phi$ and $\widetilde{\lambda}_{q,q'}(\widetilde{\delta}_{A,A'}(f,\sigma))=t\circ\Phi$, we have $t=t_1'\times t_2'$.
    
    Looking back on the construction of the split in Step 2, we can see that the splitting of the generators of $t$ as in (\ref{generators of t split}) is obviously unique. Then the splitting formula $t=t_1\times t_2$  of $t$ is unique and $\begin{cases}
        t_1'=t_1\\
        t_2'=t_2.
    \end{cases}$ It gives $\begin{cases}
        \widetilde{\lambda}_{q,q'}(\widetilde{\delta}_{A,A'}(f,\sigma))_1=t_1'\circ\Phi_1=t_1\circ\Phi_1\\
        \widetilde{\lambda}_{q,q'}(\widetilde{\delta}_{A,A'}(f,\sigma))_2=t_2'\circ\Phi_2=t_2\circ\Phi_2.
    \end{cases}$ Hence, the splitting of $\widetilde{\lambda}_{q,q'}(\widetilde{\delta}_{A,A'}(f,\sigma))$ in (\ref{split result of main Th}) is unique.
    
\end{proof}

We have the derived autoequivalence version of Theorem \ref{split_NAANA'A'} as follows.
\begin{corollary}\label{split_NAA}
    Let notation be as above over an algebraically closed field of characteristic $0$. For arbitrary $(f,\sigma)\in\mathcal{G}\textrm{-}\mathrm{Aut}(\mathbf{D}^b_{\mathcal{G}}(A))$, the splitting 
     \begin{equation*}
         \widetilde{\lambda}_{q}(\widetilde{\delta}_{A}(f,\sigma))=\Phi_{(f,\sigma)}\times\Psi_{(f,\sigma)}
     \end{equation*} holds for a unique combination of $\begin{cases}
         \Phi_{(f,\sigma)}\in\mathrm{Aut}(\mathbf{D}^b(\mathrm{Kum}^{n-1}(A)))\\
         \Psi_{(f,\sigma)}\in\mathrm{Aut}(\mathbf{D}^b(A)).
     \end{cases}$

\end{corollary}

\begin{remark}
    Actually, for $(f,\sigma)=((t_a)_*,\mathrm{id})$ in Example \ref{translation split eg}, the factorization (\ref{f=t_a split}) is exactly the first formula of (\ref{generators of t split}) with $\vec{a_1'}=\vec{0},a_2'=a$. The factorization (\ref{f=otimesL split}) for $(f,\sigma)=(-\otimes L,\mathrm{id})$ in Example \ref{line bundle split eg} is the second formula of (\ref{generators of t split}) with $L_1=(L^{\boxtimes n})|_{N_A},L_2=L^n$. The factorization (\ref{f=shift split}) for $(f,\sigma)=([1],\mathrm{id})$ in Example \ref{line bundle split eg} is the last formula of (\ref{generators of t split}).
\end{remark}

\subsection{Lifting some autoequivalences of $\mathbf{D}^b(A)$ to autoequivalences of $\mathbf{D}^b(\mathrm{Kum}^{n-1}(A))$}\label{finite index n2}

In Theorem \ref{split_NAANA'A'}, we get an equivalence $\Phi_{(f,\sigma)}$ of the derived categories $\mathbf{D}^b(\mathrm{Kum}^{n-1}(A))$ and $\mathbf{D}^b(\mathrm{Kum}^{n-1}(A'))$ for an arbitrary element $(f,\sigma)$ in $\mathcal{G}\textrm{-}\mathrm{Eq}(\mathbf{D}^b_{\mathcal{G}}(A),\mathbf{D}^b_{\mathcal{G}}(A'))$. But in practice, our objective is to obtain an equivalence of the derived categories $\mathbf{D}^b(\mathrm{Kum}^{n-1}(A))$ and $\mathbf{D}^b(\mathrm{Kum}^{n-1}(A'))$ from a derived equivalence $f$ in $\mathrm{Eq}(\mathbf{D}^b(A),\mathbf{D}^b(A'))$. In this subsection, we will see that it can be done for an arbitrary $f$ in $\mathrm{Eq}(\mathbf{D}^b(A),\mathbf{D}^b(A'))$ up to index $n^2$.

We start from the following lemma.

\begin{lemma}\label{gamma_AA' G-Eq}
    Let notation be as above over an algebraically closed field of characteristic $0$. Then the set $\gamma_{A,A'}(\mathcal{G}\textrm{-}\mathrm{Eq}(\mathbf{D}^b_\mathcal{G}(A),\mathbf{D}^b_\mathcal{G}(A')))$ consists of all the symplectic maps $$g=\begin{pmatrix}
        g_1&g_2\\
        g_3&g_4
    \end{pmatrix}\in\mathrm{Sp}(A,A')$$ with $$g_1:A\to A', g_2:\widehat{A}\to A',g_3:A\to\widehat{A'},g_4:\widehat{A}\to\widehat{A'}$$
    as in \cite[Definition 4.4]{Magni:22}, such that the induced map $(g_2)_*:\mathrm{H}^1(\widehat{A},\mathbb{Z})\to\mathrm{H}^1(A',\mathbb{Z})$ in cohomology should satisfy $(g_2)_*\frac{1}{n}\mathrm{H}^1(\widehat{A},\mathbb{Z})\subset n\mathrm{H}^1(A',\mathbb{Z})$.
\end{lemma}

Actually, it describes the set $G\textrm{-}\mathrm{Sp}(A,A')$ using the correspondence in Remark \ref{two Orlov's representation remark}. It is a computable special case of subsection \ref{finite index Th_GrhoA}.

\begin{proof}
    As a special case of Theorem \ref{Th_GrhoAA'} and Corollary \ref{Lift derived eq_AA'}, we have $$\mathcal{G}\textrm{-}\rho_{A,A'}(\mathcal{G}\textrm{-}\mathrm{Eq}(\mathbf{D}^b_\mathcal{G}(A),\mathbf{D}^b_\mathcal{G}(A')))=\mathrm{res}(\mathcal{G}\textrm{-}\mathrm{SO}^+_{\mathrm{Hdg}}(V_{A},V_{A'})).$$
    
    As in Example \ref{G-SO V_B G=Ahat[n]}, the set $\mathcal{G}\textrm{-}\mathrm{SO}^+_{\mathrm{Hdg}}(V_{A},V_{A'})$ consists of all Hodge isometries in $\mathrm{SO}^+_{\mathrm{Hdg}}(V_{A},V_{A'})$ whose extension to rational coefficients maps the subset $n\mathrm{H}^1(A,\mathbb{Z})\oplus\frac{1}{n}\mathrm{H}^1(\widehat{A},\mathbb{Z})$ to $n\mathrm{H}^1(A',\mathbb{Z})\oplus\frac{1}{n}\mathrm{H}^1(\widehat{A'},\mathbb{Z})$. We find that the restriction set, $\mathrm{res}(\mathcal{G}\textrm{-}\mathrm{SO}^+_{\mathrm{Hdg}}(V_{A},V_{A'}))$, consists of all Hodge isometries in $\mathrm{SO}^+_{\mathrm{Hdg}}(V_{A},V_{A'})$ whose extension to rational coefficients maps the subset $n\mathrm{H}^1(A,\mathbb{Z})\oplus\frac{1}{n}\mathrm{H}^1(\widehat{A},\mathbb{Z})$ to $n\mathrm{H}^1(A',\mathbb{Z})\oplus\frac{1}{n}\mathrm{H}^1(\widehat{A'},\mathbb{Z})$. It gives $$\mathrm{res}(\mathcal{G}\textrm{-}\mathrm{SO}^+_{\mathrm{Hdg}}(V_{A},V_{A'}))=\mathcal{G}\textrm{-}\mathrm{SO}^+_{\mathrm{Hdg}}(V_{A},V_{A'}).$$
    So we have $$\rho_{A,A'}(\mathcal{G}\textrm{-}\mathrm{Eq}(\mathbf{D}^b_\mathcal{G}(A),\mathbf{D}^b_\mathcal{G}(A')))=\mathcal{G}\textrm{-}\mathrm{SO}^+_{\mathrm{Hdg}}(V_{A},V_{A'}).$$
    Using the relation between $\gamma_{A,A'}$ and $\rho_{A,A'}$ in Remark \ref{two Orlov's representation remark}, we get that the set $\gamma_{A,A'}(\mathcal{G}\textrm{-}\mathrm{Eq}(\mathbf{D}^b_\mathcal{G}(A),\mathbf{D}^b_\mathcal{G}(A')))$ consists of all the symplectic maps $$g=\begin{pmatrix}
        g_1&g_2\\
        g_3&g_4
    \end{pmatrix}\in\mathrm{Sp}(A,A')$$ with $$g_1:A\to A', g_2:\widehat{A}\to A',g_3:A\to\widehat{A'},g_4:\widehat{A}\to\widehat{A'}$$
    as in \cite[Definition 4.4]{Magni:22}, such that the induced map in cohomology $g_*:V_A\to V_{A'}$ is an element in $\mathcal{G}\textrm{-}\mathrm{SO}^+_{\mathrm{Hdg}}(V_A,V_{A'})$.

    In general, a symplectic map $g\in\mathrm{Sp}(A,A')$ induces $g_*|_{V_A,V_{A'}}\in\mathrm{SO}^+_\mathrm{Hdg}(V_A,V_{A'})$ by Remark \ref{two Orlov's representation remark}. That is, 
    \begin{align*}
        (g_1)_*:\mathrm{H}^1(A,\mathbb{Z})\to\mathrm{H}^1(A',\mathbb{Z}), (g_2)_*:\mathrm{H}^1(\widehat{A},\mathbb{Z})\to\mathrm{H}^1(A',\mathbb{Z}),\\
        (g_3)_*:\mathrm{H}^1(A,\mathbb{Z})\to\mathrm{H}^1(\widehat{A'},\mathbb{Z}),(g_4)_*:\mathrm{H}^1(\widehat{A},\mathbb{Z})\to\mathrm{H}^1(\widehat{A'},\mathbb{Z}).
    \end{align*}
    Therefore,
    \begin{align*}
        &g_*|_{V_A,V_{A'}}\in\mathcal{G}\textrm{-}\mathrm{SO}^+_{\mathrm{Hdg}}(V_A,V_{A'})
        \Leftrightarrow \begin{cases}
            g_*|_{V_A,V_{A'}}\in\mathrm{SO}^+_{\mathrm{Hdg}}(V_A,V_{A'})\\
            g_*(n\mathrm{H}^1(A,\mathbb{Z})\oplus\frac{1}{n}\mathrm{H}^1(\widehat{A},\mathbb{Z}))\subset n\mathrm{H}^1(A',\mathbb{Z})\oplus\frac{1}{n}\mathrm{H}^1(\widehat{A'},\mathbb{Z})
        \end{cases}\\
        \Leftrightarrow &\begin{cases}
        g_*|_{V_A,V_{A'}}\in\mathrm{SO}^+_{\mathrm{Hdg}}(V_A,V_{A'})\\
        (g_1)_*:n\mathrm{H}^1(A,\mathbb{Z})\to n\mathrm{H}^1(A',\mathbb{Z})\\
        (g_2)_*:\frac{1}{n}\mathrm{H}^1(\widehat{A},\mathbb{Z})\to n\mathrm{H}^1(A',\mathbb{Z})\\
        (g_3)_*:n\mathrm{H}^1(A,\mathbb{Z})\to \frac{1}{n}\mathrm{H}^1(\widehat{A'},\mathbb{Z})\\
        (g_4)_*:\frac{1}{n}\mathrm{H}^1(\widehat{A},\mathbb{Z})\to \frac{1}{n}\mathrm{H}^1(\widehat{A'},\mathbb{Z})
    \end{cases}
     \Leftrightarrow \begin{cases}
        g_*|_{V_A,V_{A'}}\in\mathrm{SO}^+_{\mathrm{Hdg}}(V_A,V_{A'})\\
        (g_2)_*:\frac{1}{n}\mathrm{H}^1(\widehat{A},\mathbb{Z})\to n\mathrm{H}^1(A',\mathbb{Z}).
    \end{cases}
    \end{align*}
    Here, the condition about $(g_1)_*,(g_3)_*$ and $(g_4)_*$ is guaranteed since $g_*|_{V_A,V_{A'}}$ is an element of $\mathrm{SO}^+_{\mathrm{Hdg}}(V_A,V_{A'})$.

    In sum, the set $\gamma_{A,A'}(\mathcal{G}\textrm{-}\mathrm{Eq}(\mathbf{D}^b_\mathcal{G}(A),\mathbf{D}^b_\mathcal{G}(A')))$ consists of all the symplectic maps $$g=\begin{pmatrix}
        g_1&g_2\\
        g_3&g_4
    \end{pmatrix}\in\mathrm{Sp}(A,A')$$ with $$g_1:A\to A', g_2:\widehat{A}\to A',g_3:A\to\widehat{A'},g_4:\widehat{A}\to\widehat{A'}$$
    as in \cite[Definition 4.4]{Magni:22}, such that $(g_2)_*$ maps $\frac{1}{n}\mathrm{H}^1(\widehat{A},\mathbb{Z})$ to $n\mathrm{H}^1(A',\mathbb{Z})$.
\end{proof}

By Lemma \ref{gamma_AA' G-Eq}, the set $\gamma_{A,A'}(\mathcal{G}\textit{-}\mathrm{Eq}(\mathbf{D}^b_{\mathcal{G}}(A),\mathbf{D}^b_{\mathcal{G}}(A')))$ is a subgroup of $\mathrm{Sp}(A,A')$ of index $n^2$. Using Orlov's Theorem \ref{Orlov fundamental th}, the image of the natural forgetful map $F_{n_A,n_{A'}}:\mathcal{G}\textit{-}\mathrm{Eq}(\mathbf{D}^b_{\mathcal{G}}(A),\mathbf{D}^b_{\mathcal{G}}(A'))\to\mathrm{Eq}(\mathbf{D}^b(A),\mathbf{D}^b(A'))$ in Proposition \ref{delta_AA' Eq} is an index $n^2$ subset of $\mathrm{Eq}(\mathbf{D}^b(A),\mathbf{D}^b(A'))$.

Moreover, for an arbitrary element $f\in F_{n_A,n_{A'}}(\mathcal{G}\textit{-}\mathrm{Eq}(\mathbf{D}^b_{\mathcal{G}}(A),\mathbf{D}^b_{\mathcal{G}}(A')))\subseteq\mathrm{Eq}(\mathbf{D}^b(A),\mathbf{D}^b(A'))$, we may use Lemma \ref{find sigma for overline f_0} for $\mathcal{G}=\widehat{A}[n]\cong\mathcal{G}'=\widehat{A'}[n]$ to get a set $\sigma$ of $\mathcal{G}$-equivariance natural transformations for $f$ such that $(f,\sigma)\in\mathcal{G}\textit{-}\mathrm{Eq}(\mathbf{D}^b_{\mathcal{G}}(A),\mathbf{D}^b_{\mathcal{G}}(A'))$.

So, we may summarize the main Theorem \ref{split_NAANA'A'} and Corollary \ref{split_NAA} as follows.
\begin{theorem}\label{split_NAANA'A' index n2}
     Let notation be as above over an algebraically closed field of characteristic $0$. For an arbitrary $f\in\mathrm{Eq}(\mathbf{D}^b(A),\mathbf{D}^b(A'))$, up to finite index $n^2$, the splitting 
     \begin{equation*}
         \widetilde{\lambda}_{q,q'}(\widetilde{\delta}_{A,A'}(f,\sigma))=\Phi_{(f,\sigma)}\times\Psi_{(f,\sigma)}
     \end{equation*} holds for a unique combination of $\begin{cases}
         \Phi_{(f,\sigma)}\in\mathrm{Eq}(\mathbf{D}^b(\mathrm{Kum}^{n-1}(A)),\mathbf{D}^b(\mathrm{Kum}^{n-1}(A')))\\
         \Psi_{(f,\sigma)}\in\mathrm{Eq}(\mathbf{D}^b(A),\mathbf{D}^b(A')),
     \end{cases}$
     where $\sigma$ is a set of $\mathcal{G}$-equivariance natural transformations constructed in Lemma \ref{find sigma for overline f_0} for $f$, such that $(f,\sigma)\in\mathcal{G}\textit{-}\mathrm{Eq}(\mathbf{D}^b_{\mathcal{G}}(A),\mathbf{D}^b_{\mathcal{G}}(A'))$.
\end{theorem}

\begin{corollary}\label{split_NAA index n2}
    Let notation be as above over an algebraically closed field of characteristic $0$. For an arbitrary $f\in\mathrm{Aut}(\mathbf{D}^b(A))$, up to finite index $n^2$, the splitting 
     \begin{equation*}
         \widetilde{\lambda}_{q}(\widetilde{\delta}_{A}(f,\sigma))=\Phi_{(f,\sigma)}\times\Psi_{(f,\sigma)}
     \end{equation*} holds for a unique combination of $\begin{cases}
         \Phi_{(f,\sigma)}\in\mathrm{Aut}(\mathbf{D}^b(\mathrm{Kum}^{n-1}(A)))\\
         \Psi_{(f,\sigma)}\in\mathrm{Aut}(\mathbf{D}^b(A))
     \end{cases}$,
     where $\sigma$ is a set of $\mathcal{G}$-equivariance natural transformations constructed in Lemma \ref{find sigma for overline f_0} for $f$, such that $(f,\sigma)\in\mathcal{G}\textit{-}\mathrm{Aut}(\mathbf{D}^b_{\mathcal{G}}(A))$.

\end{corollary}

\subsection{Derived equivalences between generalized Kummer varieties from Theorem \ref{split_NAANA'A'}}\label{Derived equivalences between generalized Kummer varieties from Theorem}
By Theorem \ref{split_NAANA'A'}, we get a unique equivalence $\Phi_{(f,\sigma)}$ of the derived categories $\mathbf{D}^b(\mathrm{Kum}^{n-1}(A))$ and $\mathbf{D}^b(\mathrm{Kum}^{n-1}(A'))$ for an arbitrary element $(f,\sigma)$ in $\mathcal{G}\textrm{-}\mathrm{Eq}(\mathbf{D}^b_{\mathcal{G}}(A),\mathbf{D}^b_{\mathcal{G}}(A'))$. In this way, we obtain a great deal of derived equivalences between generalized Kummer varieties. In this subsection, we will describe these equivalences via Orlov's representation.

We abuse the notation and consider the derived equivalence $\Phi_{(f,\sigma)}$ obtained as an element of $\mathrm{Eq}(\mathbf{D}^b_{\mathfrak{S}_n}(N_A),\mathbf{D}^b_{\mathfrak{S}_n}(N_{A'}))$ rather than that of $\mathrm{Eq}(\mathbf{D}^b(\mathrm{Kum}^{n-1}(A)),\mathbf{D}^b(\mathrm{Kum}^{n-1}(A'))$.
We have $\eta_1=\gamma_{N_A,N_{A'}}(\Phi_{(f,\sigma)})$ as in the proof of Proposition \ref{split_NAA}. It suffices to describe $$\eta_1=p_{N_A\times\widehat{N_A},N_{A'}\times\widehat{N_{A'}}}\circ\gamma_{N_A\times A,N_{A'}\times A'}(\widetilde{\lambda}_{q,q'}(\widetilde{\delta}_{A,A'}(f,\sigma)))=p_{N_A\times\widehat{N_A},N_{A'}\times\widehat{N_{A'}}}\circ\eta$$
for all $(f,\sigma)\in\mathcal{G}\textrm{-}\mathrm{Eq}(\mathbf{D}^b_{\mathcal{G}}(A),\mathbf{D}^b_{\mathcal{G}}(A'))$ clearly. Once done, $\Phi_{(f,\sigma)}$ can be described by Orlov's representation in Theorem \ref{Orlov fundamental th}.

\begin{proposition}\label{gamma_1_AA'}
    Let notation be as above over an algebraically closed field of characteristic $0$. Then 
    the set $p_{N_A\times\widehat{N_A},N_{A'}\times\widehat{N_{A'}}}\circ\gamma_{N_A\times A, N_{A'}\times A'}(\widetilde{\lambda}_{q,q'}(\widetilde{\delta}_{A,A'}(\mathcal{G}\textrm{-}\mathrm{Eq}(\mathbf{D}^b_\mathcal{G}(A),\mathbf{D}^b_\mathcal{G}(A')))))$ consists of all symplectic maps of the form $$\begin{pmatrix}
            \begin{pmatrix}
                g_1&&\\
                &\ddots&\\
                &&g_1
            \end{pmatrix}\Biggl|_{N_A,N_{A'}} &\begin{pmatrix}
                g_2&&\\
                &\ddots&\\
                &&g_2
            \end{pmatrix}\Biggl|_{\widehat{N_A},N_{A'}}\\
            \begin{pmatrix}
                g_3&&\\
                &\ddots&\\
                &&g_3
            \end{pmatrix}\Biggl|_{N_A,\widehat{N_{A'}}}&\begin{pmatrix}
                g_4&&\\
                &\ddots&\\
                &&g_4
            \end{pmatrix}\Biggl|_{\widehat{N_A},\widehat{N_{A'}}}
        \end{pmatrix}\in\mathrm{Sp}(N_A,N_{A'}),$$
    with $$g=\begin{pmatrix}
        g_1&g_2\\
        g_3&g_4
    \end{pmatrix}\in\mathrm{Sp}(A,A'),\phantom{1}\mathrm{where}\phantom{1}g_1:A\to A', g_2:\widehat{A}\to A',g_3:A\to\widehat{A'},g_4:\widehat{A}\to\widehat{A'},$$
    such that $(g_2)_*$ maps $\frac{1}{n}\mathrm{H}^1(\widehat{A},\mathbb{Z})$ to $n\mathrm{H}^1(A',\mathbb{Z})$.

\end{proposition}

\begin{proof}
    \textit{Step 1:}  Compute $\gamma_{A,A'}(\mathcal{G}\textrm{-}\mathrm{Eq}(\mathbf{D}^b_\mathcal{G}(A),\mathbf{D}^b_\mathcal{G}(A')))$.

    It is done in Lemma \ref{gamma_AA' G-Eq}. Actually, it describes the set $G\textrm{-}\mathrm{Sp}(A,A')$ using the correspondence in Remark \ref{two Orlov's representation remark}.

    \textit{Step 2:} Compute $\gamma_{A^n,A'^n}(\widetilde{F}_{q,q'}(\widetilde{\delta}_{A,A'}(\mathcal{G}\textrm{-}\mathrm{Eq}(\mathbf{D}^b_\mathcal{G}(A),\mathbf{D}^b_\mathcal{G}(A')))))$.
    
    For arbitrary $(f=\Phi_{\mathcal{E}},\sigma)\in\mathcal{G}\textrm{-}\mathrm{Eq}(\mathbf{D}^b_\mathcal{G}(A),\mathbf{D}^b_\mathcal{G}(A'))$, we compare $\gamma_{A,A'}(f,\sigma)$ from $F_{\mathcal{E}}$ and $\gamma_{A^n,A'^n}(\widetilde{F}_{q,q'}(\widetilde{\delta}_{A,A'}(f,\sigma)))$ from $F_{\widetilde{F}_{q,q'}(\widetilde{\delta}_{A,A'}(\Phi_{\mathcal{E}},\sigma))}$ by applying the 2-functor $\Pi$ in (\ref{def of Pi}) and the forgetful 2-functor $\mathrm{for}:\mathcal{D}^b_{\mathfrak{S}_n}(\mathcal{P}roj)\to\mathcal{D}^b(\mathcal{P}roj)$, forgetting the lifts to the $\mathfrak{S}_n$-actions, to the definition of $F_\mathcal{E}$ in (\ref{def of F_E}) as below.
    \begin{align*}
        \xymatrix{
        \mathbf{D}^b(A\times\widehat{A})\ar[rr]^{\Pi_{A\times \widehat{A}}}\ar[dd]_{\mathrm{id}\times\Phi_{\mathcal{P}_A}^{-1}}&&\mathbf{D}^b_{\mathfrak{S}_n}(A^n\times\widehat{A}^n)\ar[rr]^{\mathrm{for}_{A^n\times\widehat{A}^n}}&&\mathbf{D}^b(A^n\times\widehat{A}^n)\ar[dd]^{\mathrm{id}\times\Phi_{\mathcal{P}_{A^n}}^{-1}}\\
        \ar@{|->}[rrrr]^{\mathrm{for}\circ\Pi}&&&&\\
        \mathbf{D}^b(A\times{A})\ar[rr]^{\Pi_{A\times A}}\ar[dd]_{(\mu_A)_*}&&\mathbf{D}^b_{\mathfrak{S}_n}(A^n\times{A}^n)\ar[rr]^{\mathrm{for}_{A^n\times A^n}}&&\mathbf{D}^b(A^n\times{A}^n)\ar[dd]^{(\mu_{A^n})_*}\\
        \ar@{|->}[rrrr]^{\mathrm{for}\circ\Pi}&&&&\\
        \mathbf{D}^b(A\times{A})\ar[rr]^{\Pi_{A\times A}}\ar[dd]_{\Phi_{\mathcal{E}}\times \Phi_{\mathcal{E}_R}}&&\mathbf{D}^b_{\mathfrak{S}_n}(A^n\times{A}^n)\ar[rr]^{\mathrm{for}_{A^n\times A^n}}&&\mathbf{D}^b(A^n\times{A}^n)\ar[dd]^{\mathrm{for}(\widetilde{\delta}_{A,A'}(\Phi_{\mathcal{E}},\sigma)\times \widetilde{\delta}_{A,A'}(\Phi_{\mathcal{E}_R},\sigma_{\Phi_{\mathcal{E}_R}}=\sigma^{-1}))}\\
        \ar@{|->}[rrrr]^{\mathrm{for}\circ\Pi}&&&&\\
        \mathbf{D}^b(A'\times{A'})\ar[rr]^{\Pi_{A'\times A'}}\ar[dd]_{(\mu_{A'})^*}&&\mathbf{D}^b_{\mathfrak{S}_n}(A'^n\times{A'}^n)\ar[rr]^{\mathrm{for}_{A'^n\times A'^n}}&&\mathbf{D}^b(A'^n\times{A'}^n)\ar[dd]^{(\mu_{A'^n})^*}\\
        \ar@{|->}[rrrr]^{\mathrm{for}\circ\Pi}&&&&\\
        \mathbf{D}^b(A'\times{A'})\ar[rr]^{\Pi_{A'\times A'}}\ar[dd]_{\mathrm{id}\times\Phi_{\mathcal{P}_{A'}}}&&\mathbf{D}^b_{\mathfrak{S}_n}(A'^n\times{A'}^n)\ar[rr]^{\mathrm{for}_{A'^n\times A'^n}}&&\mathbf{D}^b(A'^n\times{A'}^n)\ar[dd]^{\mathrm{id}\times\Phi_{\mathcal{P}_{A'^n}}}\\
        \ar@{|->}[rrrr]^{\mathrm{for}\circ\Pi}&&&&\\
        \mathbf{D}^b(A'\times\widehat{A'})\ar[rr]^{\Pi_{A'\times\widehat{A'}}}&&\mathbf{D}^b_{\mathfrak{S}_n}(A'^n\times\widehat{A'}^n)\ar[rr]^{\mathrm{for}_{A'^n\times\widehat{A'}^n}}&&\mathbf{D}^b(A'^n\times\widehat{A'}^n)
        }
    \end{align*}
    Here, $\mathrm{for}\circ\Pi(\Phi_{\mathcal{P}_A})=\Phi_{\mathcal{P}_{A^n}}, \mathrm{for}\circ\Pi((\mu_A)_*)=(\mu_{A^n})_*, \mathrm{for}\circ\Pi((\mu_A)^*)=(\mu_{A^n})^*$ by the definition of the Poincar\'e line bundle and the map $\mu_A$. Analog for $A'$. Moreover, $(\widetilde{\delta}_{A,A'}(\Phi_{\mathcal{E}},\sigma))_R=\widetilde{\delta}_{A,A'}(\Phi_{\mathcal{E}_R},\sigma_{\Phi_{\mathcal{E}_R}}=\sigma^{-1})$ by the properties of $\widetilde{\delta}_{A,A'}$. The diagram above gives $\mathrm{for}\circ\Pi(F_\mathcal{E})=F_{\widetilde{F}_{q,q'}(\widetilde{\delta}_{A,A'}(\Phi_{\mathcal{E}},\sigma))}$. 
    
    Using the formulas similar to (\ref{F_E vs gamma}), we have $$(\gamma_{A^n,A'^n}(\widetilde{F}_{q,q'}(\widetilde{\delta}_{A,A'}((f,\sigma)))))_*=\mathrm{for}\circ\Pi(\gamma_{A,A'}((f,\sigma))_*).$$ 
    By the result of Step 1, the set $\gamma_{A^n,A'^n}(\widetilde{F}_{q,q'}(\widetilde{\delta}_{A,A'}(\mathcal{G}\textrm{-}\mathrm{Eq}(\mathbf{D}^b_\mathcal{G}(A),\mathbf{D}^b_\mathcal{G}(A')))))$ consists of all the symplectic maps of the form 
    $$\begin{pmatrix}
            \begin{pmatrix}
                g_1&&\\
                &\ddots&\\
                &&g_1
            \end{pmatrix} &\begin{pmatrix}
                g_2&&\\
                &\ddots&\\
                &&g_2
            \end{pmatrix}\\
            \begin{pmatrix}
                g_3&&\\
                &\ddots&\\
                &&g_3
            \end{pmatrix}&\begin{pmatrix}
                g_4&&\\
                &\ddots&\\
                &&g_4
            \end{pmatrix}
        \end{pmatrix}
        \in\mathrm{Sp}(A^n,A'^n)$$ with $$g=\begin{pmatrix}
        g_1&g_2\\
        g_3&g_4
    \end{pmatrix}\in\mathrm{Sp}(A,A'),\phantom{1}\mathrm{where}\phantom{1}g_1:A\to A', g_2:\widehat{A}\to A',g_3:A\to\widehat{A'},g_4:\widehat{A}\to\widehat{A'},$$
    such that $(g_2)_*$ maps $\frac{1}{n}\mathrm{H}^1(\widehat{A},\mathbb{Z})$ to $n\mathrm{H}^1(A',\mathbb{Z})$.

    \textit{Step 3:} Finish the proof.
    
    From (\ref{diagram_GrhoAA'}) and Remark \ref{two Orlov's representation remark}, we have the commutative diagram
    \begin{align*}
    \xymatrix{
    \widetilde{\lambda}_{q,q'}(\widetilde{\delta}_{A,A'}(\mathcal{G}\textrm{-}\mathrm{Eq}(\mathbf{D}^b_\mathcal{G}(A),\mathbf{D}^b_\mathcal{G}(A'))))\ar[rr]^(0.58){\gamma_{N_A\times A,N_{A'}\times A'}}&&\mathrm{Sp}(N_A\times A,N_{A'}\times A')\\
    \widetilde{\delta}_{A,A'}(\mathcal{G}\textrm{-}\mathrm{Eq}(\mathbf{D}^b_\mathcal{G}(A),\mathbf{D}^b_\mathcal{G}(A')))\ar[u]^{\widetilde{\lambda}_{q,q'}}\ar[d]_{\widetilde{F}_{q,q'}}\ar[rr]^(0.55){G\textrm{-}\gamma_{N_A\times A,N_{A'}\times A'}}&&\mathrm{Sp}(N_A\times A,N_{A'}\times A')\ar[d]_{\mathrm{restriction}}\ar@{=}[u]\\
    \widetilde{F}_{q,q'}(\widetilde{\delta}_{A,A'}(\mathcal{G}\textrm{-}\mathrm{Eq}(\mathbf{D}^b_\mathcal{G}(A),\mathbf{D}^b_\mathcal{G}(A'))))\ar[rr]^(0.6){\gamma_{A^n,A'^n}}&&\mathrm{Sp}(A^n, A'^n),
    }
    \end{align*}
    where $G\textrm{-}\gamma_{N_A\times A,N_{A'}\times A'}:=\gamma_{N_A\times A,N_{A'}\times A'}\circ\widetilde{\lambda}_{q,q'}$ and the restriction map comes from the lifting via $q,q'$ as below, for $g\in\mathrm{Sp}(A,A')$.
    \begin{align*}
        \xymatrix{
        N_A\times A\times\widehat{N_A}\times\widehat{A}\ar[rr]&&N_{A'}\times A'\times\widehat{N_{A'}}\times\widehat{A'}\\
        N_A\times A\times\widehat{A}^n\ar[u]^{(\mathrm{id}_{N_A\times A})\times\widehat{q}}\ar[d]_{q\times\mathrm{id}_{\widehat{A}^n}}&&N_{A'}\times A'\times\widehat{{A'}}^n\ar[u]_{(\mathrm{id}_{N_{A'}\times A'})\times\widehat{q'}}\ar[d]^{q'\times\mathrm{id}_{\widehat{A'}^n}}\\
        A^n\times\widehat{A}^n\ar[rr]_{\Delta(g)}&\ar@{=>}[uu]^{\mathrm{lift}}&A'^n\times\widehat{A'}^n
        }
    \end{align*}

    By Step 2, we get the information of the lower side. So we get \begin{align*}
        &\gamma_{N_A\times A,N_{A'}\times A'}(\widetilde{\lambda}_{q,q'}(\widetilde{\delta}_{A,A'}(\mathcal{G}\textrm{-}\mathrm{Eq}(\mathbf{D}^b_\mathcal{G}(A),\mathbf{D}^b_\mathcal{G}(A')))))\\
        =&(\gamma_{A^n,A'^n}(\widetilde{F}_{q,q'}(\widetilde{\delta}_{A,A'}(\mathcal{G}\textrm{-}\mathrm{Eq}(\mathbf{D}^b_\mathcal{G}(A),\mathbf{D}^b_\mathcal{G}(A'))))))\circ\iota_{q,\widehat{q}})|_{N_{A'}\times A'\times\widehat{N_{A'}}\times\widehat{A'}},
    \end{align*}
    where $\iota_{q,\widehat{q}}:N_{A}\times A\times\widehat{N_{A}}\times\widehat{A}\hookrightarrow A^n\times\widehat{A}^n$ is the inclusion originated from \\$q:N_{A}\times A\to A^n$ and $\widehat{q}:\widehat{N_{A}}\times\widehat{A}\to\widehat{A}^n$. Moreover, the desired $$p_{N_A\times\widehat{N_A},N_{A'}\times\widehat{N_{A'}}}(\gamma_{N_A\times A,N_{A'}\times A'}(\widetilde{\lambda}_{q,q'}(\widetilde{\delta}_{A,A'}(\mathcal{G}\textrm{-}\mathrm{Eq}(\mathbf{D}^b_\mathcal{G}(A),\mathbf{D}^b_\mathcal{G}(A'))))))$$
    can be expressed as $$(((\gamma_{A^n,A'^n}(\widetilde{F}_{q,q'}(\widetilde{\delta}_{A,A'}(\mathcal{G}\textrm{-}\mathrm{Eq}(\mathbf{D}^b_\mathcal{G}(A),\mathbf{D}^b_\mathcal{G}(A'))))))\circ\iota_{q,\widehat{q}})\circ\iota_{N_A\times\widehat{N_A}})|_{N_{A'}\times A'\times\widehat{N_{A'}}\times\widehat{A'}})|_{N_{A'}\times\widehat{N_{A'}}},$$
    where $\iota_{N_A\times \widehat{N_{A}}}:N_A\times \widehat{N_{A}}\hookrightarrow N_A\times A\times \widehat{N_{A}}\times\widehat{A}$ is the natural inclusion. Obviously, the composition of inclusions $$\iota_{q,\widehat{q}}\circ\iota_{N_A\times \widehat{N_{A}}}:N_A\times \widehat{N_{A}}\hookrightarrow N_A\times A\times \widehat{N_{A}}\times\widehat{A}\to A^n\times\widehat{A}^n$$ is just composed by the natural inclusions $N_A\hookrightarrow A^n,\widehat{N_A}\hookrightarrow\widehat{A}^n$ by definition of $N_A$. 
    So, the set $p_{N_A\times\widehat{N_A},N_{A'}\times\widehat{N_{A'}}}\circ\gamma_{N_A\times A, N_{A'}\times A'}(\widetilde{\lambda}_{q,q'}(\widetilde{\delta}_{A,A'}(\mathcal{G}\textrm{-}\mathrm{Eq}(\mathbf{D}^b_\mathcal{G}(A),\mathbf{D}^b_\mathcal{G}(A')))))$ consists of all symplectic maps of the form $$\begin{pmatrix}
            \begin{pmatrix}
                g_1&&\\
                &\ddots&\\
                &&g_1
            \end{pmatrix}\Biggl|_{N_A,N_{A'}} &\begin{pmatrix}
                g_2&&\\
                &\ddots&\\
                &&g_2
            \end{pmatrix}\Biggl|_{\widehat{N_A},N_{A'}}\\
            \begin{pmatrix}
                g_3&&\\
                &\ddots&\\
                &&g_3
            \end{pmatrix}\Biggl|_{N_A,\widehat{N_{A'}}}&\begin{pmatrix}
                g_4&&\\
                &\ddots&\\
                &&g_4
            \end{pmatrix}\Biggl|_{\widehat{N_A},\widehat{N_{A'}}}
        \end{pmatrix}\in\mathrm{Sp}(N_A,N_{A'}),$$
    with $$g=\begin{pmatrix}
        g_1&g_2\\
        g_3&g_4
    \end{pmatrix}\in\mathrm{Sp}(A,A'),\phantom{1}\mathrm{where}\phantom{1}g_1:A\to A', g_2:\widehat{A}\to A',g_3:A\to\widehat{A'},g_4:\widehat{A}\to\widehat{A'},$$
    such that $(g_2)_*$ maps $\frac{1}{n}\mathrm{H}^1(\widehat{A},\mathbb{Z})$ to $n\mathrm{H}^1(A',\mathbb{Z})$.
\end{proof}

The proposition gives a criterion of derived equivalences between generalized Kummer varieties that can be lifted from the $\mathcal{G}$-functors, which are derived equivalences between $\mathbf{D}^b_{\mathcal{G}}(A)$ and $\mathbf{D}^b_{\mathcal{G}}(A')$. Here we abuse the notation and regard derived equivalences in $\mathrm{Eq}(\mathbf{D}^b(\mathrm{Kum}^{n-1}(A)), \mathbf{D}^b(\mathrm{Kum}^{n-1}(A')))$ as derived equivalences in $\mathrm{Eq}(\mathbf{D}^b_{\mathfrak{S}_n}(N_A), \mathbf{D}^b_{\mathfrak{S}_n}(N_{A'}))$ by conjugation via the BKR isomorphisms.

\begin{theorem}\label{describe derived equiv NANA'}
    Let notation be as above over an algebraically closed field over characteristic $0$. A derived equivalence $\Phi\in\mathrm{Eq}(\mathbf{D}^b_{\mathfrak{S}_n}(N_A), \mathbf{D}^b_{\mathfrak{S}_n}(N_{A'}))$ is of the form $\Phi_{(f,\sigma)}$ as in Theorem \ref{split_NAANA'A'} if and only if $\gamma_{N_A,N_{A'}}(\Phi)$ is a symplectic map of the form $$\begin{pmatrix}
            \begin{pmatrix}
                g_1&&\\
                &\ddots&\\
                &&g_1
            \end{pmatrix}\Biggl|_{N_A,N_A} &\begin{pmatrix}
                g_2&&\\
                &\ddots&\\
                &&g_2
            \end{pmatrix}\Biggl|_{\widehat{N_A},N_A}\\
            \begin{pmatrix}
                g_3&&\\
                &\ddots&\\
                &&g_3
            \end{pmatrix}\Biggl|_{N_A,\widehat{N_A}}&\begin{pmatrix}
                g_4&&\\
                &\ddots&\\
                &&g_4
            \end{pmatrix}\Biggl|_{\widehat{N_A},\widehat{N_A}}
        \end{pmatrix}\in\mathrm{Sp}(N_A,N_{A'}),$$
    with $$g=\begin{pmatrix}
        g_1&g_2\\
        g_3&g_4
    \end{pmatrix}\in\mathrm{Sp}(A,A'),\phantom{1}\mathrm{where}\phantom{1}g_1:A\to A', g_2:\widehat{A}\to A',g_3:A\to\widehat{A'},g_4:\widehat{A}\to\widehat{A'},$$
    such that $(g_2)_*$ maps $\frac{1}{n}\mathrm{H}^1(\widehat{A},\mathbb{Z})$ to $n\mathrm{H}^1(A',\mathbb{Z})$.
\end{theorem}

Analogous to the derived autoequivalence case.
\begin{corollary}\label{describe derived equiv NA}
    Let notation be as above over an algebraically closed field of characteristic $0$. Then 
    the set $p_{N_A\times\widehat{N_A},N_A\times\widehat{N_A}}\circ\gamma_{N_A\times A}(\widetilde{\lambda}_{q}(\widetilde{\delta}_{A}(\mathcal{G}\textrm{-}\mathrm{Aut}(\mathbf{D}^b_\mathcal{G}(A)))))$ consists of all symplectic maps of the form $$\begin{pmatrix}
            \begin{pmatrix}
                g_1&&\\
                &\ddots&\\
                &&g_1
            \end{pmatrix}\Biggl|_{N_A,N_{A}} &\begin{pmatrix}
                g_2&&\\
                &\ddots&\\
                &&g_2
            \end{pmatrix}\Biggl|_{\widehat{N_A},N_{A}}\\
            \begin{pmatrix}
                g_3&&\\
                &\ddots&\\
                &&g_3
            \end{pmatrix}\Biggl|_{N_A,\widehat{N_{A}}}&\begin{pmatrix}
                g_4&&\\
                &\ddots&\\
                &&g_4
            \end{pmatrix}\Biggl|_{\widehat{N_A},\widehat{N_{A}}}
        \end{pmatrix}\in\mathrm{Sp}(N_A),$$
    with $$g=\begin{pmatrix}
        g_1&g_2\\
        g_3&g_4
    \end{pmatrix}\in\mathrm{Sp}(A),\phantom{1}\mathrm{where}\phantom{1}g_1:A\to A, g_2:\widehat{A}\to A,g_3:A\to\widehat{A},g_4:\widehat{A}\to\widehat{A},$$
    such that $(g_2)_*$ maps $\frac{1}{n}\mathrm{H}^1(\widehat{A},\mathbb{Z})$ to $n\mathrm{H}^1(A,\mathbb{Z})$.

    Moreover, a derived autoequivalence $\Phi\in\mathrm{Aut}(\mathbf{D}^b_{\mathfrak{S}_n}(N_A))$ is of the form $\Phi_{(f,\sigma)}$ as in Corollary \ref{split_NAA} if and only if $\gamma_{N_A}(\Phi)$ is a symplectic map of the form $$\begin{pmatrix}
            \begin{pmatrix}
                g_1&&\\
                &\ddots&\\
                &&g_1
            \end{pmatrix}\Biggl|_{N_A,N_A} &\begin{pmatrix}
                g_2&&\\
                &\ddots&\\
                &&g_2
            \end{pmatrix}\Biggl|_{\widehat{N_A},N_A}\\
            \begin{pmatrix}
                g_3&&\\
                &\ddots&\\
                &&g_3
            \end{pmatrix}\Biggl|_{N_A,\widehat{N_A}}&\begin{pmatrix}
                g_4&&\\
                &\ddots&\\
                &&g_4
            \end{pmatrix}\Biggl|_{\widehat{N_A},\widehat{N_A}}
        \end{pmatrix}\in\mathrm{Sp}(N_A),$$
    with $$g=\begin{pmatrix}
        g_1&g_2\\
        g_3&g_4
    \end{pmatrix}\in\mathrm{Sp}(A),\phantom{1}\mathrm{where}\phantom{1}g_1:A\to A, g_2:\widehat{A}\to A,g_3:A\to\widehat{A},g_4:\widehat{A}\to\widehat{A},$$
    such that $(g_2)_*$ maps $\frac{1}{n}\mathrm{H}^1(\widehat{A},\mathbb{Z})$ to $n\mathrm{H}^1(A,\mathbb{Z})$.
\end{corollary}

\subsection{Derived autoequivalences of Kummer K3 surfaces}
In this subsection, we give an introspection of the derived autoequivalences of Kummer K3 surfaces.

\begin{remark}\label{Kummer K3 remark 1}
    By the BKR equivalence, we have $\mathbf{D}^b_{\mathfrak{S}_n}(N_A)\simeq\mathbf{D}^b(\mathrm{Kum}^{n-1}(A))$ using the natural action of $\mathfrak{S}_n$ on $N_A$. In particular, if $n=2$, we have the derived category of a Kummer K3 surface being $\mathbf{D}^b(\mathrm{Kum}^{1}(A))\simeq\mathbf{D}^b_{\mathfrak{S}_2}(N_A)\simeq\mathbf{D}^b_{\mathfrak{S}_2}(A)$. Here, we use isomorphisms $i_A:A\cong N_A,\widehat{i_A}:\widehat{N_A}\cong\widehat{A}$ for $n=2$, which come from $$N_A=\mathrm{ker}(\Sigma:A^2\to A)=\{(a,-a)|a\in A\}\cong A.$$ Using these isomorphisms, we get a series of derived autoequivalences $\Phi_{(f,\sigma)}=\Phi$ in $\mathrm{Aut}(\mathbf{D}^b_{\mathfrak{S}_2}(N_A))\simeq\mathrm{Aut}(\mathbf{D}^b(\mathrm{Kum}^{1}(A)))$ by lifting arbitrary $(f,\sigma)\in\mathcal{G}\textrm{-}\mathrm{Aut}(\mathbf{D}^b_{\mathcal{G}}(A)))$ with $\mathcal{G}=\widehat{A}[2]$ from Corollary \ref{split_NAA}. Moreover, as in Corollary \ref{describe derived equiv NA}, a derived autoequivalence $\Phi\in\mathrm{Aut}(\mathbf{D}^b_{\mathfrak{S}_2}(N_A))\simeq\mathrm{Aut}(\mathbf{D}^b_{\mathfrak{S}_2}(A))$ is of the form $\Phi_{(f,\sigma)}$ if and only if $\gamma_{N_A}(\Phi)$
    is a symplectic map of the form $$\begin{pmatrix}
            \begin{pmatrix}
                g_1&\\
                
                &g_1
            \end{pmatrix}\Biggl|_{N_A,N_A} &\begin{pmatrix}
                g_2&\\
                
                &g_2
            \end{pmatrix}\Biggl|_{\widehat{N_A},N_A}\\
            \begin{pmatrix}
                g_3&\\
                
                &g_3
            \end{pmatrix}\Biggl|_{N_A,\widehat{N_A}}&\begin{pmatrix}
                g_4&\\
                
                &g_4
            \end{pmatrix}\Biggl|_{\widehat{N_A},\widehat{N_A}}
        \end{pmatrix}\in\mathrm{Sp}(N_A),$$
    with $$g=\begin{pmatrix}
        g_1&g_2\\
        g_3&g_4
    \end{pmatrix}\in\mathrm{Sp}(A),\phantom{1}\mathrm{where}\phantom{1}g_1:A\to A, g_2:\widehat{A}\to A,g_3:A\to\widehat{A},g_4:\widehat{A}\to\widehat{A},$$
    such that $(g_2)_*$ maps $\frac{1}{2}\mathrm{H}^1(\widehat{A},\mathbb{Z})$ to $2\mathrm{H}^1(A,\mathbb{Z})$.

    Regarding $\Phi$ as an autoequivalence of $\mathbf{D}^b_{\mathfrak{S}_2}(A)$ via $i_A,\widehat{i_A}$, we get $\gamma_A(\Phi)$ being a map of the form $$(i_A^{-1}\times\widehat{i_A})\circ\begin{pmatrix}
            \begin{pmatrix}
                g_1&\\
                
                &g_1
            \end{pmatrix}\Biggl|_{N_A,N_A} &\begin{pmatrix}
                g_2&\\
                
                &g_2
            \end{pmatrix}\Biggl|_{\widehat{N_A},N_A}\\
            \begin{pmatrix}
                g_3&\\
                
                &g_3
            \end{pmatrix}\Biggl|_{N_A,\widehat{N_A}}&\begin{pmatrix}
                g_4&\\
                
                &g_4
            \end{pmatrix}\Biggl|_{\widehat{N_A},\widehat{N_A}}
        \end{pmatrix}\circ(i_A\times\widehat{i_A}^{-1})=\begin{pmatrix}
        g_1&g_2\\
        g_3&g_4
    \end{pmatrix}=g\in\mathrm{Sp}(A),$$ such that $(g_2)_*$ maps $\frac{1}{2}\mathrm{H}^1(\widehat{A},\mathbb{Z})$ to $2\mathrm{H}^1(A,\mathbb{Z})$.

    To conclude, we can see that the set of the derived autoequivalences $$\Phi_{(f,\sigma)}\in\mathrm{Aut}(\mathbf{D}^b(\mathrm{Kum}^1(A)))\simeq\mathrm{Aut}(\mathbf{D}^b_{\mathfrak{S}_2}(N_A))\simeq\mathrm{Aut}(\mathbf{D}^b_{\mathfrak{S}_2}(A))$$ obtained from Corollary \ref{split_NAA} is only a proper subset of $\mathrm{Aut}(\mathbf{D}^b_{\mathfrak{S}_2}(A))$. To be concrete, a derived autoequivalence $\Phi\in\mathrm{Aut}(\mathbf{D}^b_{\mathfrak{S}_2}(A))\simeq\mathrm{Aut}(\mathbf{D}^b(\mathrm{Kum}^1(A)))$ is of the form $\Phi_{(f,\sigma)}$ as in Corollary \ref{split_NAA} if and only if $\gamma_{A}(\Phi)$ is a symplectic map of the form $g=\begin{pmatrix}
        g_1&g_2\\
        g_3&g_4
    \end{pmatrix}\in\mathrm{Sp}(A)$ with $g_1:A\to A, g_2:\widehat{A}\to A,g_3:A\to\widehat{A},g_4:\widehat{A}\to\widehat{A}$
    such that $(g_2)_*$ maps $\frac{1}{2}\mathrm{H}^1(\widehat{A},\mathbb{Z})$ to $2\mathrm{H}^1(A,\mathbb{Z})$.
\end{remark}

Note that $\mathbf{D}^b(\mathrm{Kum}^1(A))\simeq\mathbf{D}^b_{\mathfrak{S}_2}(A)$ is not equivalent to $\mathbf{D}^b(A)$. Indeed, the algebraic $K$-group of $\mathrm{Kum}^1(A)$, $\mathcal{K}^*(\mathrm{Kum}^1(A))\simeq\mathcal{K}^*_{\mathfrak{S}_2}(A)$, with a generator associated to each of the 16 points of order 2 in $A$, is a larger lattice than $\mathcal{K}^*(A)$. (See \cite[subsection 10.2]{BKR:01}.)

\begin{remark}\label{Kummer K3 remark 2}
    On the other hand, the derived category $\mathbf{D}^b(\mathrm{Kum}^1(A))$ is equivalent to the $\mathfrak{S}_2$-equivariant category $\mathbf{D}^b_{\mathfrak{S}_2}(A)$, where the group $\mathfrak{S}_2$ acts by involution $(-1)_*$ by the BKR isomorphism in \cite[subsection 10.2]{BKR:01}.
    
    To be precise, the group $\mathfrak{S}_2=\{1,-1\}$ acts on $A$, where $1\in\mathfrak{S}_2$ corresponds to $\mathrm{id}_A$ and $-1\in\mathfrak{S}_2$ corresponds to the multiplicity automorphism by $-1$, $-1:A\to A$, of order 2. It induces an action $(\rho^{\mathfrak{S}_2},\theta^{\mathfrak{S}_2})$ of $\mathfrak{S}_2$ on the derived category $\mathbf{D}^b(A)$ as in Definition \ref{def of group action on category}, with $$\rho_1^{\mathfrak{S}_2}=\mathrm{id}_{\mathbf{D}^b(A)},\rho_{-1}^{\mathfrak{S}_2}=(-1)^*:\mathbf{D}^b(A)\to\mathbf{D}^b(A),\theta^{\mathfrak{S}_2}=\mathrm{id}.$$
    This action is unique up to equivalence by \cite[Theorem 1.11]{Bayer Perry:23} and $\mathrm{H}^2(\mathfrak{S}_2,k^*)=0$. As in \cite[subsection 3.3]{Ploog:05}, we have an induced action $(\rho^{\mathfrak{S}_{2,\Delta}},\theta^{\mathfrak{S}_{2,\Delta}})$ of the diagonal group $\mathfrak{S}_{2,\Delta}$ on the derived category $\mathbf{D}^b(A\times A)$, with $$\rho_{(1,1)}^{\mathfrak{S}_{2,\Delta}}=\mathrm{id}_{\mathbf{D}^b(A\times A)},\rho_{(-1,-1)}^{\mathfrak{S}_{2,\Delta}}=(-1,-1)^*:\mathbf{D}^b(A\times A)\to\mathbf{D}^b(A\times A),\theta^{\mathfrak{S}_{2,\Delta}}=\mathrm{id}.$$
    Also, we have an induced action $(\rho^{\mathfrak{S}^2_{2}},\theta^{\mathfrak{S}_{2}^2})$ of the product group $\mathfrak{S}_{2}^2=\mathfrak{S}_2\times\mathfrak{S}_2$ on $\mathbf{D}^b(A\times A)$, with $\rho_{(h_1,h_2)}^{\mathfrak{S}_{2}^2}=(h_1,h_2)^*\in\mathrm{Aut}(\mathbf{D}^b(A\times A)),\forall(h_1,h_2)\in\mathfrak{S}_2^2$ and $\theta^{\mathfrak{S}_{2}^2}=\mathrm{id}$. For the finite abelian group $\mathfrak{S}_{2,\Delta}$, we see that the action above is unique up to equivalence by \cite[Theorem 1.11]{Bayer Perry:23} and $\mathrm{H}^2(\mathfrak{S}_{2,\Delta},k^*)=0$. But it is not true for the group $\mathfrak{S}_{2}^2\cong(\mathbb{Z}/2\mathbb{Z})^2$ in general by \cite[Theorem 1.11]{Bayer Perry:23}.%
    \footnote{For example, $\mathrm{H}^2((\mathbb{Z}/2\mathbb{Z})^2,\mathbb{C}^*)=\mathbb{Z}/2\mathbb{Z}\not=0$ as in \cite[Chapter 25]{Huppert:67}.}
    
    We may find derived equivalences in $\mathrm{Aut}(\mathbf{D}^b_{\mathfrak{S}_2}(A))$ using the following two group homomorphisms from \cite[Proposition 3.18 or subsection 4.3]{Ploog:05}. Here, the two group homomorphisms are $2:1$ and the map $for$ is surjective.
    \begin{align*}
        inf:&\mathrm{Aut}^{\mathfrak{S}_{2,\Delta}}(\mathbf{D}^b(A))\to\mathrm{Aut}(\mathbf{D}^b_{\mathfrak{S}_2}(A))\\
        for:&\mathrm{Aut}^{\mathfrak{S}_{2,\Delta}}(\mathbf{D}^b(A))\to\mathrm{Aut}(\mathbf{D}^b(A))^{\mathfrak{S}_{2,\Delta}}
    \end{align*}

    The related notations in \cite[subsections 3.2 and 3.3]{Ploog:05} are as follows.
    \begin{enumerate}
        \item The groups above are defined to be \begin{align*}
        \mathrm{Aut}(\mathbf{D}^b(A))^{\mathfrak{S}_{2,\Delta}}&:=\{\Phi_E\in\mathrm{Aut}(\mathbf{D}^b(A))|h^*\circ\Phi_E=\Phi_E\circ h^*,\forall h\in\mathfrak{S}_2\}\\
        &\phantom{:}=\{E\in\mathbf{D}^b(A\times A)|(h,h)^*E\cong E,\forall h\in\mathfrak{S}_2\}\\
        \mathrm{Aut}^{\mathfrak{S}_{2,\Delta}}(\mathbf{D}^b(A))&:=\{(E,\phi_E)\in\mathbf{D}^b_{\mathfrak{S}_{2,\Delta}}(A\times A)|\Phi_E\in\mathrm{Aut}(\mathbf{D}^b(A))\},
    \end{align*} where $\Phi_E$ is the Fourier-Mukai transform with kernel $E$.
        \item The forgetful morphism ``$for$" maps an element%
        \footnote{Note that for arbitrary $E\in\mathrm{Aut}(\mathbf{D}^b(A))^{\mathfrak{S}_{2,\Delta}}$, we have exactly 2 families of isomorphisms $\phi_E^1,\phi_E^2$ such that $for((E,\phi_E^1))=for((E,\phi_E^2))=E$, since the group homomorphism ``$for$" is $2:1$ and surjective. To be precise, $\phi_{E,(h,h)}:E\xrightarrow{\cong}\rho^{\mathfrak{S}_{2,\Delta}}_{(h,h)}E,\forall h\in\mathfrak{S}_2$ is either $\begin{cases}
            \phi_{E,(1,1)}^1=\mathrm{id}_E\\
            \phi_{E,(-1,-1)}^1=(-1,-1)^*
        \end{cases}$ or $\begin{cases}
            \phi_{E,(1,1)}^2=\mathrm{id}_E\\
            \phi_{E,(-1,-1)}^2=\omega^*\circ(-1,-1)^*\circ\omega^*
        \end{cases}$, where the derived autoequivalence $\omega^*\in\mathrm{Aut}(\mathbf{D}^b(A\times A))$ is induced by the permutation isomorphism $\omega:A\times A\to A\times A$.\label{phi12}}%
        $(E,\phi_E)\in\mathrm{Aut}^{\mathfrak{S}_{2,\Delta}}(\mathbf{D}^b(A))$ to $E\in\mathrm{Aut}(\mathbf{D}^b(A))^{\mathfrak{S}_{2,\Delta}}$.
        \item We define the \textit{left inflation} of $(E,\phi_E)\in\mathrm{Aut}^{\mathfrak{S}_{2,\Delta}}(\mathbf{D}^b(A))$ as $$\mathfrak{S}_2\cdot(E,\phi_E):=(\mathfrak{S}_2\cdot E,\phi_{\mathfrak{S}_2\cdot E}).$$ It is made up of the following.
        \begin{enumerate}[(i)]
            \item The \textit{left inflation} of $E$ is defined as $$\mathfrak{S}_2\cdot E:=\underset{h\in\mathfrak{S}_2}{\oplus}(h,1)^*E=E\oplus(-1,1)^*E.$$
            \item The groups of isomorphisms is $\phi_{\mathfrak{S}_2\cdot E}=(\phi_{\mathfrak{S}_2\cdot E,(h_1,h_2)})_{(h_1,h_2)\in\mathfrak{S}_2^2}$. The isomorphism $\phi_{\mathfrak{S}_2\cdot E,(h_1,h_2)}$ fits the following commutative diagram, for arbitrary $(h_1,h_2)\in\mathfrak{S}_2^2$.%
            \footnote{Here, the permutation of summand is identity if $h_1=h_2$. It is $\omega^*$ induced by the permutation isomorphism if $h_1\not=h_2$. After computation, we get $\begin{cases}
                \phi^1_{\mathfrak{S}_2\cdot E,(1,1)}=\mathrm{id}_{\mathfrak{S}_2\cdot E}\\
                \phi^1_{\mathfrak{S}_2\cdot E,(1,-1)}=(-1,-1)^*\circ\omega^*\\
                \phi^1_{\mathfrak{S}_2\cdot E,(-1,1)}=\omega^*\\
                \phi^1_{\mathfrak{S}_2\cdot E,(-1,-1)}=(-1,-1)^*
            \end{cases}$ and $\begin{cases}
                \phi^2_{\mathfrak{S}_2\cdot E,(1,1)}=\mathrm{id}_{\mathfrak{S}_2\cdot E}\\
                \phi^2_{\mathfrak{S}_2\cdot E,(1,-1)}=\omega^*\circ(-1,-1)^*\\
                \phi^2_{\mathfrak{S}_2\cdot E,(-1,1)}=\omega^*\\
                \phi^2_{\mathfrak{S}_2\cdot E,(-1,-1)}=\omega^*\circ(-1,-1)^*\circ\omega^*
            \end{cases}$ using the choices of two families of isomorphisms $\phi^1$ and $\phi^2$ in the footnote \ref{phi12}, respectively.}%
            \begin{align*}
                \xymatrix{
                \mathfrak{S}_2\cdot E\ar[rrr]^(0.4){\textrm{permutation of summands}}_(0.4){\cong}\ar[d]^{\phi_{\mathfrak{S}_2\cdot E,(h_1,h_2)}}_{\cong}&&&\underset{h\in\mathfrak{S}_2}{\oplus}(h_2^{-1}hh_1,1)^*E\ar[rrr]^(0.45){\underset{g\in\mathfrak{S}_2}{\oplus}\phi_{(h_2^{-1}hh_1,1)^*E,(h_2,h_2)}}_(0.45){\cong}&&&\underset{h\in\mathfrak{S}_2}{\oplus}(h_2^{-1}hh_1,1)^*(h_2,h_2)^*E\ar@{=}[d]\\
                (h_1,h_2)^*\mathfrak{S}_2\cdot E\ar@{=}[rrr]&&&\underset{h\in\mathfrak{S}_2}{\oplus}(h_1,h_2)^*(h,1)^*E\ar@{=}[rrr]&&&\underset{h\in\mathfrak{S}_2}{\oplus}(hh_1,h_2)^*E
                }
            \end{align*}
        \end{enumerate}
        \item The inflation homomorphism ``$inf$" maps an element $(E,\phi_E)\in\mathrm{Aut}^{\mathfrak{S}_{2,\Delta}}(\mathbf{D}^b(A))$ to $\Phi_{(E,\phi_E)}^{\mathfrak{S}_2}:=\Phi_{\mathfrak{S}_2\cdot(E,\phi_E)}\in\mathrm{Aut}(\mathbf{D}^b_{\mathfrak{S}_2}(A))$. Here, $\Phi_{\mathfrak{S}_2\cdot(E,\phi_E)}$ is the Fourier-Mukai transform in the linearized setting with kernel $\mathfrak{S}_2\cdot(E,\phi_E)$ as in \cite[subsection 3.2]{Ploog:05}. 
    \end{enumerate}
    
    To conclude, we get some derived autoequivalence of the Kummer K3 surface $\mathrm{Kum}^1(A)$, which are of the form $\Phi_{(E,\phi_E)}^{\mathfrak{S}_2}\in\mathrm{Aut}(\mathbf{D}^b_{\mathfrak{S}_2}(A))$ for
    $$E\in\mathbf{D}^b(A\times A)\textrm{ with }\Phi_E\in\mathrm{Aut}(\mathbf{D}^b(A)),(-1,-1)^*E\cong E\textrm{ and }(E,\phi_E)\in\mathbf{D}^b_{\mathfrak{S}_{2,\Delta}}(A\times A).$$
    
\end{remark}

\begin{remark}\label{Kummer K3 remark 3}
    Inspired by the thought in Remark \ref{Kummer K3 remark 1}, one may try to use Orlov's representation to describe the derived autoequivalences $\Phi_{(E,\phi_E)}^{\mathfrak{S}_2}\in\mathrm{Aut}(\mathbf{D}^b_{\mathfrak{S}_2}(A))$ obtained from Remark \ref{Kummer K3 remark 2}. But this method is useless.

    We first try to describe the group $\mathrm{Aut}(\mathbf{D}^b(A))^{\mathfrak{S}_{2,\Delta}}$ by Orlov's representation. By definition, it is made up of all $E\in\mathbf{D}^b(A\times A)$ with $\Phi_E\in\mathrm{Aut}(\mathbf{D}^b(A))$ and $(-1,-1)^*E\cong E$. Regarding $\mathrm{Aut}(\mathbf{D}^b(A))^{\mathfrak{S}_{2,\Delta}}$ as a subgroup of $\mathrm{Aut}(\mathbf{D}^b(A))$, we have the following exact sequence of group homomorphisms by \cite[Proposition 4.8]{Ploog:05}.
    $$0\to A[2]\times\widehat{A}[2]\times\mathbb{Z}\to\mathrm{Aut}(\mathbf{D}^b(A))^{\mathfrak{S}_{2,\Delta}}\xrightarrow{\gamma_A}\mathrm{Sp}(A)\to 0$$
    Here, we still use the notation $\gamma_A$ to represent the restriction of Orlov's representation $\gamma_A:\mathrm{Aut}(\mathbf{D}^b(A))\to\mathrm{Sp}(A)$ to $\mathrm{Aut}(\mathbf{D}^b(A))^{\mathfrak{S}_{2,\Delta}}$. This shows that Orlov's representation is useless for describing $\mathrm{Aut}(\mathbf{D}^b(A))^{\mathfrak{S}_{2,\Delta}}$, since the homomorphism $\gamma_A$ in the exact sequence above is surjective. So, it is not a good strategy to use Orlov's representation to describe the derived autoequivalences $\Phi_{(E,\phi_E)}^{\mathfrak{S}_2}\in\mathrm{Aut}(\mathbf{D}^b_{\mathfrak{S}_2}(A))$ obtained from Remark \ref{Kummer K3 remark 2}.%
    \footnote{Actually, the derived equivalences $\Phi_{(E,\phi_E)}^{\mathfrak{S}_2}$ obtained from Remark \ref{Kummer K3 remark 2} satisfies the condition (\ref{describe phiS2EphiE}).
    \par
    By \cite[Exercise 5.12 $ii)$ and $iii)$]{Huybrechts:06}, we have
    \begin{align*}
        \Phi_{\mathfrak{S}_2\cdot E}\circ(-1)_*&=\Phi_{E\oplus(-1,1)^*E}\circ(-1)_*=\Phi_{(-1,1)^*(E\oplus(-1,1)^*E)}=\Phi_{(-1,1)^*E\oplus E}=\Phi_{\mathfrak{S}_2\cdot E}\\
        (-1)^*\circ\Phi_{\mathfrak{S}_2\cdot E}&=(-1)^*\circ\Phi_{E\oplus(-1,1)^*E}=\Phi_{(1,-1)^*(E\oplus(-1,1)^*E)}=\Phi_{(1,-1)^*E\oplus (-1,-1)^*E}\\
        &=\Phi_{(-1,1)^*E\oplus E}=\Phi_{\mathfrak{S}_2\cdot E},
    \end{align*} since $\Phi_{\mathfrak{S}_2\cdot E}=\Phi_{E\oplus(-1,1)^*E}$ and $(-1,-1)^*E\cong E$. 
    \par
    Consequently, a derived autoequivalence $\Phi\in\mathrm{Aut}(\mathbf{D}^b_{\mathfrak{S}_2}(A))\simeq\mathrm{Aut}(\mathbf{D}^b(\mathrm{Kum}^1(A)))$ is of the form $\Phi^{\mathfrak{S}_2}_{(E,\phi_E)}$ as in Remark \ref{Kummer K3 remark 2}, then \begin{equation}\label{describe phiS2EphiE}
        \Phi\circ(-1)_*=\Phi=(-1)^*\circ\Phi
\end{equation} if we regard $\Phi$ as an element of $\mathrm{Aut}(\mathbf{D}^b(A))$. Note that, (\ref{describe phiS2EphiE}) is only a necessary condition for $\Phi$ being of the form $\Phi^{\mathfrak{S}_2}_{(E,\phi_E)}$ as in Remark \ref{Kummer K3 remark 2}, by the above reasoning.}%

    In Remark \ref{Kummer K3 remark 1}, we have a derived autoequivalence $\Phi\in\mathrm{Aut}(\mathbf{D}^b_{\mathfrak{S}_2}(A))\simeq\mathrm{Aut}(\mathbf{D}^b(\mathrm{Kum}^1(A)))$ of the form $\Phi_{(f,\sigma)}$ as in Corollary \ref{split_NAA} if and only if $$\gamma_A(\Phi)=\begin{pmatrix}
            g_1&g_2\\g_3&g_4
        \end{pmatrix}\in\mathrm{Sp}(A)\textrm{ with }(g_2)_*(\mathrm{\frac{1}{2}}\mathrm{H}^1(\widehat{A},\mathbb{Z}))\subset 2\mathrm{H}^1(A,\mathbb{Z}).$$
    So, the method of getting the derived autoequivalences of a Kummer K3 surface in Remark \ref{Kummer K3 remark 2} is totally different from the one in Corollary \ref{split_NAA}.
\end{remark}

We may end this subsection to investigate whether Corollary \ref{split_NAA} gives all possible lifts from $\mathrm{Aut}(\mathbf{D}^b(A))$ to $\mathrm{Aut}(\mathbf{D}^b(\mathrm{Kum}^1(A)))$ from the aspect of even cohomology of the abelian surface $A$.

First, we give a description of derived autoequivalences in $\mathrm{Aut}(\mathbf{D}^b(\mathrm{Kum}^1(A)))$ lifted from $\mathrm{Aut}(\mathbf{D}^b(A))$ by a subgroup $G_A$ as below.

\begin{passage}\label{G_A for lifting 1}
    Using the notation in \cite[section VIII.5]{Barth Hulek Peters Ven:15}, the definition of a Kummer K3 surface $\mathrm{Kum}^1(A)$ is from the following commutative diagram \begin{align*}
        \xymatrix{
           \widetilde{A}\ar[rr]^{\sigma}\ar[d]_{\widetilde{p}}&&A\ar[d]^{{p}}\\
           \mathrm{Kum}^1(A)\ar[rr]&&A/\langle 1_A,-1_A\rangle
        }
    \end{align*}
    where $p,\widetilde{p}$ are quotient maps and $\sigma$ is the map blowing up the fixed points. For $a\in A[2]$, we have $e_i:=\widetilde{p}(\sigma^{-1}(a))$ and $\mathbb{Z}^W:=\overset{16}{\underset{i=1}{\oplus}}\mathbb{Z}e_i$. 
    
    We have a primitive embedding $\alpha:=\widetilde{p}_*\circ\sigma^*:\mathrm{H}^2(A,\mathbb{Z})\to\mathrm{H}^2(\mathrm{Kum}^1(A),\mathbb{Z})$, with $(\alpha(x),\alpha(y))=2(x,y)$ for all $x,y\in\mathrm{H}^2(A,\mathbb{Z})$, by \cite[Proposition VIII.5.1]{Barth Hulek Peters Ven:15}. The image of $\alpha$ is $(\mathbb{Z}^W)^{\perp}\subset\mathrm{H}^2(\mathrm{Kum}^1(A),\mathbb{Z})$ by \cite[Corollary VIII.5.6]{Barth Hulek Peters Ven:15}. Moreover, the compactification $\alpha_{\mathbb{C}}$ satisfies $\alpha_{\mathbb{C}}(\mathrm{H}^{0,2}(A))=\mathrm{H}^{0,2}(\mathrm{Kum}^1(A))$ by \cite[Proposition VIII.5.2]{Barth Hulek Peters Ven:15}. Using the projection formula, the monomorphism property of $\alpha$ gives $$\widetilde{p}_*\circ\sigma^*:\mathbb{Z}\simeq\mathrm{H}^0(A,\mathbb{Z})\to\mathrm{H}^0(\mathrm{Kum}^1(A),\mathbb{Z})\simeq\mathbb{Z},1\mapsto 2.$$ But the map $\widetilde{p}_*\circ\sigma^*:\mathrm{H}^4(A,\mathbb{Z})\xrightarrow{\simeq}\mathrm{H}^4(\mathrm{Kum}^1(A),\mathbb{Z})$ is an isomorphism. So, the extended homomorphism $\widetilde{\alpha}:=\widetilde{p}_*\circ\sigma^*:\mathrm{H}^*(A,\mathbb{Z})\xrightarrow{\simeq}\mathrm{H}^*(\mathrm{Kum}^1(A),\mathbb{Z})$ is \textit{not} primitive.

    The derived global Torelli theorem for K3 surfaces implies that the group of derived autoequivalences $\mathrm{Aut}(\mathbf{D}^b(\mathrm{Kum}^1(A)))$ corresponds to $\mathrm{SO}_{\mathrm{Hdg}}(\widetilde{\mathrm{H}}(\mathrm{Kum}^1(A),\mathbb{Z}))$, the group of special orthogonal Hodge isometries of the Mukai lattice with respect to the Mukai pairing, by the induced maps on the cohomology. Here, the Mukai pairing is $$\langle(\alpha_0,\alpha_1,\alpha_2),(\beta_0,\beta_1,\beta_2)\rangle:=\alpha_0.\beta_2-\alpha_1.\beta_1+\alpha_2.\beta_0\in\mathbb{Z}$$ for $\alpha_i,\beta_i\in\mathrm{H}^{2i}(\mathrm{Kum}^1(A)),\mathbb{Z})$. The Mukai lattice $\widetilde{\mathrm{H}}(\mathrm{Kum}^1(A),\mathbb{Z})$ is ${\mathrm{H}}^*(\mathrm{Kum}^1(A),\mathbb{Z})$ with the Mukai pairing and the weight-two Hodge structure \begin{align*}
        \widetilde{\mathrm{H}}^{1,1}(\mathrm{Kum}^1(A),\mathbb{Z})&:=(\mathrm{H}^0\oplus\mathrm{H}^4)(\mathrm{Kum}^1(A),\mathbb{Z})\oplus(\mathrm{H}^{1,1}(\mathrm{Kum}^1(A))\cap\mathrm{H}^*(\mathrm{Kum}^1(A),\mathbb{Z})),\\
        \widetilde{\mathrm{H}}^{2,0}(\mathrm{Kum}^1(A),\mathbb{Z})&:=\mathrm{H}^{2,0}(\mathrm{Kum}^1(A))\cap\mathrm{H}^*(\mathrm{Kum}^1(A),\mathbb{Z}).
    \end{align*}
    Similarly, we may define the integral extended Mukai lattice $\widetilde{\mathrm{H}}(A,\mathbb{Z})$ for the abelian surface $A$.

    Obviously, using the induced maps of cohomology, the group of all the derived autoequivalences in $\mathrm{Aut}(\mathbf{D}^b(\mathrm{Kum}^1(A)))$ lifted from $\mathrm{Aut}(\mathbf{D}^b(A))$ corresponds to a subgroup $G_A$ of $\mathrm{SO}_{\mathrm{Hdg}}(\widetilde{\mathrm{H}}(\mathrm{Kum}^1(A),\mathbb{Z}))$, which is composed of special orthogonal Hodge isometries that map $\widetilde{p}_*\circ\sigma^*(\widetilde{\mathrm{H}}(A,\mathbb{Z}))\subset\widetilde{\mathrm{H}}(\mathrm{Kum}^1(A),\mathbb{Z})$ to itself.
\end{passage}

\begin{passage}\label{G_A for lifting 2}
    Actually, the subgroup $G_A$ of $\mathrm{SO}_{\mathrm{Hdg}}(\widetilde{\mathrm{H}}(\mathrm{Kum}^1(A),\mathbb{Z}))$ is composed of special orthogonal Hodge isometries that map $\mathbb{Z}^W\subset\widetilde{\mathrm{H}}(\mathrm{Kum}^1(A),\mathbb{Z})$ to itself.

    Indeed, every special orthogonal Hodge isometry in $\mathrm{SO}_{\mathrm{Hdg}}(\widetilde{\mathrm{H}}(\mathrm{Kum}^1(A),\mathbb{Z}))$ maps $\widetilde{\mathrm{H}}(\mathrm{Kum}^1(A),\mathbb{Z})$ to $\widetilde{\mathrm{H}}(\mathrm{Kum}^1(A),\mathbb{Z})$, with
    \begin{equation}\label{element in SOHdgHtilde(Kum1A,Z)}
        \begin{aligned}
            \mathrm{H}^{2,0}(\mathrm{Kum}^1(A))\cap\widetilde{\mathrm{H}}(\mathrm{Kum}^1(A),\mathbb{Z})&\to\mathrm{H}^{2,0}(\mathrm{Kum}^1(A))\cap\widetilde{\mathrm{H}}(\mathrm{Kum}^1(A),\mathbb{Z})\\
            \mathrm{H}^{0,2}(\mathrm{Kum}^1(A))\cap\widetilde{\mathrm{H}}(\mathrm{Kum}^1(A),\mathbb{Z})&\to\mathrm{H}^{0,2}(\mathrm{Kum}^1(A))\cap\widetilde{\mathrm{H}}(\mathrm{Kum}^1(A),\mathbb{Z})\\
            (\mathrm{H}^{0,0}\oplus \mathrm{H}^{1,1}\oplus \mathrm{H}^{2,2})(\mathrm{Kum}^1(A))\cap\widetilde{\mathrm{H}}(\mathrm{Kum}^1(A),\mathbb{Z})&\to(\mathrm{H}^{0,0}\oplus \mathrm{H}^{1,1}\oplus \mathrm{H}^{2,2})(\mathrm{Kum}^1(A))\cap\widetilde{\mathrm{H}}(\mathrm{Kum}^1(A),\mathbb{Z}).
        \end{aligned}
    \end{equation}
    Every special orthonogal Hodge isometry in the subgroup $G_A$ maps $\widetilde{p}_*\circ\sigma^*(\widetilde{\mathrm{H}}(A,\mathbb{Z}))$ to $\widetilde{p}_*\circ\sigma^*(\widetilde{\mathrm{H}}(A,\mathbb{Z}))$, with \begin{equation}\label{element in G_A}
        \begin{aligned}
            \widetilde{p}_*\circ\sigma^*(\mathrm{H}^{2,0}(A)\cap\widetilde{\mathrm{H}}(A,\mathbb{Z}))&\to\widetilde{p}_*\circ\sigma^*(\mathrm{H}^{2,0}(A)\cap\widetilde{\mathrm{H}}(A,\mathbb{Z}))\\
            \widetilde{p}_*\circ\sigma^*(\mathrm{H}^{0,2}(A)\cap\widetilde{\mathrm{H}}(A,\mathbb{Z}))&\to\widetilde{p}_*\circ\sigma^*(\mathrm{H}^{0,2}(A)\cap\widetilde{\mathrm{H}}(A,\mathbb{Z}))\\
            \widetilde{p}_*\circ\sigma^*((\mathrm{H}^{0,0}\oplus \mathrm{H}^{1,1}\oplus \mathrm{H}^{2,2})(A)\cap\widetilde{\mathrm{H}}(A,\mathbb{Z}))&\to\widetilde{p}_*\circ\sigma^*((\mathrm{H}^{0,0}\oplus \mathrm{H}^{1,1}\oplus \mathrm{H}^{2,2})(A)\cap\widetilde{\mathrm{H}}(A,\mathbb{Z})).
        \end{aligned}
    \end{equation}

    As $\alpha_{\mathbb{C}}(\mathrm{H}^{2,0}(A))=\mathrm{H}^{2,0}(\mathrm{Kum}^1(A))$, the first two lines of (\ref{element in SOHdgHtilde(Kum1A,Z)}) are equivalent to the first two lines of (\ref{element in G_A}). The last line of (\ref{element in G_A}) is equivalent to \begin{equation*}
    \begin{aligned}
        &((2\mathrm{H}^{0,0}\oplus \mathrm{H}^{2,2})(\mathrm{Kum}^1(A))\oplus(\mathrm{H}^{1,1}(\mathrm{Kum}^1(A))\cap(\mathbb{Z}^W)^\perp))\cap\widetilde{\mathrm{H}}(\mathrm{Kum}^1(A),\mathbb{Z})\\
        &\to((2\mathrm{H}^{0,0}\oplus \mathrm{H}^{2,2})(\mathrm{Kum}^1(A))\oplus(\mathrm{H}^{1,1}(\mathrm{Kum}^1(A))\cap(\mathbb{Z}^W)^\perp))\cap\widetilde{\mathrm{H}}(\mathrm{Kum}^1(A),\mathbb{Z}).
    \end{aligned}    
    \end{equation*}
    Since every element of $G_A$ satisfies the last line of (\ref{element in SOHdgHtilde(Kum1A,Z)}), the condition can be simplified to $$(2\mathrm{H}^{0,0}\oplus \mathrm{H}^{2,2})(\mathrm{Kum}^1(A))\oplus(\mathrm{H}^{1,1}(\mathrm{Kum}^1(A))\cap(\mathbb{Z}^W)^\perp)\to(2\mathrm{H}^{0,0}\oplus \mathrm{H}^{2,2})(\mathrm{Kum}^1(A))\oplus(\mathrm{H}^{1,1}(\mathrm{Kum}^1(A))\cap(\mathbb{Z}^W)^\perp).$$ 
    It is equivalent to $(\mathbb{Z}^W)^\perp\to(\mathbb{Z}^W)^\perp$, and moreover, $\mathbb{Z}^W\to\mathbb{Z}^W$, since every element of $G_A$ satisfies the first two lines of (\ref{element in SOHdgHtilde(Kum1A,Z)}).

    In conclusion, the subgroup $G_A$ of $\mathrm{SO}_{\mathrm{Hdg}}(\widetilde{\mathrm{H}}(\mathrm{Kum}^1(A),\mathbb{Z}))$ is composed of special orthogonal Hodge isometries that map $\mathbb{Z}^W\subset\widetilde{\mathrm{H}}(\mathrm{Kum}^1(A),\mathbb{Z})$ to itself.
\end{passage}

To compare the description of $\Phi\in\mathrm{Aut}(\mathbf{D}^b_{\mathfrak{S}_2}(A))\simeq\mathrm{Aut}(\mathbf{D}^b(\mathrm{Kum}^1(A)))$ being of the form $\Phi_{(f,\sigma)}$ in Remark \ref{Kummer K3 remark 1} with the subgroup $G_A$ of $\mathrm{SO}_{\mathrm{Hdg}}(\widetilde{\mathrm{H}}(\mathrm{Kum}^1(A),\mathbb{Z}))$, we need to translate the condition for $\gamma_A(\Phi)=g\in\mathrm{Sp}(A)$ in the level of rational even cohomology $\widetilde{\mathrm{H}}(A,\mathbb{Q})$ using the triality of $\mathrm{Spin}(8)$ as in subsection \ref{triality of Spin(8)}. Note that the triality proposition can only be used for the Kummer K3 surface case since $\dim(V_{N_A})=\dim(V_A)=8$.

\begin{passage}\label{Jg*J-1 general}
    Let $V_A:=\mathrm{H}^1(A, \mathbb{Z})\oplus\mathrm{H}^1(\widehat{A}, \mathbb{Z}), W_A:=\mathrm{H}^1(A, \mathbb{Z}), W_A':=\mathrm{H}^1(\widehat{A}, \mathbb{Z})$ and $p_{W_A},p_{W_A'}$ be two projections of $V_A$. Using the bilinear symmetric form $Q_A$ of $V_A$ as in subsection \ref{Relating the images of lambda_q(f,sigma) and f via Orlov's representation}, we have the natural isomorphism $W'_A=\mathrm{H}^1(\widehat{A}, \mathbb{Z})\cong\mathrm{H}^1(A, \mathbb{Z})^*$ that maps $\beta\in\mathrm{H}^1(\widehat{A}, \mathbb{Z})$ to $2Q_A(\beta,-)=(\beta,-)_{V_A}$. Using Remarks \ref{SP vs SO remark} and \ref{Kummer K3 remark 1}, a derived autoequivalence $\Phi\in\mathrm{Aut}(\mathbf{D}^b_{\mathfrak{S}_2}(A))\simeq\mathrm{Aut}(\mathbf{D}^b(\mathrm{Kum}^1(A)))$ is of the form $\Phi_{(f,\sigma)}$ as in Corollary \ref{split_NAA} if and only if the symplectic map $g=\gamma_{A}(\Phi)$ induces $g_*\in\mathrm{SO}^+_{\mathrm{Hdg}}(V_A)$ with $p_{W_A}\circ g_*\circ p_{W_A'}$ mapping $\frac{1}{2}W_A'$ to $2W_A$.

    Similarly for the rational vector space $V_{A,\mathbb{Q}}$, we have $V_{A,\mathbb{Q}}=W_{A,\mathbb{Q}}\oplus W'_{A,\mathbb{Q}}$. We denote by $S^+_{A,\mathbb{Q}}:=\mathrm{H}^\mathrm{even}(A,\mathbb{Q})=\widetilde{\mathrm{H}}(A,\mathbb{Q})$ a half-spin representation. The triality of $\mathrm{Spin}(8)$ in Proposition \ref{Ex 20.51 Fulton Harris} implies that the element $g_*\in\mathrm{SO}^+_{\mathrm{Hdg}}(V_{A,\mathbb{Q}})$ in the standard representation corresponds to an element $J_A\circ g_*\circ J_A^{-1}\in\mathrm{SO}^+(S^+_{A,\mathbb{Q}})$ in the half-spin representation. So, the condition $p_{W_A}\circ g_*\circ p_{W_A'}(\frac{1}{2}W_A')\subset2W_A$ corresponds to $J_A\circ p_{W_A}\circ g_*\circ p_{W_A'}\circ J_A^{-1}(\frac{1}{2}J_A(W_A'))\subset2J_A(W_A)$. 
    
    We still use $g_*:\widetilde{\mathrm{H}}(A,\mathbb{Q})\to\widetilde{\mathrm{H}}(A,\mathbb{Q})$ to denote the isomorphism in even cohomology induced by $\Phi$. Then, we can conclude that a derived autoequivalence $\Phi\in\mathrm{Aut}(\mathbf{D}^b_{\mathfrak{S}_2}(A))\simeq\mathrm{Aut}(\mathbf{D}^b(\mathrm{Kum}^1(A)))$ is of the form $\Phi_{(f,\sigma)}$ as in Corollary \ref{split_NAA} if and only if $\Phi$ induces an isomorphism $g_*:\widetilde{\mathrm{H}}(A,\mathbb{Q})\to\widetilde{\mathrm{H}}(A,\mathbb{Q})$ in even cohomology such that $p_{J_A(W'_{A,\mathbb{Q}})}\circ g_*\circ p_{J_A(W_{A,\mathbb{Q}})}$ maps $\frac{1}{2}J_A(W'_{A})$ to $2J_A(W_A)$, where $p_{J_A(W_{A,\mathbb{Q}})}$ and $p_{J_A(W'_{A,\mathbb{Q}})}$ are two projections of $\widetilde{\mathrm{H}}(A,\mathbb{Q})=J_A(W_{A,\mathbb{Q}})\oplus J_A(W'_{A,\mathbb{Q}})$.

    This kind of $g_*$ forms only a subgroup of $G_A$. That means that there may be other ways to lift some derived autoequivalences of abelian surfaces to those of the corresponding Kummer K3 surfaces.%
    \footnote{I suspect that the method in Corollary \ref{split_NAA} gives the unique way to lift some derived autoequivalences of abelian surfaces to those of the corresponding generalized Kummer varieties, which are not Kummer K3 surfaces.}%
    
\end{passage}

\begin{passage}\label{Jg*J-1 explicit}
    We may compute the subspaces $J_A(W_A)$ and $J_A(W'_{A})$ for the automorphism $J_A$ using the bilinear symmetric form $Q_A$ and the contents in passages \ref{section 20.1 even, Fulton Harris}, \ref{section 20.3 even, Fulton Harris} and Proposition \ref{Ex 20.50 Fulton Harris}.

    Using the bilinear symmetric form $Q_A$, we may choose elements \begin{align*}
        v_1&=dz_1+(d\bar{z_1}\wedge dz_2\wedge d\bar{z_2})\in V_{A}\\
        s_1&=1+ (dz_1\wedge d\bar{z_1}\wedge dz_2\wedge d\bar{z_2})\in S^+_{A}
    \end{align*} such that $\langle v_1, v_1\rangle_{V_A}=\langle s_1, s_1\rangle_{S^+_A}=1$ by the notation in passage \ref{section 20.3 even, Fulton Harris}. Indeed, the element $d\bar{z_1}\wedge dz_2\wedge d\bar{z_2}\in\mathrm{H}^3(A,\mathbb{Z})\cong\mathrm{H}^1(\widehat{A},\mathbb{Z})$ corresponds to the element $\ell_{z_1}\in\mathrm{H}^1(\widehat{A},\mathbb{Z})^*$ such that $\ell_{z_1}(dz_1)=1, \ell_{z_1}(d\overline{z_1})=\ell_{z_1}(dz_2)=\ell_{z_1}(d\overline{z_2})=0$ by the natural isomorphism $\mathrm{H}^1(\widehat{A},\mathbb{Z})\cong\mathrm{H}^1(A,\mathbb{Z})^*$.
    
    Now, we may use the elements $v_1\in V_{A}, s_1\in S^+_{A}$ to construct the automorphism $J_A$ in Proposition \ref{Ex 20.50 Fulton Harris}. 
    Precisely, for \begin{align*}
        w'&=a_1d\bar{z_1}\wedge dz_2\wedge d\bar{z_2}+a_2d{z_1}\wedge dz_2\wedge d\bar{z_2}+a_3d{z_1}\wedge d\bar{z_1}\wedge d\bar{z_2}+a_4d{z_1}\wedge d\bar{z_1}\wedge d{z_2}\in W'_A\\
        w&=b_1dz_1+b_2d\bar{z_1}+b_3dz_2+b_4d\bar{z_2}\in W_A
    \end{align*} with $a_i, b_i\in\mathbb{Z}, i=1,2,3,4$, we have \begin{align*}
        J_A(w')&=2Q_A(w',v_1)s_1-w'\cdot v_1\cdot s_1\\
        &=a_1(1+2dz_1\wedge d\bar{z_1}\wedge dz_2\wedge d\bar{z_2})-a_2(dz_2\wedge d\bar{z_2})-a_3(d\bar{z_1}\wedge d\bar{z_2})-a_4(d\bar{z_1}\wedge dz_2)\in S^+_{A}\\
        J_A(w)&=2Q_A(w,v_1)s_1-w\cdot v_1\cdot s_1\\
        &=b_1(2+dz_1\wedge d\bar{z_1}\wedge dz_2\wedge d\bar{z_2})-b_2(dz_1\wedge d\bar{z_1})-b_3(d{z_1}\wedge d{z_2})-b_4(d{z_1}\wedge d\bar{z_2})\in S^+_{A}.
    \end{align*}
    So the subspaces $J_A(W_A)$ and $J_A(W'_{A})$ for $\widetilde{\mathrm{H}}(A,\mathbb{Z})$ can be expressed as \begin{align*}
        J_A(W'_A)&=\langle1+2dz_1\wedge d\bar{z_1}\wedge dz_2\wedge d\bar{z_2},dz_2\wedge d\bar{z_2},d\bar{z_1}\wedge d\bar{z_2},d\bar{z_1}\wedge dz_2\rangle_{\mathbb{Z}}\\
        J_A(W_A)&=\langle2+dz_1\wedge d\bar{z_1}\wedge dz_2\wedge d\bar{z_2},dz_1\wedge d\bar{z_1},d{z_1}\wedge d{z_2},d{z_1}\wedge d\bar{z_2}\rangle_{\mathbb{Z}}
    \end{align*} using the generators with integer coefficients.
\end{passage}

\begin{remark}
    Moreover, using passages \ref{Jg*J-1 general} and \ref{Jg*J-1 explicit}, we have the following equivalence for $n\geq 2$ in general. 
    
    Let $g=\begin{pmatrix}
            g_1&g_2\\g_3&g_4
        \end{pmatrix}\in\mathrm{Sp}(A)$ as before, then
    \begin{align*}
        (g_2)_*(\frac{1}{n}\mathrm{H}^1(\widehat{A},\mathbb{Z}))\subset n\mathrm{H}^1(A,\mathbb{Z})
        \Leftrightarrow (g_2)_*(\frac{1}{n}J_A(W'_A))\subset J_A(W_A),
    \end{align*}
    where the subspaces $J_A(W_A)$ and $J_A(W'_{A})$ for $\widetilde{\mathrm{H}}(A,\mathbb{Z})$ are \begin{align*}
        J_A(W'_A)&=\langle1+2dz_1\wedge d\bar{z_1}\wedge dz_2\wedge d\bar{z_2},dz_2\wedge d\bar{z_2},d\bar{z_1}\wedge d\bar{z_2},d\bar{z_1}\wedge dz_2\rangle_{\mathbb{Z}}\\
        J_A(W_A)&=\langle 2+dz_1\wedge d\bar{z_1}\wedge dz_2\wedge d\bar{z_2},dz_1\wedge d\bar{z_1},d{z_1}\wedge d{z_2},d{z_1}\wedge d\bar{z_2}\rangle_{\mathbb{Z}}
    \end{align*} using the generators with integer coefficients.

\end{remark}

\subsection{Triality of $\mathrm{Spin}(8)$}\label{triality of Spin(8)}
In this subsection, we recall some notation and results about the triality of $\mathrm{Spin}(8)$ in \cite[section 20]{Fulton Harris:04} and \cite[Chapter IV]{Chevalley:54} used in passages \ref{Jg*J-1 general} and \ref{Jg*J-1 explicit}. 

First, we recap the Clifford algebra and half-spin representations.

\begin{passage}
    Given a real vector space $V$ and a bilinear symmetric form $Q$ on $V$, the \textit{Clifford algebra} $C(V,Q)$ can be defined to be the universal algebra with this property: if $E$ is any associative algebra with unit, and a linear mapping $j:V\to E$ is given such that $j(v)^2=Q(v,v)\cdot1$ for all $v\in V$, or equivalently \begin{equation}\label{universal prop of C(V,Q)}
    j(v_1)\cdot j(v_2)+j(v_2)\cdot j(v_1)=2Q(v_1,v_2)\cdot 1
    \end{equation}
    for all $v_1,v_2\in V$, then there should be a unique homomorphism of algebras from $C(V,Q)$ to $E$ extending $j$. The Clifford algebra can be constructed by taking the tensor algebra $T^*(V):=\underset{n\geq0}{\oplus}V^{\otimes n}$ and setting $C(V,Q)=T^*(V)/I(Q)$, where $I(Q)$ is the two-sided ideal generated by all elements of the form $v\otimes v-Q(v,v)\cdot1$. It is automatic that the constructed $C(V,Q)$ satisfies the universal property required.

    Since the ideal $I(Q)\subset T^*(V)$ is generated by elements of even degree, the Clifford algebra inherits a $\mathbb{Z}/2\mathbb{Z}$ grading: $C(V,Q)=C(V,Q)^{\mathrm{even}}\oplus C(V,Q)^{\mathrm{odd}}$, where $C(V,Q)^{\mathrm{even}}$ and $C(V,Q)^{\mathrm{odd}}$ are spanned by products of an even number or an odd number of elements in $V$ respectively. In particular, $C(V,Q)^{\mathrm{even}}$ is a subalgebra of $C(V,Q)$.
\end{passage}

\begin{passage}\label{section 20.1 even, Fulton Harris}
    When $\dim(V)=2n$ is even, we write $V=W\oplus W'$, where $W$ and $W'$ are $n$-dimensional isotropic spaces for $Q$. This decomposition determines an isomorphism of algebras $C(V,Q)\cong\mathrm{End}(\wedge^*W)$, where $\wedge^*W=\underset{i=0}{\overset{n}{\oplus}}\wedge^iW$ by \cite[Lemma 20.9]{Fulton Harris:04}. 
    
    Indeed, this isomorphism is deduced from the universal property of $C(V,Q)$ to the linear mapping $j=l\oplus l':V=W\oplus W'\to\mathrm{End}(\wedge^*W)$ that satisfies (\ref{universal prop of C(V,Q)}). Here, (\ref{universal prop of C(V,Q)}) can be verified by setting $l(w)=L_w$, the left mutiplication by $w$ on the exterior algebra $\wedge^*W$, and setting $l'(w')=D_\ell$, where $\ell\in W^*$ is defined by $\ell(w)=2Q(w,w'),\forall w\in W$ and $D_{\ell}$ is the derivation of $\wedge^*W$ such that $$D_{\ell}(w_1\wedge\dots\wedge w_r):=\sum(-1)^{i-1}\ell(w_i)(w_1\wedge\dots\hat{w_i}\wedge\dots\wedge w_r).$$

    We deduce to get an isomorphism $C(V,Q)^{\mathrm{even}}\cong\mathrm{End}(\wedge^{\mathrm{even}}W)\oplus\mathrm{End}(\wedge^{\mathrm{odd}}W)$. Using the embedding of Lie algebras, $\mathfrak{so}_{2n}(\mathbb{R})\hookrightarrow C(V,Q)^{\mathrm{even}}$, in \cite[Lemma 20.7]{Fulton Harris:04}, we now have an embedding of Lie algebras:
    $$\mathfrak{so}_{2n}(\mathbb{R})\subset C(V,Q)^{\mathrm{even}}\cong\mathfrak{gl}(\wedge^{\mathrm{even}}W)\oplus\mathfrak{gl}(\wedge^{\mathrm{odd}}W).$$
    Hence, we have two representations of $\mathfrak{so}_{2n}(\mathbb{R})$, which we denote by $S^+=\wedge^{\mathrm{even}}W$ and $S^-=\wedge^{\mathrm{odd}}W$. They are irreducible representations by \cite[Proposition 20.15]{Fulton Harris:04}. These two representations are called the \textit{half-spin representations} of $\mathfrak{so}_{2n}(\mathbb{R})$. The sum $S=S^+\oplus S^-=\wedge^*W$ is called the \textit{spin representation}.
\end{passage}

Now, we may state the triality of $\mathrm{Spin}(8)$. That is, for the $\dim V=8$ case. Actually, triality for $\mathfrak{so}(V)$ is only valid when $\dim V=8$. 

Indeed, simple Lie algebra determines the Dynkin diagram for the irreducible root system. When $\dim V=2n+1$ is odd with $n\geq 2$, the Dynkin diagram $B_n$ has a trivial automorphism group. When $\dim V=2n$ is even with $n\geq4$, the Dynkin diagram $D_n$ satisfies $\mathrm{Aut}(D_n)=\begin{cases}
    \mathfrak{S}_3, n=4\\
    \mathbb{Z}/2\mathbb{Z}, n\geq5.
\end{cases}$ 

\begin{passage}\label{section 20.3 even, Fulton Harris}
    Let the notation be as above. The linear map $j:V\to\mathrm{End}(\wedge^*W)$ in passage \ref{section 20.1 even, Fulton Harris}, gives rise to bilinear maps $V\times S^+\xrightarrow{\cdot}S^-, V\times S^-\xrightarrow{\cdot}S^+$. Let $\langle-,-\rangle_V$ denote the symmetric form corresponding to the quadratic form in $V$ (That is, $Q(-,-)=\langle-,-\rangle_V$.) and similarly for $S^+$ and $S^-$. Define a product $S^+\times S^-\xrightarrow{\cdot}V, (s,t)\mapsto s\cdot t$ by requiring $\langle v,s\cdot t\rangle_V=\langle v\cdot s, t\rangle_{S^-}$. 

    The above products determine a commutative but non-associative product on the direct sum $V\oplus S^+\oplus S^-$. The operation $(v,s,t)\mapsto \langle v\cdot s,t\rangle_{S^-}:=F(v+s+t)$ determines a cubic form on $V\oplus S^+\oplus S^-$. It determines a symmetric trilinear form $\Theta$ on $V\oplus S^+\oplus S^-$ by polarization. That is, $$\Theta(v,s,t):=F(v)+F(s)+F(t)-F(v+s)-F(v+t)-F(s+t)+F(v+s+t).$$
\end{passage}

The following two propositions give an algebraic version of triality.

\begin{proposition}[{\cite[Exercise 20.50]{Fulton Harris:04}, see also \cite[section 4.3]{Chevalley:54}}]\label{Ex 20.50 Fulton Harris}
    Choose elements $v_1\in V$ and $s_1\in S^+$ with $\langle v_1,v_1\rangle_V=\langle s_1,s_1\rangle_{S^+}=1$. Construct a map $J$ to be the composition of two involutions $\mu$ and $\nu$, which are determined by the following:
    \begin{enumerate}
        \item $\mu$ interchanges $S^+$ and $S^-$, and maps $V$ to itself, with $\mu(s)=v_1\cdot s$ for $s\in S^+$ and $\mu(v)=2\langle v,v_1\rangle_Vv_1-v$ for $v\in V$.
        \item $\nu$ interchanges $V$ and $S^-$, and maps $S^+$ to itself, with $\nu(v)=v\cdot s_1$ for $v\in V$ and $\mu(s)=2\langle s,s_1\rangle_{S^+}s_1-s$ for $s\in S^+$.
    \end{enumerate}
    Then, the map $J$ is an automorphism of $V\oplus S^+\oplus S^-$ of order $3$, such that $$J(V)=S^+, J(S^+)=S^-, J(S^-)=V,$$ preserving their quadratic forms, and compatible with the cubic form $F$.
\end{proposition}

Note that the map $J$ depends on the choice of $v_1,s_1$.

\begin{proposition}[{\cite[Exercise 20.51]{Fulton Harris:04}, see also \cite[sections 4.2-4.5]{Chevalley:54}} and \cite{Jacobson:79}]\label{Ex 20.51 Fulton Harris}
    In this algebraic form, the triality of $\mathrm{Spin}(V)\cong\mathrm{Spin}(8)$ is to say that there is an automorphism $j\in\mathrm{Aut}(\mathrm{Spin}(V))$ of order $3$, which is compatible with $J$. It means that for all $x\in\mathrm{Spin}(V)$, we have the following commutative diagram. \begin{align*}
        \xymatrix{
        V\ar[rr]^{J}\ar[d]^{\rho(x)}&&S^+\ar[rr]^{J}\ar[d]^{\rho^+(j(x))}&&S^-\ar[rr]^{J}\ar[d]^{\rho^-(j^2(x))}&&V\ar[d]^{\rho(x)}\\
        V\ar[rr]^{J}&&S^+\ar[rr]^{J}&&S^-\ar[rr]^{J}&&V
        }
    \end{align*}

    If $j':\mathfrak{so}(V)\to\mathfrak{so}(V)$ is the map induced by $j$, then we have the ``local triality equation" $$\Theta(Xv,s,t)+\Theta(v,Ys, t)+\Theta(v,s,Zt)=0$$ for $X\in\mathfrak{so}(V)\cong\mathfrak{so}(8), Y=j'(X), Z=j'(Y)$. It can be translated to the fact that $j$ is compatible with the trilinear form $\Theta$.
\end{proposition}

\bibliographystyle{plain}

\end{document}